\documentclass[final]{amsart} 
\usepackage{graphicx}%
\usepackage{multirow}%
\usepackage{amsmath,amssymb,amsfonts}%
\usepackage{amsthm}%
\usepackage{mathrsfs}%
\usepackage[title]{appendix}%
\usepackage{xcolor}%
\usepackage{textcomp}%
\usepackage{tikz}
\usepackage[hidelinks]{hyperref}
\usepackage{manyfoot}%
\usepackage{booktabs}%
\usepackage{algorithm}%
\usepackage{algorithmicx}%
\usepackage{algpseudocode}%
\usepackage{listings}%
\usepackage{thmtools}
\usepackage{graphicx}
\usepackage{stackengine,scalerel}
\usepackage{mathtools}
\usepackage{subcaption}
\usepackage{tkz-euclide}
\usepackage{microtype}
\usepgflibrary {shadings}
\usepackage{import}
\usetikzlibrary{fadings,patterns,shadows.blur,shapes, decorations.markings}
\usetikzlibrary{arrows.meta}
\usetikzlibrary{patterns.meta}
\usepackage[
backend=biber,
style=numeric,
sorting=nyt,
isbn=false,
url=false,
maxalphanames=4,
maxbibnames=99
]{biblatex}
\declaretheorem[numberwithin=section]{theorem}

\newcommand\ttimes{\mathbin{\ThisStyle{\ensurestackMath{%
  \stackengine{-1\LMpt}{\SavedStyle\times}
  {\SavedStyle_{\hstretch{.9}{\mkern1mu\sim}}}{O}{c}{F}{T}{S}}}}}
\newcommand\pmttimes{\mathbin{\ThisStyle{\ensurestackMath{%
  \stackengine{-1\LMpt}{\SavedStyle\times}
  {\SavedStyle_{\hstretch{.9}{\mkern1mu(\sim)}}}{O}{c}{F}{T}{S}}}}}
\newcounter{claimcounter}
\newenvironment{claim}[1]{
\refstepcounter{claimcounter}
\par\noindent\textbf{Claim \theclaimcounter:}\space#1}{}
\newenvironment{claimproof}[1]{\par\noindent\textbf{Proof:}\space#1}{\hfill $\blacksquare$}

\DeclarePairedDelimiter{\floor}{\lfloor}{\rfloor}

 \newcommand{\ZZ}{\mathbb{Z}}
 
\newcommand{\RR}{\mathbb{R}}

\newcommand{\sub}{\subseteq}

\renewcommand{\P}{\textbf{P}}
\newcommand{\NP}{\textbf{NP}}
\newcommand{\coNP}{\textbf{coNP}}
\newcommand{\NPcoNP}{$\text{\NP{}}\cap\text{co-\textbf{NP}}$}
\def\rdots{\rotatebox[origin=l]{29}{$\scriptscriptstyle\ldots\mathstrut$}}
\newtheorem{lemma}[theorem]{Lemma}
\newtheorem{prop}[theorem]{Proposition}
\newtheorem{corr}[theorem]{Corollary}
\theoremstyle{definition}
\newtheorem{defn}[theorem]{Definition}
\newtheorem{remark}[theorem]{Remark}
\newtheorem{example}[theorem]{Example}
\newtheorem{notation}[theorem]{Notation}
\newtheorem{convention}[theorem]{Convention}
\newtheorem{procedure}[theorem]{Procedure}

  \newcommand{\norm}[1]{\left\lVert#1\right\rVert}
     \newcommand{\del}{\partial}
\def\centerarc[#1](#2)(#3:#4:#5)
    { \draw[#1] ($(#2)+({#5*cos(#3)},{#5*sin(#3)})$) arc (#3:#4:#5); }

\newcommand{\ths}{split handle structure}
\newcommand{\pns}{duplicate}
\title[Recognition of Seifert fibered spaces]{Recognition of Seifert fibered spaces with boundary is in NP}

\author{Adele Jackson}
\email{\href{mailto:adele.jackson@maths.ox.ac.uk}{adele.jackson@maths.ox.ac.uk}}
\date{June 28, 2023}

\def\subjclassname{\textup{2020} Mathematics Subject Classification}
\expandafter\let\csname subjclassname@1991\endcsname=\subjclassname
\expandafter\let\csname subjclassname@2000\endcsname=\subjclassname

\begin{document}

\begin{abstract}
We show that the decision problem of recognising whether a triangulated 3-manifold admits a Seifert fibered structure with non-empty boundary is in NP.
We also show that the problem of deciding whether a given triangulated Seifert fibered space with non-empty boundary admits certain Seifert data is in \NPcoNP{}.
We do this by proving that in any triangulation of a Seifert fibered space with boundary there is both a fundamental horizontal surface of small degree and a complete collection of normal vertical annuli whose total weight is bounded by an exponential in the square of the triangulation size.
\end{abstract}

\maketitle

\tableofcontents

\section{Introduction}

One basic decision problem in low-dimensional topology is \textsc{3-manifold homeomorphism}: given two 3-manifolds, decide whether they are homeomorphic.
As a consequence of Perelman's proof of Thurston's geometrisation conjecture, this problem is decidable (several proofs have been given of this; see~\cite{ScottShort} for an overview).
In contrast, the $n$-manifold homeomorphism problem for fixed $n\geq 4$ is undecidable~\cite{Markov}.

The complexity of this problem, however, is not very well understood.
Kuperberg showed that there is an algorithm for \textsc{3-manifold homeomorphism} which has running time bounded by a bounded tower of exponentials -- that is, that it is at most of the order of $$ 2^{2^{\rdots^t}} $$ for some fixed height, where the height is not known~\cite{Kuperberg}.
It is also known, by using the JSJ graph, that the problem is at least as hard as finite graph isomorphism~\cite{LackenbyConditionallyHard}.

In practice, however, we are quite good at deciding \textsc{3-manifold homeomorphism}.
Matveev and Tarkaev produced a census of all 103,041 (orientable, closed, irreducible) 3-manifolds up to 13 tetrahedra in 2020~\cite{MatveevTarkaev}.
In the same year, Ben Burton tabulated the more than 300 million prime knots of up to 19 crossings, which involves showing that an associated list of 300 million knot complements does not contain duplicates~\cite{BurtonKnotTable19Crossings}.
Both groups work by enumerating all possible manifolds in the class they are considering, then for each pair, either showing they are homeomorphic using ad hoc methods or distinguishing them using invariants such as homology groups, Turaev-Viro invariants, and the number of subgroups of the fundamental group of a given index (see~\cite{BurtonClosedCensus11} and~\cite{MatveevTarkaev} for further discussion on the challenges of this approach).
If the 3-manifold homeomorphism problem took anything like the theoretical running time in practice, this work would have been impossible.

To better understand the difficulty of this problem, we ask: is \textsc{3-manifold homeomorphism} in \NP{}?
(The class \NP{} is a complexity class of decision problems, defined in Definition~\ref{defn:NP}).
To prove this, we might first hope to show that recognising hyperbolic and Seifert fibered manifolds is in \NP, and then attack the general case using geometrisation.

This paper concerns itself with the Seifert fibered case when the boundary is non-empty.
Seifert fibered spaces are 3-manifolds that are circle bundles over orbifolds.
We consider two algorithmic problems.
First, the \textsc{Seifert fibered space with boundary recognition} decision problem: given a triangulation of some 3-manifold, decide whether it admits a Seifert fibered structure with non-empty boundary.

\begin{restatable}{theorem}{SFSrecognitionNP}
The problem \textsc{Seifert fibered space with boundary recognition} is in \NP{}.
\label{thm:SFSrecognitionNP}
\end{restatable}

Second, the \textsc{naming Seifert fibered space with boundary} problem: given a set of Seifert data and a triangulation of a Seifert fibered 3-manifold with non-empty boundary, decide if the manifold is homeomorphic to the Seifert fibered space with that data.

\begin{restatable}{theorem}{SFSnamingNP}
The problem \textsc{naming Seifert fibered space with boundary} is in \NPcoNP{}.
\label{thm:SFSnamingNP}
\end{restatable}

The recognition problem for several other classes of 3-manifold is also known to be in \NP{}, and we will directly use some of these results in this paper.
The earliest work was in the setting of knot complements, where Hass, Lagarias and Pippenger used normal surface theory to prove that unknot recognition is in \NP{}~\cite{HassLagariasPippenger}.
Lackenby later showed that recognising whether a knot is non-trivial is also in \NP{}, or equivalently that unknot recognition is in \coNP{}~\cite{LackenbyCertificationKnottedness}.

Among the class of all orientable triangulated 3-manifolds, the first problem shown to be in \NP{} was 3-sphere recognition (by independent work of Ivanov and Schleimer~\cite{Schleimer3Sphere, Ivanov}), and Zentner proved that 3-sphere recognition is in \coNP{} assuming the generalised Riemann hypothesis holds~\cite{Zentner}.
Ivanov also proved that recognising simple manifolds such as $B^3$, $S^1\times S^2$, $\RR P^3$ and $S^1\times D^2$ is in \NP{}~\cite{Ivanov}, which implies that recognising the $S^2\times\RR$ geometry is in \NP{}.
Lackenby and Haraway-Hoffman between them showed that recognition of $I$-bundles over surfaces is in \NP{}~\cite{LackenbyCertificationKnottedness, HarawayHoffman}.
To give some further results on geometric classes, Lackenby and Schleimer showed that recognition of elliptic manifolds (the geometry of $S^3$) is in \NP{}~\cite{LacSchleimerElliptic}, and via some straightforward homology computations in the author's thesis, the recognition problem is in \NP{} for closed Euclidean and Nil manifolds~\cite[\S4.2]{JacksonThesis}; the Euclidean and Nil cases with boundary are covered by the results in this paper.
Haraway and Hoffman additionally proved that, among orientable irreducible manifolds with non-empty boundary, hyperbolic geometry recognition is in \coNP{} assuming the generalised Riemann hypothesis.

\begin{remark}
One application of Theorems~\ref{thm:SFSrecognitionNP} and~\ref{thm:SFSnamingNP} is to torus knot recognition.
Baldwin and Sivek have previously shown that, given a knot complement, deciding if it is the complement of a torus knot is in \NP{}, and that the problem is in \coNP{} assuming the generalised Riemann hypothesis~\cite{BaldwinSivek}.
As a consequence of Theorems~\ref{thm:SFSrecognitionNP} and~\ref{thm:SFSnamingNP}, we can strengthen the first part of this to the 3-manifold setting -- that is, prove that torus knot recognition among all triangulated 3-manifolds is in \NP{}.
The complement of the torus knot $T(p, q)$ is the Seifert fibered space $[D^2, -s/q, -r/p]$ where $r/s$ is a fraction such that $ps-qr = 1$.
Given an arbitrary triangulated 3-manifold, we can certify that it is $T(p, q)$ by certifying that it is Seifert fibered and has this Seifert data.
\end{remark}

After giving some background on computational complexity, normal surface theory and Seifert fibered spaces in Section~\ref{section:background}, we prove in Section~\ref{section:fundamentalhorizontalsurface} that there is a fundamental horizontal surface of minimal degree in all Seifert fibered spaces with non-empty boundary aside from a few exceptional cases.

We then need to establish the theory of \emph{\ths{}s}, which are defined in Section~\ref{section:ths} and are essential for the work in Section~\ref{section:fundamentalannulicollection}.
They arise from cutting handle structures along normal surfaces, then studying normal surfaces in the resulting manifold that are disjoint from the cut-open boundary.
We will use them to bound the size of a maximal collection of ``relatively'' fundamental surfaces.
In standard normal surface theory, we can bound the size of a single minimal essential surface by an exponential in the size of the triangulation $\norm{\mathcal{T}}$.
We wish to exhaust our manifold with such surfaces.
Na\"ively applying normal surface theory, the bound we would obtain would be a tower of exponentials $$c^{c^{\rdots^{\norm{\mathcal{T}}}}}$$ whose height would be linear in $\norm{\mathcal{T}}$.
Using \ths{}s, we can instead bound the size of the collection by $c^{\norm{\mathcal{T}}^2}$ (see Corollary~\ref{corr:stackedfundbound} for a precise statement).
This result is the technical heart of the paper.
This work is based on ideas used by King~\cite[\S3.2-3]{KingPolytopality} and Lackenby~\cite[\S12.2]{LackenbyCertificationKnottedness}, among others, which have not previously been rigorously generalised.
We delay some of the technical proofs in this theory to Appendix~\ref{appendix:normalsurfacesinths}, as they largely follow ideas from standard normal surface theory as described by Matveev~\cite[\S4]{Matveev}.

After developing the theory of \ths{}s, we apply it in Section~\ref{section:fundamentalannulicollection} to show that there is a collection of normal annuli $M$, of total weight at most $c^{\norm{\mathcal{T}}^2}$, that cut $M$ into a collection of solid tori.

The bulk of the work in proving Theorems~\ref{thm:SFSrecognitionNP} and~\ref{thm:SFSnamingNP} is to establish that recognition of circle bundles over surfaces with boundary is in \NP{}, which we do in Section~\ref{section:surfacebundlerecognition} using the results from Sections~\ref{section:fundamentalhorizontalsurface} and~\ref{section:fundamentalannulicollection}.
This is as, by previous work of the author~\cite{JacksonTriangulationComplexity}, the general case can be reduced to this problem: all singular fibers (circle fibers over the cone points of the base orbifold) other than those of multiplicity two can be made simplicial in the 82\textsuperscript{nd} barycentric subdivision of any triangulation $\mathcal{T}$ of $M$, so we can drill them out and record the slope of a meridian of each singular fiber.
In the first part of Section~\ref{section:SFSrecognition} we extend from circle bundles over surfaces to the case when $M$ has singular fibers of multiplicity two, and then can use the barycentric subdivision approach to prove the general case.

The author would like to thank Saul Schleimer for pointing out the connection with torus knot recognition.
The author would also like to thank the referee for their thorough and helpful comments.

\section{Background and conventions}
\label{section:background}
All 3-manifolds in this paper are compact and orientable.

\subsection{Computational complexity}
We give a quick introduction to some computational complexity classes; for a thorough introduction to the topic, see~\cite{AroraBarak}.

A \emph{problem} is a function from the set of finite binary strings to itself; from its \emph{input} to its \emph{output}.
The \emph{size} of an input string $T$, $|T|$, is the number of bits in the string.
A \emph{decision problem} is a problem whose range is the set $\{\text{``yes''}, \text{``no''}\}$.
An \emph{algorithm} that \emph{solves} a problem $P$ is a Turing machine that, given an input to the problem, computes the corresponding output.
We say a problem is \emph{decidable} if there exists an algorithm to solve it.
There exist undecidable decision problems, such as the \emph{halting problem} (famously proven by Turing) and, given a group presentation, deciding whether it is a presentation of the trivial group (by work of Adian-Rabin).

A \emph{complexity class} is (roughly speaking) a set of problems with solutions that satisfy certain restrictions.
We will be interested in three complexity classes: \P{}, \NP{} and \coNP{}.
An algorithm runs in \emph{polynomial time} if there is a polynomial $p$ such that if $T$ is an input, the running time of the algorithm is at most $p(|T|)$.

\begin{defn}
\label{defn:P}
A decision problem lies in \P{} if it has a polynomial time solution.
\end{defn}

We might think of this as the class of problems that are ``easy to evaluate''.
One example (by computing homology) is 2-manifold homeomorphism.

Informally, a decision problem $D$ lies in \NP{} if, given an input $T$ such that the output $D(T)$ is ``yes'', there is a proof of this output (a \emph{certificate}) that can be verified in polynomial time in $|T|$.
We might think of this as the class of problems that are ``easy to verify''.
Note that this definition is asymmetric -- while, if a problem is in \NP{}, we can certify a ``yes'', we may or may not be able to quickly certify a ``no''.

\begin{defn}
\label{defn:NP}
    A decision problem $D$ lies in \NP{} if there exists a polynomial time Turing machine $V$ and a polynomial $p$ such that for each input $T$, $D(T)$ is ``yes'' if and only if there exists some $c$, whose size is bounded by $p(|T|)$, such that $V(T, c) = \text{``yes''}$.
    We say that $c$ is a \emph{certificate} for $T$.

    A decision problem $D$ lies in \coNP{} if its complement problem, $\bar{D}$, which is defined by $\bar{D}(T) = \neg D(T)$, lies in \NP{}.
\end{defn}

\begin{remark}
    If a problem $D$ lies in $\P{}$, then $D$ also lies in \NP{} and \coNP{}, as the polynomial time solution to $D$ is sufficient to certify both a ``yes'' and a ``no'' output.
\end{remark}

\subsection{Triangulations and normal surface theory}
\label{section:normalsurface}

We will use the definition of \emph{triangulation} that is conventional in low-dimensional topology (and sometimes called a pseudo-triangulation): a triangulation of a 3-manifold $M$ is a collection of tetrahedra, with affine gluing maps between their faces, such that if we execute these gluing maps, the interior of the result is homeomorphic to $M$.

A \emph{normal curve} in a triangulated surface is a curve whose intersection with each triangle is a collection of arcs that run between {distinct} edges of the triangle.
Normal surfaces, which are a higher-dimensional analogue of normal curves, were developed by Haken in the 1960s to give a combinatorial representation of interesting surfaces in a triangulated 3-manifold~\cite{HakenNormalSurfaces}.
For a rigorous exposition of this theory, see \S3 and \S4 of~\cite{Matveev}.

Let $\mathcal{T}$ be a triangulation of a 3-manifold $M$.
A surface $F$ in $M$ is \emph{normal} with respect to $\mathcal{T}$ if it intersects each tetrahedron in a collection of discs, each of whose boundary curves intersects each edge of the tetrahedron at most once.
We call these \emph{elementary discs}.
Each of these discs must be a triangle or quadrilateral.
There are seven possible disc types in each tetrahedron, three of which are shown in Figure~\ref{fig:normal}.

\begin{figure}[th]
  \centering
  \resizebox{0.4\textwidth}{!}{\includegraphics{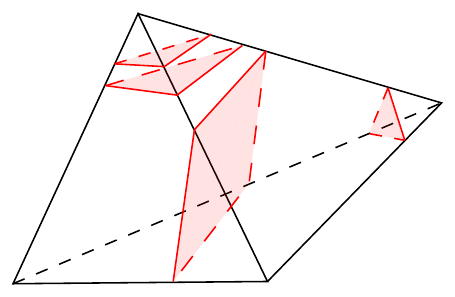}}
  \caption{An example of the intersection of a normal surface with a tetrahedron.}
  \label{fig:normal}
\end{figure}

\begin{prop}[Proposition 3.3.24~\cite{Matveev}]
\label{prop:normalrepsofessential}
    Let $S$ be an incompressible, $\del$-incompressible properly-embedded surface in an irreducible and $\del$-irreducible 3-manifold with a triangulation $\mathcal{T}$, such that $S$ is not a 2-sphere or a disc.
    There is a normal surface that is isotopic to $S$.
\end{prop}

\begin{defn}
    \label{defn:normalsurfaceweight}
    The \emph{size} of a normal surface $F$, $s(F)$, is the number of elementary discs in it.
    The \emph{edge weight} of a normal surface $F$, $w(F)$, is $|F\cap \mathcal{T}^1|$.
    A normal surface is \emph{minimal} if it is of minimal edge weight in its isotopy class in the 3-manifold.
\end{defn}

There is a natural identification of a normal surface $F$ with a vector $v_F$ in $\ZZ_{\geq 0}^{7\norm{\mathcal{T}}}$ by, for each tetrahedron $T$, writing down the count of each of the seven types of elementary disc in $F\cap T$.
If $F$ and $G$ are normal surfaces that are not normally isotopic (that is, isotopic fixing the 1-skeleton of $\mathcal{T}$), then $v_F$ and $v_G$ will be distinct.
We can add two such vectors $v_F$ and $v_G$ to obtain a vector $v_{F+G}$, which itself may or may not correspond to a normal surface that (when it exists) we call $F+G$.
Normal surface addition is additive on size, edge weight, Euler characteristic, and $\ZZ_2$-homology.

\begin{defn}
A normal surface $F$ is \emph{fundamental} if there is no way of writing it as $F = G_1 + G_2$ where $G_1$ and $G_2$ are non-empty normal surfaces.
\label{defn:fundamental}
\end{defn}

If $F$ is fundamental, we can use linear programming techniques to bound its size $s(F)$.

\begin{lemma}[Lemma 6.1~\cite{HassLagariasPippenger}]
   Let $F$ be a fundamental normal surface in a 3-manifold triangulation $\mathcal{T}$.
   The size of $F$ is at most $\norm{\mathcal{T}}^2 2^{7\norm{\mathcal{T}}+2}$.
   \label{lemma:triangulationfundamentalsurfacebound}
\end{lemma}

We will induct on weight to show that there are fundamental representatives of certain desirable surfaces (for example, in Proposition~\ref{prop:fundamentalhorizontalsurface}).
To do this, we need some results restricting summands of minimal essential normal surfaces.

\begin{theorem}[Theorem 6.5~\cite{JacoTollefson}]
\label{thm:jacotollefson}
    Let $F$ be an orientable, incompressible and $\del$-incompressible connected minimal normal surface in an orientable, irreducible and $\del$-irreducible manifold $M$ with a triangulation $\mathcal{T}$.
    Suppose that $nF = G_1 + G_2$ for some $n$.
    Then $G_1$ and $G_2$ are incompressible and $\del$-incompressible, and neither has any components of positive Euler characteristic.
\end{theorem}
Jaco and Tollefson's version of this result does not mention the sphere or $\RR P^2$ exclusion.
However, in Lemma 6.6 of~\cite{JacoTollefson} they show that under the same assumptions as in Theorem~\ref{thm:jacotollefson}, $G_1 \cup G_2$ contains no disc patches, so neither $G_1$ nor $G_2$ can be a sphere or $\RR P^2$.

In Matveev's book, he gives a variant of this result that does not require $F$ to be orientable.
\begin{theorem}[Theorem 4.1.36~\cite{Matveev}]
\label{thm:matveev}
   Let $F$ be an incompressible, $\del$-incompressible, minimal connected normal surface $F$ in an orientable, irreducible, $\del$-irreducible manifold $M$ with triangulation $\mathcal{T}$, such that $F = G_1 + G_2$.
   Then $G_1$ and $G_2$ are incompressible and $\del$-incompressible, and have no components of positive Euler characteristic. 
\end{theorem}

In Section~\ref{section:ths} we will define a variant of normal surface theory in the setting of \emph{\ths{}s}.
Much of the structure there will be parallel to standard normal surface theory.
For example, the analogous statements there to Proposition~\ref{prop:normalrepsofessential}, Lemma~\ref{lemma:triangulationfundamentalsurfacebound} and Theorem~\ref{thm:matveev} are respectively Proposition~\ref{prop:cannormaliseincompressiblesurfaces}, Lemma~\ref{lemma:fundamentalsurfacebound} and Proposition~\ref{prop:incompressiblenormalsurface}.

\subsection{Seifert fibered spaces}
\label{section:sfs}

A Seifert fibered space is an orientable 3-manifold that admits a fibration by circles whose base space is an orbifold (which necessarily has isolated cone points), or, equivalently (by work of Epstein), admits a foliation by circles.
Such a manifold can have boundary, in which case its boundary will be a union of fibres and hence will be a collection of tori.
Six of the eight Thurston geometries are Seifert fibered.

We can describe a Seifert fibered manifold $M$ by its \emph{Seifert data} $[\Sigma, q_1/p_1,\ldots,q_n/p_n]$ where $\Sigma$ is the underlying surface of the orbifold, the $n$ cone points of the orbifold have angle $2\pi/p_i$ (for $p_i\in\ZZ_{\geq 2}$), and the integer $q_i$ (for each $i$) determines the local fibration over each cone point.
When the manifold is closed, we also need to give its Euler number, $e\in\ZZ$, but as we will deal only with the non-empty boundary case, this will not be a factor in the work in this paper.
The \emph{singular fibers} of $M$ are the circle fibers over the cone points of the orbifold.
For a more thorough description of Seifert fibered spaces and their properties, see Ch.\ 10 of~\cite{Martelli}.

When do two sets of Seifert data correspond to homeomorphic 3-manifolds?
First, we give the standard criterion for when two sets of Seifert data correspond to the same Seifert fibration.
A proof is given in Proposition 10.3.13 of~\cite{Martelli}.

\begin{prop}
\label{prop:seifertdataisomorphism}
The Seifert fibrations associated to two sets of Seifert data $[\Sigma_1, q_1/p_1,\ldots, q_m/p_m]$ and $[\Sigma_2, r_1/s_1,\ldots,r_n/s_n]$ are isomorphic (preserving orientation) when $\Sigma_i$ has non-empty boundary if and only if $\Sigma_1$ and $\Sigma_2$ are homeomorphic, and after discarding all fractions $q_i/p_i$ and $r_i/s_i$ with $p_i$ or $s_i$ equal to 1, we have that $m=n$ and (up to reordering) $p_i = s_i$ and $q_i \equiv r_i\pmod{p_i}$. 
If the $\Sigma_i$ are closed, we additionally need to check that the Euler numbers $\sum_{i=1}^m q_i/p_i$ and $\sum_{j=1}^m r_i/s_i$ are equal, where we do not discard fractions with denominator one.
\end{prop}

Second, there are a few 3-manifolds that have more than one Seifert fibration, as shown by Waldhausen~\cite[Theorem 10.1]{Waldhausen2}.
For a proof in English see Theorem 10.4.19 of~\cite{Martelli}.

\begin{theorem}
\label{thm:nonisomorphicSFS}
    Seifert fibered spaces admit unique Seifert fibrations, aside from:
   \begin{enumerate}
        \item the solid torus fibres as $[D^2]$ and $[D^2, p/q]$;
       \item $[D^2, 1/2, -1/2]$ is also the circle bundle over the M\"obius band;
       \item $[S^2, 1/2, -1/2, q/p] \cong [\RR P^2, p/q]$;
       \item $[S^2, 1/2, 1/2, -1/2, -1/2] \cong [K]$, the twisted product of the Klein bottle with a circle;
        \item lens spaces (including $S^1\times S^2$) fibre in many ways.
   \end{enumerate}
\end{theorem}

We now give proofs of a few lemmas that we will use later in the paper.

\begin{lemma}
\label{lemma:SFShomeopolytime}
    Given two sets of Seifert data, where we are guaranteed that the associated manifolds $M_1$ and $M_2$ are not lens spaces, there is a polynomial time algorithm to decide if $M_1$ and $M_2$ are homeomorphic.
\end{lemma}

\begin{proof}
The algorithm is as follows.
Replacing two singular fibres $q_1/p_1$ and $q_2/p_2$ with $(q_1+p_1)/p_1$ and $(q_2-p_2)/p_2$ or, if $M$ has non-empty boundary, replacing $q/p$ with $(q+p)/p$, does not change the Seifert fibration up to isomorphism.
We can also remove trivial singular fibres (those with coefficient $0/1$).
Use these operations and their inverses to ensure that the multiplicity $p$ is at least 2 for all singular fibres in $M_1$ and $M_2$.
We can use the criteria in Proposition~\ref{prop:seifertdataisomorphism} to check (in polynomial time) if two sets of Seifert data with multiplicities at least 2 give orientation-preservingly isomorphic Seifert fibrations.
To see if they are isomorphic disregarding orientation, we also compare the Seifert data for $M_1$ with that for $M_2$ with the signs of all of the singular fibre fractions reversed.
If $M_1$ and $M_2$ have isomorphic Seifert fibrations, return ``yes''.
Otherwise, we need to check if we are in one of the cases in Theorem~\ref{thm:nonisomorphicSFS}.

If one Seifert fibration is $[D^2]$ or $[D^2, p/q]$, it is enough to check that the other is too.
Similarly, if one manifold is a circle bundle over a M\"obius band (case 2 of Theorem~\ref{thm:nonisomorphicSFS}), check if the other is isomorphic to $[D^2, 1/2, -1/2]$.
Cases 3 and 4 of Theorem~\ref{thm:nonisomorphicSFS} are analogous.

If none of these steps have concluded that the two manifolds are homeomorphic, then, as they are not lens spaces, they are not.
\end{proof}

Orientable incompressible and $\del$-incompressible surfaces in Seifert fibered spaces have been classified: they are either \emph{horizontal} (transverse to the fibration) or \emph{vertical} (a union of regular fibers) \cite[Thm.\ 2.8]{Waldhausen1}.
We will need to also classify the nonorientable ones; fortunately, while lesser known, these are also quite well-understood.
We first need to describe incompressible surfaces in solid tori.

\begin{prop}[\cite{Przytycki,Rannard}]
\label{prop:incompressibleinsolidtori}
The incompressible, non-$\del$-parallel, non-$S^2$ surfaces up to isotopy in a solid torus are classified by their intersections with the boundary.
The intersection of one of these surfaces with the boundary is a single curve of slope $\frac{p}{q}$ where $q$ is even (that is, the slopes intersect a meridian curve an even number of times) and all these slopes can be achieved.
Any such surface with non-zero genus is non-orientable and $\del$-compressible.
Incompressible $\del$-parallel surfaces are the annuli which are unions of fibers in fibrations of the solid torus as an $S^1$-bundle, as well as $\del$-parallel discs.
\end{prop}

\begin{prop}[\cite{Przytycki, Rubinstein}]
\label{prop:incompressibleinthickenedtorus}
An incompressible surface in $T^2\times I$ is isotopic to one of the following:
\begin{enumerate}
    \item a trivial sphere or disc;
    \item an annulus $\gamma\times I$;
    \item a $\del$-parallel annulus or torus;
    \item a nonorientable surface $F$, which is $\del$-compressible and uniquely determined by two different slopes $\frac{p_0}{q_0} = F\cap (T^2\times {0})$ and $\frac{p_1}{q_1} = F\cap (T^2\times {1})$ where the curves representing these slopes intersect an even number of times.
\end{enumerate}
In the last case, $F$ has non-orientable genus equal to the length of the minimal sequence of curves in the torus $(\gamma_1, \ldots,\gamma_n)$ from $\frac{p_0}{q_0}$ to $\frac{p_1}{q_1}$ where $\gamma_i$ and $\gamma_{i+1}$ intersect twice.
\end{prop}

\begin{defn}
\label{defn:pseudovertical}
Let $M$ be a Seifert fibered space and let $T$ be a collection of solid torus neighbourhoods of each singular fiber such that $T$ is a union of fibers.
A surface in a Seifert fibered space is \emph{pseudo-vertical} if it is isotopic to a surface that is a union of fibers in $M-T$ and is incompressible in each solid torus component of $T$.
\end{defn}

A vertical surface (that is, a union of regular fibers of $M$) is also pseudo-vertical as we can isotope it to be disjoint from $T$.

\begin{lemma}
\label{lemma:nonorientablesurfacesinSFS}
Let $M$ be an irreducible Seifert fibered space with non-empty boundary, and with a (possibly empty) graph $\Gamma$ in $\del M$ consisting of a collection of vertical fibers.
If $S$ is an incompressible, $\del$-incompressible surface in $M$ disjoint from $\Gamma$, then $S$ is isotopic to a horizontal or pseudo-vertical surface, or is a $\del$-parallel disc or a trivial sphere.
\end{lemma}

This result is a modification of the standard proof in the case when $S$ is orientable (see~\cite[Proposition 1.11]{Hatcher3Mflds} or~\cite[Proposition 10.4.9]{Martelli}) and Mijatovi\'{c}'s result in the case where there are no singular fibers~\cite[Proposition 2.6]{MijatovicTriangulationsSFS}).
If $M$ is closed then $S$ may additionally be \emph{pseudo-horizontal}; this case is discussed in~\cite[Theorem 2.5]{Frohman}.

\begin{proof}
Suppose that $S$ is not a $\del$-parallel disc or a trivial sphere.
If $M$ is the solid torus, the result follows from Proposition~\ref{prop:incompressibleinsolidtori}, noting that $S$ may be $\del$-parallel if it is parallel to an annulus in the boundary containing curves of $\Gamma$.
Otherwise $M$ is $\del$-irreducible.
Take a collection of disjoint vertical annuli $A$ disjoint from $\Gamma$, consisting of $n$ that separate a neighbourhood of each singular fiber from the rest of $M$, where these neighbourhoods are themselves disjoint from $\Gamma$, and then some further annuli that cut the rest of $M$ into a solid torus.
Isotope $S$ such that $S$ is transverse to $A$ and such that $|S\cap A|$ is minimised.
Consider $S\cap A$, which consists of arcs and closed curves.
Note that $S\cap A$ does not contain any curves that are trivial in $S$ or $A$, as then by the irreducibility of $M$ and the incompressibility of $S$ we could reduce $|S\cap A|$.
It also does not contain any arcs that are $\del$-parallel in $S$ or $A$ as $S$ is $\del$-incompressible and $M$ is irreducible and $\del$-irreducible.
As any arc in an annulus that starts and ends on the same boundary component is $\del$-parallel, this means that $S\cap A$ contains none of these.
Thus $S$ intersects $A$ in a collection of spanning arcs (that is, arcs that run from one boundary component of $A$ to the other) and vertical fibers.
By the same reasoning, the same holds for the intersection of $S$ with each annulus of $\del M - \del A$.
Note that as $S$ is embedded it must intersect each annulus of $A$ and $\del M - \del A$ in only one of these two types.
If $S$ intersects any annulus in a spanning arc, it intersects all neighbouring annuli to that one in a spanning arc, and hence (as $M$ is connected) all annuli in the collection.
Thus the two types are incompatible, so $S\cap A$ must consist of only one of the two.

Let $M_0$ be $M\backslash\backslash A$, and let $S_0$ be $S\backslash \backslash A$ in $M_0$.
We claim that $S_0$ is incompressible: consider the boundary of some compressing disc.
This curve bounds a disc in $S$ as $S$ is incompressible, and that disc intersects the annuli $A$ in simple closed curves. We can use the irreducibility of $M$ to isotope $S$ through this disc, in the process reducing $|S\cap A|$.

Suppose that $S\cap A$ consists of vertical fibers.
Let $M_1$ be the non-singular-fiber-neighbourhood component of $M_0$, and let $S_1$ be $S_0\cap M_1$.
Recall the classification of incompressible surfaces in the solid torus from Proposition~\ref{prop:incompressibleinsolidtori}.
Then $S_1$ is a collection of vertical annuli: it is incompressible, not a meridian disc, and cannot be an incompressible non-orientable surface as one of its boundary curves intersects the meridian once, which is odd.
The remaining part of $S$, its intersection with the singular fiber neighbourhoods, can vary: if a given fiber has odd multiplicity, $S$ does not intersect the fiber so must be disjoint from the neighbourhood by the minimality of $|S\cap A|$.
If the multiplicity is even, it may intersect the neighbourhood in a punctured non-orientable incompressible surface.
Gluing up, we find that $S$ is pseudo-vertical as claimed. 

Suppose that $S\cap A$ consists of spanning arcs.
(In this case $\Gamma$ must be empty, as otherwise $S$ would intersect it.)
We claim that we can isotope $S$ so that $S_0$ is $\del$-incompressible: if it is not, let $D$ be a non-trivial $\del$-compression disc for $S_0$.
Consider the arc $\alpha = \del D \cap \del M_0$.
We will isotope $\alpha$ so that it is contained in $\del M_0\cap \del M$.
In this situation we are done: as $S$ is $\del$-incompressible, we can use this boundary compression to reduce $|S\cap A|$.
Note that $\alpha$ consists (up to isotopy) of a collection of arcs in annuli. These annuli are alternately from $\del M_0\cap \del M$ and $\del M_0\cap A$.
If $\alpha$ is contained in one component of $\del M_0\cap A$, we can use an isotopy in a collar of the boundary to push $\alpha$ into an adjacent component of $\del M_0 \cap \del M$.
Otherwise, starting at one end of $\alpha$, use an isotopy in the collar of the boundary to push this arc $\alpha$ into the adjacent annulus.
We can continue this until $\alpha$ is contained in a single annulus.
Thus $S_0$ is $\del$-incompressible so is a collection of meridian discs in each of the solid tori, and hence is horizontal.
\end{proof}

Finally, we classify the Seifert fibered spaces over Euler characteristic zero orbifolds.

\begin{lemma}
\label{lemma:horizontalmobiusclassification}
    The only Seifert fibered spaces containing horizontal M\"obius bands are those with the Seifert data $[D^2, 1/2, 1/2]$ or that of a circle bundle over a M\"obius band.
    The only additional Seifert fibered space containing a horizontal annulus has the Seifert structure of a circle bundle over an annulus.
\end{lemma}
\noindent In the first two cases $M$ is homeomorphic to $K\ttimes I$.

\begin{proof}
    Write $S$ for the M\"obius band.
    Let $M$ be such a Seifert fibered space.
    As $\chi(S) = 0$, the base orbifold $G$ has $\chi(G) = 0$.
    If we write $\chi(G) =2-a-b-\sum_i (1-\frac{1}{p_i})$ where $b$ is the number of boundary components, $a$ is twice its genus (if orientable) or its nonorientable genus (otherwise), and $(p_i)$ is the multiplicities of the singular fibers, then we can see that $a + (b-1) + \sum_i (1-\frac{1}{p_i}) = 1$.
    As $b\geq 1$, each of these terms is nonnegative.

    That $M$ contains a horizontal M\"obius band is equivalent to the statement that the base orbifold $G$ of $M$ is covered by a M\"obius band.
    We can work by cases: if $a >0$, we have a M\"obius band with no singular fibers.
    If $b > 1$, then $b$ must be 2 so we have an annulus, which the M\"obius band does not cover but the annulus (trivially) does.
    Otherwise, $a=0$ and $b=1$, so $G$ is a disc with some singular fibers with multiplicities such that $\sum_i (1-\frac{1}{p_i}) = 1$.
    For each $p > 1$, we have $\frac{1}{2} \leq 1-\frac{1}{p} < 1$, so the only solution is two singular fibers, each with multiplicity two.
\end{proof}

\subsection{Conventions and notation}
We will write $T^2$ for the 2-torus and $K$ for the Klein bottle.

\noindent If $S$ is a properly-embedded sub-manifold of a (piecewise-linear) manifold $M$, the manifold $M\backslash\backslash S$ is the complement of a small open neighbourhood $N(S)$ of $S$ in $M$.

\noindent If $A$ is a subset of a manifold, $|A|$ is the number of connected components of $A$.

\begin{convention}
\label{convention:handlestructure}
We take the definition of a handle structure to require the following:
\begin{enumerate}
    \item each $k$-handle, with product structure $D^k\times D^{3-k}$, intersects the handles of lower index in exactly $\del D^k \times D^{3-k}$, and is disjoint from the other $k$-handles;
    \item 1-handles and 2-handles intersect in a manner compatible with their respective product structures; that is, a 1-handle $D^1\times D^2$ intersects each 2-handle $D^2\times D^1$ in segments of the form $D^1\times \gamma$ in the 1-handle and $\lambda\times D^1$ in the 2-handle, where $\gamma$ and $\lambda$ are collections of arcs in $\del D^2$ in the respective product structures.
\end{enumerate}
\end{convention}

\begin{defn}
    The \emph{size} of a triangulation $\mathcal{T}$, $\norm{\mathcal{T}}$, is the number of tetrahedra in it.
    The \emph{size} of a handle structure $\mathcal{H}$, $\norm{\mathcal{H}}$, is the number of 0-handles in it.
\end{defn}

\section{Existence of a minimal degree fundamental horizontal surface}
\label{section:fundamentalhorizontalsurface}

Let $M$ be a Seifert fibered space whose boundary is non-empty, equipped with a triangulation $\mathcal{T}$.
We will show that, so long as $M$ is not on a short list of exceptions, there is a fundamental horizontal surface in $M$ whose induced covering of the base orbifold of the Seifert fibration is of minimal degree.

To produce the desired fundamental surface we will use normal surface theory, as described in Section~\ref{section:normalsurface}, and the facts about incompressible and $\del$-incompressible surfaces in Seifert fibered spaces that we discussed in Section~\ref{section:sfs}.

\begin{prop}
Suppose that $M$ is a Seifert fibered space with non-empty boundary that is not $S^1\times D^2$, $T^2\times I$, or $K\ttimes I$.
Let $\mathcal{T}$ be a triangulation of $M$.
Let $p$ be the lowest common multiple of the multiplicities of the singular fibers.
There is a fundamental horizontal normal surface in $M$ that is a degree $p$ cover of the underlying orbifold.
If $M$ is $S^1\times D^2$, then there is a fundamental normal meridian disc.
\label{prop:fundamentalhorizontalsurface}
\end{prop}

\begin{proof}
The solid torus case follows from Corollary 6.4 of~\cite{JacoTollefson}.
Otherwise, let $n$ be the number of singular fibers (which may be zero).
Note that $M$ is irreducible and $\del$-irreducible.

Consider $M$ to be constructed by taking a circle bundle $M'$ over a surface $\Sigma$ with non-empty boundary, then gluing $n$ solid tori (that is, neighbourhoods of singular fibers) along vertical annuli to a single boundary component of $M'$.
If one of these solid tori is a neighbourhood of a $(p_i, q_i)$ singular fiber, then there is a meridian curve in its boundary, transverse to the fibration, that intersects the gluing annulus in $p_i$ spanning arcs.

Take a degree $p$ horizontal surface in $M'$: if $\Sigma$ is orientable, this will be $p$ copies of $\Sigma$; otherwise it will be $\floor{\frac{p}{2}}$ copies of the double cover of $\Sigma$, in addition to one copy of $\Sigma$ if $p$ is odd.
Either way, its intersection with each vertical annulus in the boundary of $M'$ will be $p$ spanning arcs.
We can thus take $\frac{p}{p_i}$ meridian discs in each of these singular fiber neighbourhoods and attach them to the degree $p$ horizontal surface in $M'$ to form a degree $p$ horizontal surface in $M$.
This surface is incompressible (as it is a finite degree cover of the base orbifold and hence is $\pi_1$-injective) and $\del$-incompressible (by the same argument on the double of $M$), and does not contain any trivial spheres or discs, so let $F$ be a minimal normal surface that is isotopic to it.

\begin{claim}
\label{claim:fundhoriz:negativeeulerchar}
    The Euler characteristic of $F$ is negative.
\end{claim}

\begin{claimproof}
As $\chi(F)$ is a multiple of the Euler characteristic of the base orbifold $O$ of $M$, it is enough to show that $\chi(O)$ is negative.
We classified the Seifert fibered spaces whose orbifolds have Euler characteristic zero and non-empty boundary in Lemma~\ref{lemma:horizontalmobiusclassification}, and found that a circle bundle over one of them is $K\ttimes I$ or $T^2\times I$.
The only way $\chi(O)$ can be positive is if $O$ is a disc with at most one cone point, in which case $M$ is a solid torus.
Our manifold $M$ is not one of these three manifolds.
\end{claimproof}

Suppose that $F = G_1 + G_2$ is a non-trivial sum of normal surfaces that minimises $|G_1\cap G_2|$ among all such non-trivial decompositions of $F$.
If one of $G_1$ or $G_2$ were not connected, we could write $G_1$, say, as $G_1' \cup G_1''$, and then $F = G_1' + (G_1'' + G_2)$ would be a sum with $|G_1'\cap (G_1'' + G_2)| < |G_1\cap G_2|$, so the $G_i$ must be connected.

\begin{claim}
At least one of $G_1$ and $G_2$ is horizontal.
\end{claim}

\begin{claimproof}
As $F$ is horizontal it is incompressible and $\del$-incompressible, and is not a trivial disc or sphere.
By Theorem~\ref{thm:matveev}, the same holds for $G_1$ and $G_2$.
Thus $G_1$ and $G_2$ are horizontal or pseudo-vertical by Lemma~\ref{lemma:nonorientablesurfacesinSFS}.

Suppose both $G_1$ and $G_2$ are pseudo-vertical. 
As $\del F$ is not a vertical curve, there must be at least one component of each of $\del G_1$ and $\del G_2$ on each boundary component of $M$, so $|\del M|$ is at most two.
As both $G_1$ and $G_2$ have non-empty boundary, they are either annuli or nonorientable surfaces with one boundary component.

Suppose $p$ is odd.
As summing normal surfaces and curves is additive on $\ZZ_2$-homology, $\del F = \del G_1 + \del G_2$ in $H_1(\del M; \ZZ_2)$.
Since $F$ intersects each regular fiber $p$ times, $\del F$ intersects any boundary component in $p'$ curves where $p'$ divides $p$ and thus is odd.
As a consequence, $\del F$ is nontrivial in the restriction to the $\ZZ_2$-homology of each boundary component.
Since $p$ is odd, there are no even multiplicity singular fibers, so there are no nonorientable pseudo-vertical surfaces.
In this case, as $G_1$ and $G_2$ are both vertical annuli, we can show that $\del (G_1 + G_2)$ is trivial in $\ZZ_2$-homology on at least one boundary component.
If there is only one boundary component, $\del (G_1 + G_2) \equiv 2\del G_1 \equiv (0,0)\in H_1(T^2; \ZZ_2)$.
If there are two, as there is at least one component of the boundary of each of the surfaces $G_1$ and $G_2$ on each boundary component, $G_1\cup G_2$ intersects each boundary component as a union of two vertical curves, which similarly is trivial in $\ZZ_2$-homology.
Either way, we have a contradiction.

Otherwise $p$ is even so $F$ is orientable.
Then $2F$, the double of $F$ as a normal surface vector, is minimal in its isotopy class and is incompressible and $\del$-incompressible.
As $2F = 2G_1 + 2G_2$, by Theorem~\ref{thm:jacotollefson} the doubles of $G_1$ and $G_2$ are incompressible and $\del$-incompressible.
As $2G_i$ is orientable, and $\del (2G_i)$ is two copies of $\del G_i$ and so consists of vertical fibers, $2G_i$ is vertical and thus has Euler characteristic 0.
But as $\chi(2G_i) = 2\chi(G_i)$, each $G_i$ also has Euler characteristic zero, so $\chi(F) = 0$, which contradicts Claim~\ref{claim:fundhoriz:negativeeulerchar}.
\end{claimproof}

We can thus assume that $G_1$ is horizontal.
It remains to show that the degree of its induced covering of the base orbifold is $p$.
Now, $G_1$ intersects the singular fiber neighbourhoods (which were cut out by vertical annuli) in a collection of meridian discs.
A meridian disc around a multiplicity $p_i$ fiber intersects the relevant vertical annulus in $p_i$ spanning arcs.
We know that, up to isotopy, $G_1$ intersects the circle bundle $M'$ as a horizontal surface, and so in particular intersects each vertical annulus in the boundary component of $M'$ along which we glued the singular fibers in an equal number of spanning arcs.
Thus this number must be a multiple of all of the multiplicities: that is, it is $kp$ for some integer $k$, recalling that $p$ is the lowest common multiple of the $p_i$, and $G_1|_{M'}$ is a degree $kp$ cover of $\Sigma$.
(If $n = 0$, taking the lowest common multiple of the empty set to be 1 by definition, this reasoning holds vacuously.)

Note that $k$ is at least one, and $\frac{1}{k}\chi(G_1) = \chi(F) < 0$ by Claim~\ref{claim:fundhoriz:negativeeulerchar}.
Thus $\chi(G_1)$ is uniquely maximised when $k = 1$; as $\chi(G_1) \geq \chi(F)$, and this maximum achieves equality, $k$ must be 1 and hence $G_1$ is a degree $p$ horizontal surface.
Thus there is a fundamental such surface.
\end{proof}

\section{Split handle structures}
\label{section:ths}

In this section we introduce \ths{}s, which naturally arise when we cut handle structures along normal surfaces.
The motivation for \ths{}s is that they support a theory of normal surfaces (compare to Section~\ref{section:normalsurface}) and their complexity does not grow fast when we cut along one of these normal surfaces.
We will use them in Section~\ref{section:fundamentalannulicollection} to show that there is a maximal collection of normal vertical annuli whose edge weight is at most $c^{\norm{\mathcal{T}}^2}$.

To begin, we change perspectives from the triangulation $\mathcal{T}$ to its dual handle structure $\mathcal{H}$.
Taking this dual is a standard, canonical operation, and there is a natural inclusion map from normal surfaces in the triangulation to normal surfaces in the dual handle structure.

\begin{defn}
\label{defn:dual_handle}
Let $\mathcal{T}$ be a triangulation (or cell structure) of a 3-manifold $M$.
The \emph{dual handle structure} $\mathcal{H}$ for $M$ is formed by taking one $(3-k)$-handle for each $k$-simplex (or cell) of $\mathcal{T}$ that is not contained in the boundary and gluing them in the corresponding way.
\end{defn}

Split handle structures naturally arise when we cut along normal surfaces in these dual handle structures.
We keep track of the \emph{forbidden region}, which is the part of the boundary that comes from the normal surface, and of \emph{parallelity pieces}, which are $I$-bundles over surfaces from regions where the normal surface runs close to itself.
Figure~\ref{fig:splithandleexample} is a motivating example; it shows the result of cutting a tetrahedron (in some larger triangulation) along an elementary triangle, from the point of view of first, the triangulation, and second, the dual split handle structure.
The part of the boundary of the cut-open tetrahedron that comes from the elementary surface is shaded in red.

\begin{figure}[th]
  \centering
  \begin{subfigure}[b]{0.43\textwidth}
  	\centering
   	\resizebox{0.9\textwidth}{!}{
            \tikzset{every picture/.style={line width=0.75pt}} 
            \begin{tikzpicture}[x=0.75pt,y=0.75pt,yscale=-1,xscale=1]
            \node (1) at (246.43, 78.86) {};
            \node (2) at (384.43,146.86) {};
            \node (3) at (180.43,209.86) {};
            \node (4) at (298.43,211.86) {};
            \draw[rounded corners=2pt, color=teal,line width=3pt, draw opacity=1 ][fill=red  ,fill opacity=0.4 ] (293,100.7) -- (265.36,128.52) -- (223,122.2) -- cycle ;
            \draw[rounded corners=1pt]    (1.center) -- (3.center) -- (4.center) -- cycle;
            \draw[rounded corners=1pt] (1.center) -- (4.center) -- (2.center) -- cycle;
            \draw  [dash pattern={on 4.5pt off 4.5pt}]  (3.center) -- (2.center);
            \end{tikzpicture}
        }
    \caption{An elementary disc in a tetrahedron.}
  \end{subfigure}
  \begin{subfigure}[b]{0.55\textwidth}
  	\centering
    \resizebox{0.25\textwidth}{!}{
   	\hspace*{-0.05\textwidth}
            \tikzset{every picture/.style={line width=0.75pt}} 
            \begin{tikzpicture}[x=0.75pt,y=0.75pt,yscale=-1,xscale=1]
            \draw[rounded corners=2pt, red,draw opacity=1 ][fill=red  ,fill opacity=0.4 ] (288.93,100.07) -- (265.36,128.52) -- (224.93,121.57) -- cycle ;
            \draw[rounded corners=1pt]   (288,101) -- (246.43,79.86) -- (226.5,121.57) ;
            \draw[rounded corners=1pt]   (226.5,121.57) -- (246.43,79.86) -- (265.36,128) ;
            \end{tikzpicture}
        }
        \resizebox{0.8\textwidth}{!}{
            \tikzset{every picture/.style={line width=0.75pt}} 
            \begin{tikzpicture}[x=0.75pt,y=0.75pt,yscale=-1,xscale=1]
            \node (2) at (384.43,146.86) {};
            \node (3) at (180.43,209.86) {};
            \node (4) at (298.43,211.86) {};
            \node (5) at (291,101.5) {};
            \node (6) at (265.36,127) {};
            \node (7) at (225,122.2) {};
            \draw[rounded corners=2pt, color=red,draw opacity=1 ][fill=red  ,fill opacity=0.4 ] (293,100.7) -- (265.36,128.52) -- (223,122.2) -- cycle ;
            \draw[rounded corners=1pt]    (7.center) -- (3.center) -- (4.center) -- (6.center);
            \draw[rounded corners=1pt] (6.center) -- (4.center) -- (2.center) -- (5.center);
            \draw  [dash pattern={on 4.5pt off 4.5pt}]  (3.center) -- (2.center);
            \end{tikzpicture}
        }
    \caption{After cutting along it.}
  \end{subfigure}
  \begin{subfigure}[b]{0.43\textwidth}
  	\centering
   	\resizebox{0.8\textwidth}{!}{
        \begin{tikzpicture}
            \node[circle, draw=black, line width=0.75pt, minimum size=18pt] (1) at (0,{(sqrt(3)*1.2-0.05)}) {};
            \node[circle, draw=black, line width=0.75pt, minimum size=18pt, fill=white] (2) at (-0.95*1.2,0) {};
            \node[circle, draw=black, line width=0.75pt, minimum size=18pt, fill=white] (3) at (0.95*1.2,0) {};
            \node[circle, draw=black, line width=0.75pt, minimum size=18pt, fill=white] (4) at (0,{(1/sqrt(3))+0.05}) {};
            \draw[line width=0.75pt, double, double distance=3pt, rounded corners = 4pt] (2) -- (3) -- (4) -- (2) -- (1) -- (4);
            \draw[line width=0.75pt, double, double distance=3pt, rounded corners = 4pt] (1) -- (3);
            \draw[line width=2pt, teal, rounded corners=8pt] (1.center) -- (2.center) -- (4.center) -- cycle;
        \end{tikzpicture}
        }
    \caption{The boundary curve of the elementary disc in the boundary graph of the dual handle.}
  \end{subfigure}
  \begin{subfigure}[b]{0.55\textwidth}
  	\centering
   \resizebox{0.8\textwidth}{!}{
       \resizebox{0.45\textwidth}{!}{
        \begin{tikzpicture}
            \fill[fill=red, fill opacity=0.4] (0,{(sqrt(3)*1.2-0.05)}) -- (-0.95*1.2,0) -- (0,{(1/sqrt(3))+0.05});
            \node[circle, draw=black, line width=0.75pt, minimum size=18pt, fill=white] (1) at (0,{(sqrt(3)*1.2-0.05)}) {};
            \node[circle, draw=black, line width=0.75pt, minimum size=18pt, fill=white] (2) at (-0.95*1.2,0) {};
            \node[circle, draw=black, line width=0.75pt, minimum size=18pt, fill=white] (4) at (0,{(1/sqrt(3))+0.05}) {};
            \draw[line width=0.75pt, double, double distance=3pt, rounded corners = 4pt] (2) -- (1) -- (4) -- (2);
        \end{tikzpicture}
        }
   	\resizebox{0.7\textwidth}{!}{
        \begin{tikzpicture}
            \fill[fill=red, fill opacity=0.4] (0,{(sqrt(3)*1.2-0.05)}) -- (-0.95*1.2,0) -- (0,{(1/sqrt(3))+0.05});
            \node[circle, draw=black, line width=0.75pt, minimum size=18pt, fill=white] (1) at (0,{(sqrt(3)*1.2-0.05)}) {};
            \node[circle, draw=black, line width=0.75pt, minimum size=18pt, fill=white] (2) at (-0.95*1.2,0) {};
            \node[circle, draw=black, line width=0.75pt, minimum size=18pt, fill=white] (3) at (0.95*1.2,0) {};
            \node[circle, draw=black, line width=0.75pt, minimum size=18pt, fill=white] (4) at (0,{(1/sqrt(3))+0.05}) {};
            \draw[line width=0.75pt, double, double distance=3pt, rounded corners = 4pt] (2) -- (3) -- (4);
            \draw[line width=0.75pt, double, double distance=3pt, rounded corners = 4pt] (2) -- (1) -- (4) -- (2);
            \draw[line width=0.75pt, double, double distance=3pt, rounded corners = 4pt] (1) -- (3);
        \end{tikzpicture}
        }
        }
    \caption{The boundary graphs of the resulting split handles. (We have rotated the first one so that the forbidden region is contained in the graph as we draw it on the plane, rather than being its complement in $S^2$.)}
  \end{subfigure}
  \caption[The pieces resulting from cutting along an elementary triangle in a tetrahedron, and the split handle structure point of view.]{The pieces resulting from cutting along an elementary triangle in a tetrahedron (in the interior of some larger triangulation), and the picture in the dual split handle structure. To depict the handles, we draw their boundary graphs (in $S^2$). The forbidden region is shaded in red. There are no parallelity pieces.}
  \label{fig:splithandleexample}
\end{figure}

Split handle structures are reminiscent of sutured handle structures, which were devised by Lackenby~\cite[\S5]{LackenbyCertificationKnottedness} to allow for normal-surface-type arguments in the context of Scharlemann's combinatorial approach to Gabai's sutured manifold decompositions~\cite{ScharlemannSutured}.
Motivated this, we will use the term ``sutures'' for the boundary of the forbidden region in $\del M$.

Figure~\ref{fig:splithandleexamplenottriangulationdual} shows a split 0-handle whose boundary graph contains sutures.
We have cut along an elementary disc in the dual split handle structure that does not correspond to an elementary disc in the triangulation, as its boundary passes through a lake (which is dual to a vertex of the tetrahedron).
Each arc in a lake produces a suture when we cut along it.

Normal surfaces in split handle structures are also evocative of the normal surface theory of handle structures with boundary pattern, if we require that elementary discs do not intersect the pattern.
However, in the boundary pattern case, we usually require that the pattern is contained in the 1-skeleton of the induced handle structure on the boundary (see~\cite[\S3]{Matveev}).
Here, the sutures are normal curves.

\begin{figure}[th]
  \centering
  \begin{subfigure}[b]{0.4\textwidth}
  	\centering
   	\resizebox{0.7\textwidth}{!}{
        \begin{tikzpicture}
            \node[circle, draw=black, line width=0.75pt, minimum size=18pt] (1) at (0,{(sqrt(3)*1.2-0.05)}) {};
            \node[circle, draw=black, line width=0.75pt, minimum size=18pt, fill=white] (2) at (-0.95*1.2,0) {};
            \node[circle, draw=black, line width=0.75pt, minimum size=18pt, fill=white] (3) at (0.95*1.2,0) {};
            \node[circle, draw=black, line width=0.75pt, minimum size=18pt, fill=white] (4) at (0,{(1/sqrt(3))+0.05}) {};
            \draw[line width=0.75pt, double, double distance=3pt, rounded corners = 4pt] (2) -- (3) -- (4) -- (2) -- (1) -- (4);
            \draw[line width=0.75pt, double, double distance=3pt, rounded corners = 4pt] (1) -- (3);
            \coordinate (o) at ($(1.center)+(-0.05, -0.2)$);
            \coordinate (x) at ($(1.center)+(-0.1, -0.1)$);
            \coordinate (y) at ($(2.center)+(-0.1, 0.1)$);
            \coordinate (t) at ($(1.center)+(-0.7, -0.35)$);
            \coordinate (r) at ($(1.center)+(-0.8, -0.3)$);
            \coordinate (p) at ($(1)!0.5!(2) + (0.3,0.05)$);
            \coordinate (s) at ($(1)!0.5!(2) + (0,-0.4)$);
            \coordinate (q) at ($(1)!0.5!(2) + (-0.5, 0.35)$);
            \draw [tension=0.7, teal, line width=2pt] plot [smooth cycle] coordinates {(x) (s) (y) (q)};
        \end{tikzpicture}
    }
    \caption{The boundary of an elementary disc in the boundary graph in $\del H$.}
  \end{subfigure}
  \begin{subfigure}[b]{0.45\textwidth}
  	\centering
   	\resizebox{0.65\textwidth}{!}{
        \begin{tikzpicture}
            \node (1) at (0,{(sqrt(3)*1.2-0.05)}) {};
            \node (2) at (-0.95*1.2,0) {};
            \node (3) at (0.95*1.2,0) {};
            \node (4) at (0,{(1/sqrt(3))+0.05}) {};
            \coordinate (o) at ($(1.center)+(-0.05, -0.2)$);
            \coordinate (x) at ($(1.center)+(-0.1, -0.1)$);
            \coordinate (y) at ($(2.center)+(-0.1, 0.1)$);
            \coordinate (t) at ($(1.center)+(-0.7, -0.35)$);
            \coordinate (r) at ($(1.center)+(-0.8, -0.3)$);
            \coordinate (p) at ($(1)!0.5!(2) + (0.3,0.05)$);
            \coordinate (s) at ($(1)!0.5!(2) + (0,0.1)$);
            \coordinate (q) at ($(1)!0.5!(2) + (-0.4, 0.35)$);
            \filldraw[line width=2pt, red, fill=red, fill opacity=0.4, tension=0.7] plot [smooth cycle] coordinates {(x) (s) (y) (q)};
            \node[circle, draw=black, line width=0.75pt, minimum size=18pt, fill=white] (5) at (0,{(sqrt(3)*1.2-0.05)}) {};
            \node[circle, draw=black, line width=0.75pt, minimum size=18pt, fill=white] (6) at (-0.95*1.2,0) {};
            \node[circle, draw=black, line width=0.75pt, minimum size=18pt, fill=white] (7) at (0.95*1.2,0) {};
            \node[circle, draw=black, line width=0.75pt, minimum size=18pt, fill=white] (8) at (0,{(1/sqrt(3))+0.05}) {};
            \draw[line width=0.75pt, double, double distance=3pt, rounded corners = 4pt] (5) -- (8) -- (6) -- (7) -- (8) -- (6);
            \draw[line width=0.75pt, double, double distance=3pt, rounded corners = 4pt] (5) -- (7);
        \end{tikzpicture}
    }
    \resizebox{0.33\textwidth}{!}{
    \begin{tikzpicture}
            \coordinate (1) at (0,0.95);
            \coordinate (2) at (0,-0.95);
            \coordinate (a) at (-0.4, 0);
            \coordinate (b) at (-0.9, 0);
            \coordinate (c) at ($(1.west)+(-0.1,0)$);
            \coordinate (f) at ($(1.west)+(-0.2,0)$);
            \coordinate (d) at ($(2.west)+(-0.1,0)$);
            \coordinate (e) at ($(2.west)+(-0.2,0)$);
            \filldraw[fill=red, fill opacity=0.4, rounded corners=1pt, draw=red, line width=2pt] (1.center) -- (c) .. controls (a) .. (d) -- (e) .. controls (b) .. (f) -- (c);
            \draw[line width=0.75pt, double, double distance=3pt] (1) -- (2);
            \node[circle, draw=black, line width=0.75pt, minimum size=16pt, fill=white] at (1) {};
            \node[circle, draw=black, line width=0.75pt, minimum size=16pt, fill=white] at (2) {};
        \end{tikzpicture}
    }
    \caption{The two split 0-handles resulting from cutting along this elementary disc.}
  \end{subfigure}
  \caption[The induced split handle structure from cutting along an elementary disc in a split 0-handle that does not correspond to one in a tetrahedron.]{The induced split handle structure from cutting along an elementary disc in a split 0-handle $H$ that does not correspond to one in a tetrahedron. The forbidden region is shaded in red, and the sutures are the thick red lines.}
  \label{fig:splithandleexamplenottriangulationdual}
\end{figure}

\begin{defn}
    Let $P$ be an $I$-bundle over a surface $\Sigma$, so $P$ is equipped with a homeomorphism to $\Sigma\pmttimes I$.
    The \emph{horizontal boundary} of $P$, $\del_h P$, is $\Sigma\pmttimes \del I$, and the \emph{vertical boundary} of $P$, $\del_v P$, is $\del\Sigma\pmttimes I$.
\end{defn}

The horizontal and vertical boundary of $\Sigma\times I$ (where $\Sigma$ is a disk with two punctures) is shown in Figure~\ref{fig:horizboundary}.

\begin{figure}[th]
  \centering
  \resizebox{0.6\textwidth}{!}{
\begin{tikzpicture}[x=0.75pt,y=0.75pt,yscale=-1,xscale=1]

\draw  [pattern={Lines[angle=90, distance=2.1mm, line width=0.2mm]}] (99.5,180.38) .. controls (99.5,152.55) and (152.22,130) .. (217.25,130) .. controls (282.28,130) and (335,152.55) .. (335,180.38) .. controls (335,208.2) and (282.28,230.75) .. (217.25,230.75) .. controls (152.22,230.75) and (99.5,208.2) .. (99.5,180.38) -- cycle ;
\draw  [fill={rgb, 255:red, 255; green, 255; blue, 255 }  ,fill opacity=1 ] (100,162.38) .. controls (100,137.87) and (152.16,118) .. (216.5,118) .. controls (280.84,118) and (333,137.87) .. (333,162.38) .. controls (333,186.88) and (280.84,206.75) .. (216.5,206.75) .. controls (152.16,206.75) and (100,186.88) .. (100,162.38) -- cycle ;
\draw  [pattern={Lines[angle=90, distance=2.1mm, line width=0.2mm]}] (250.3,170.45) .. controls (250.3,162.1) and (261.16,155.32) .. (274.55,155.32) .. controls (287.94,155.32) and (298.8,162.1) .. (298.8,170.45) .. controls (298.8,178.8) and (287.94,185.58) .. (274.55,185.58) .. controls (261.16,185.58) and (250.3,178.8) .. (250.3,170.45) -- cycle ;
\draw  [draw opacity=0][fill={rgb, 255:red, 255; green, 255; blue, 255 }  ,fill opacity=1 ] (250.3,178.83) .. controls (250.3,170.47) and (261.16,163.7) .. (274.55,163.7) .. controls (287.94,163.7) and (298.8,170.47) .. (298.8,178.83) .. controls (298.8,187.18) and (287.94,193.95) .. (274.55,193.95) .. controls (261.16,193.95) and (250.3,187.18) .. (250.3,178.83) -- cycle ;
\draw   (250.3,170.45) .. controls (250.3,162.1) and (261.16,155.32) .. (274.55,155.32) .. controls (287.94,155.32) and (298.8,162.1) .. (298.8,170.45) .. controls (298.8,178.8) and (287.94,185.58) .. (274.55,185.58) .. controls (261.16,185.58) and (250.3,178.8) .. (250.3,170.45) -- cycle ;
\draw  [draw opacity=0][fill={rgb, 255:red, 255; green, 255; blue, 255 }  ,fill opacity=1 ] (251.48,174.16) .. controls (254.63,168.09) and (263.77,163.7) .. (274.55,163.7) .. controls (285.21,163.7) and (294.27,167.99) .. (297.52,173.96) -- (274.55,178.83) -- cycle ; \draw   (251.48,174.16) .. controls (254.63,168.09) and (263.77,163.7) .. (274.55,163.7) .. controls (285.21,163.7) and (294.27,167.99) .. (297.52,173.96) ;  

\draw  [pattern={Lines[angle=90, distance=2.1mm, line width=0.2mm]}] (142.5,158.66) .. controls (142.5,147.87) and (159.98,139.13) .. (181.55,139.13) .. controls (203.12,139.13) and (220.6,147.87) .. (220.6,158.66) .. controls (220.6,169.44) and (203.12,178.19) .. (181.55,178.19) .. controls (159.98,178.19) and (142.5,169.44) .. (142.5,158.66) -- cycle ;
\draw  [draw opacity=0][fill={rgb, 255:red, 255; green, 255; blue, 255 }  ,fill opacity=1 ] (142.5,169.47) .. controls (142.5,158.68) and (159.98,149.94) .. (181.55,149.94) .. controls (203.12,149.94) and (220.6,158.68) .. (220.6,169.47) .. controls (220.6,180.26) and (203.12,189) .. (181.55,189) .. controls (159.98,189) and (142.5,180.26) .. (142.5,169.47) -- cycle ;
\draw   (142.5,158.66) .. controls (142.5,147.87) and (159.98,139.13) .. (181.55,139.13) .. controls (203.12,139.13) and (220.6,147.87) .. (220.6,158.66) .. controls (220.6,169.44) and (203.12,178.19) .. (181.55,178.19) .. controls (159.98,178.19) and (142.5,169.44) .. (142.5,158.66) -- cycle ;
\draw  [draw opacity=0][fill={rgb, 255:red, 255; green, 255; blue, 255 }  ,fill opacity=1 ] (143.73,164.59) .. controls (148.06,156.17) and (163.35,149.94) .. (181.55,149.94) .. controls (199.49,149.94) and (214.6,155.99) .. (219.18,164.23) -- (181.55,169.47) -- cycle ; \draw   (143.73,164.59) .. controls (148.06,156.17) and (163.35,149.94) .. (181.55,149.94) .. controls (199.49,149.94) and (214.6,155.99) .. (219.18,164.23) ;  
\draw[arrows = {-Latex[width=0pt 10, length=10pt]}]    (311,113.4) -- (262.34,140.82) ;
\draw[arrows = {-Latex[width=0pt 10, length=10pt]}]    (232.2,255) -- (185.23,216.8) ;
\draw[arrows = {-Latex[width=0pt 10, length=10pt]}]    (232.2,255) -- (160,147.53) ;
\draw[arrows = {-Latex[width=0pt 10, length=10pt]}]    (232.2,255) -- (267.43,160.45) ;

\draw (202,262) node [anchor=north west][inner sep=0.75pt]   [align=left] {$\displaystyle \del_{v}( \Sigma \times I)$};
\draw (292,90.2) node [anchor=north west][inner sep=0.75pt]   [align=left] {$\displaystyle \del_{h}( \Sigma \times I)$};

\end{tikzpicture}}
  \vspace*{10pt}
  \caption{The horizontal and vertical boundaries of the product of a disc with two punctures, $\Sigma$, with the interval, $I$. The fibration of the visible part of the vertical boundary $\del\Sigma\times I$ is striped. One component of the horizontal boundary is visible.}
  \label{fig:horizboundary}
\end{figure}

\begin{defn}
A \emph{\ths{}} $\mathcal{H}$ for a compact orientable 3-manifold $M$ is a partition of $M$ into:
\begin{enumerate}
    \item $k$-handles for $k$ between 0 and 3, where each $k$-handle has a homeomorphism to $D^k\times D^{3-k}$, and
    \item parallelity pieces, each with a homeomorphism to $\Sigma\pmttimes I$ for $\Sigma$ a compact surface
\end{enumerate}
and with a distinguished \emph{forbidden region} $\mathcal{I} \sub \del M$ such that the following conditions hold.
Write $\mathcal{H}^k$ for the collection of $k$-handles, $\mathcal{H^P}$ for the collection of parallelity pieces, and $\del_h \mathcal{H^P}$ and $\del_v \mathcal{H^P}$ respectively for the collection of the horizontal and vertical boundaries of the parallelity pieces.
The \emph{boundary graph} of a 0-handle $H$ in a \ths{} is the decorated graph in $\del H \cong S^2$ whose vertices, which we call \emph{islands}, are the components of $H\cap \mathcal{H}^{1}$; whose edges (which we call \emph{bridges}) are the components of $H\cap \mathcal{H}^{2}$ and $H\cap \mathcal{H^P}$; and which may have \emph{sutures}, which are the arcs of $H\cap \del\mathcal{I}-(\mathcal{H}^1\cup \mathcal{H}^2\cup \mathcal{H^P})$.
We say that the boundary graph divides $\del H$ into islands, bridges, \emph{lakes} (components of intersection between $\del H$ and $(\mathcal{H}^{3}\cup \del M) - \mathcal{I}$), and \emph{forbidden regions} which are components of $H\cap \mathcal{I}$; the sutures are the intersections between the forbidden regions and the lakes. 

We require that:
\begin{enumerate}
    \item each $k$-handle $D^k\times D^{3-k}$ intersects handles of lower index in exactly $\del D^k\times D^{3-k}$, and is disjoint from the other $k$-handles;
    \item the boundary graph of each 0-handle is connected;
    \item each parallelity piece is disjoint from the 2- and 3-handles and the other parallelity pieces;
    \item the forbidden region $\mathcal{I}\sub \del M$ contains $\del_h \mathcal{H^P}$;
    \item each 1-handle $D^1\times D^2$ intersects 2-handles $D^2\times D^1$ in components that are of the form $D^1\times \gamma$ in the 1-handle and $\lambda \times D^1$ in the 2-handle, where $\gamma$ and $\lambda$ collections of arcs in $\del D^2$ in the respective product structures;
    \item the intersection of any component $P\cong \Sigma\pmttimes I$ of ${\mathcal{H^P}}$ with a 1-handle $D^1\times D^2$ is as $D^1\times \gamma$ in the 1-handle and $\lambda \times I$ in the parallelity piece, where $\gamma$ is a collection of arcs in $\del D^2$ and $\lambda$ is a collection of arcs in $\del\Sigma$.
\end{enumerate}

\noindent If $\mathcal{I}$ is not empty, we require the following.
Write $(\del\mathcal{H})^{k}$ for each $k$ for the components of $\del M \cap \mathcal{H}^{k}$, and $(\del\mathcal{H})^{\mathcal{P}}$ for the components of intersection of $\del M$ with $\mathcal{H^P}$.
Note that $(\del\mathcal{H})^{0}$ and $(\del\mathcal{H})^{2}$ are collections of discs, and $(\del\mathcal{H})^{1}$ and $(\del\mathcal{H})^{\mathcal{P}}$ are collections of discs and possibly some annuli.
We require that $\del \mathcal{I}$ avoids $(\del\mathcal{H})^{2}$, runs through discs of $(\del\mathcal{H})^{0}$ and $(\del\mathcal{H})^{1}$ in arcs that each do not start and end on the same component of $(\del\mathcal{H})^{0}\cap (\del\mathcal{H})^{1}$, and intersects each component of $\del_v \mathcal{H^P}\cap \del M$ in exactly two arcs or curves, each of which is transverse to the $I$-bundle structure from the parallelity pieces.
\end{defn}

If the forbidden region is empty (which implies that there are no parallelity pieces), this is the usual notion of handle structure.

\begin{defn}
    A surface in a \ths{} is \emph{$\del$-compressible} if it admits a non-trivial $\del$-compression disc that is disjoint from the forbidden region.
\end{defn}

Whenever we refer to a $\del$-compression disc, we require that the disc is disjoint from the forbidden region.

\subsection{Normal surfaces}

Normal surfaces in split handle structures generalise the standard definition in handle structures (see Definition 3.4.1 of~\cite{Matveev}), which itself generalises the definition in the triangulation setting (see Section~\ref{section:normalsurface}).

\begin{defn}
\label{defn:normalsurface}
    A properly-embedded surface in $M$ is \emph{normal} with respect to a \ths{} if it satisfies the following conditions:
    \begin{enumerate}
        \item \label{defn:normalsurface:itm:3handles} it is disjoint from the 3-handles and from the forbidden region;
        \item \label{defn:normalsurface:itm:2handles} it is transverse to the $I$-bundle structure of the 2-handles $D^2\times I$ and the parallelity pieces $\Sigma\pmttimes I$, and is disjoint from their horizontal boundaries;
        \item \label{defn:normalsurface:itm:parallelitycpt} no component of it is contained in a parallelity piece;
        \item \label{defn:normalsurface:itm:1handles} $F$ intersects each 1-handle $D^1\times D^2$ in $D^1\times \lambda$ where $\lambda$ is a collection of disjoint proper arcs in the island $\{0\}\times D^2$, such that no component of $\lambda$ starts and ends on the same connected component of the intersection of the island with a lake;
        \item \label{defn:normalsurface:itm:0handles} $F$ intersects each 0-handle in discs, called \emph{elementary discs}, such that the boundary curve of each of these discs crosses each bridge and lake at most once, and if a bridge and a lake are adjacent, intersects only one of the pair.
        \item \label{defn:normalsurface:itm:lakes} the intersection of $F$ with each lake does not contain any closed curves or arcs that start and end on the same component of the intersection of the lake with an island.
    \end{enumerate}
\end{defn}

It follows from the definition that if $F$ is a normal surface, then it intersects each 2-handle $D^2\times D^1$ in sheets of the form $D^2\times\{*\}$, and similarly intersects each parallelity piece in sheets that are a section of the $I$-bundle (that is, isotopic to $\Sigma\times\{0\}$) or, if $\Sigma$ is nonorientable, the double cover of a section.
Note that the boundary of an elementary disc of $F$ in the boundary graph of a 0-handle $H$ determines the disc.
We encourage the reader to satisfy themselves that the disc boundaries shown in the split 0-handles in Figures~\ref{fig:splithandleexample} and~\ref{fig:splithandleexamplenottriangulationdual} are in fact those of elementary discs, and that the resulting boundary graphs are possible boundary graphs of 0-handles in a split handle structure.
Figures~\ref{fig:elementarynonex} shows some non-examples of elementary discs, and some further examples are shown later in this section when illustrating the induced \ths{} construction in Figure~\ref{fig:elementaryex}.

\begin{figure}[th]
  \centering
  \resizebox{0.6\textwidth}{!}{
\begin{tikzpicture}[x=0.75pt,y=0.75pt,yscale=-1,xscale=1]
\node[circle, draw=black, line width=0.75pt, minimum size=18pt, fill=white] (1) at (79.15,112.17) {};
\node[circle, draw=black, line width=0.75pt, minimum size=18pt, fill=white] (2) at (177.81,110.83) {};
\node[circle, draw=black, line width=0.75pt, minimum size=18pt, fill=white] (3) at (271,109.5) {};
\node[circle, draw=black, line width=0.75pt, minimum size=18pt, fill=white] (4) at (226,193.5) {};
\draw[line width=0.75pt, double, double distance=4pt, rounded corners = 4pt] (1) -- (2) -- (3) -- (4) -- (2);

\coordinate (p) at ($(2)!0.5!(4) + (-30,10)$);
\draw  [color=teal,draw opacity=1, line width=2pt] (79.17,112.17) .. controls (79.17,101.12) and (100.17,92.17) .. (126.08,92.17) .. controls (151.99,92.17) and (173,101.12) .. (173,112.17) .. controls (173,123.21) and (151.99,132.17) .. (126.08,132.17) .. controls (100.17,132.17) and (79.17,123.21) .. (79.17,112.17) -- cycle ;
\draw[line width=2pt, color=teal, rounded corners=8pt] (2.center) -- (3.center) -- (4.center) -- (p) -- cycle;

\end{tikzpicture}
  }
  \vspace*{10pt}
  \caption{Two non-examples of elementary disc boundaries in a split 0-handle. The disc boundaries are drawn in teal. The one on the left crosses the same lake in two arcs, while the one on the right intersects a bridge and an adjacent lake.
  Both therefore fail condition~\ref{defn:normalsurface:itm:0handles} of Definition~\ref{defn:normalsurface}.}
  \label{fig:elementarynonex}
\end{figure}

\begin{defn}
    An \emph{admissible isotopy} of a surface with respect to a \ths{} is an isotopy of the surface in the manifold that fixes the forbidden region $\mathcal{I}$ as a set.
    A \emph{normal isotopy} of a normal surface in a \ths{} is an isotopy of the surface in the manifold that fixes each of $\mathcal{I}$, $\mathcal{H}^{k}$ for each $k$, and $\mathcal{H^P}$ as a set.
\end{defn}

\begin{defn}
\label{defn:duplicate}
   A normal surface in a \ths{} with forbidden region $\mathcal{I}$ is \emph{\pns{}} if it has two components that are normally isotopic, or it has a component $F$ that is normally isotopic to one component of the horizontal boundary of a small collar of a component $I$ of the forbidden region.
\end{defn}

We give the normalisation procedure in the appendix (Procedure~\ref{procedure:normalisation}) and show that it terminates in Proposition~\ref{prop:normalisationterminates}.
We then prove the following.

\begin{restatable*}{prop}{cannormaliseincompressiblesurfaces}
\label{prop:cannormaliseincompressiblesurfaces}
    Let $F$ be an incompressible $\del$-incompressible properly-embedded surface in an irreducible $\del$-irreducible manifold $M$ with \ths{} $\mathcal{H}$, disjoint from the forbidden region, such that no component is a trivial sphere or disc or is entirely contained in a parallelity piece.
    Then $F$ is admissibly isotopic to a normal surface.
\end{restatable*}

\begin{defn}
\label{defn:inducedths}
Let $F$ be a normal surface in $M$, with respect to a \ths{} $\mathcal{H}$.
The \emph{induced \ths{}} on $M\backslash\backslash F$ is constructed as follows.

Consider $\mathcal{H}\backslash\backslash F$.
Set its forbidden region $\mathcal{I}$ to be the union of the forbidden region from $M$ and the image of $F$ in $M\backslash\backslash F$. (As $F$ is normal, these are disjoint.)
Since $F$ is disjoint from the 3-handles, we can continue to view them as 3-handles in $\mathcal{H}\backslash\backslash F$.

A component of a $k$-handle in $\mathcal{H}\backslash\backslash F$ will either become a $k$-handle or a part of a parallelity piece.
This is determined as follows.
First, if $P\cong \Sigma\pmttimes I$ is a parallelity piece of $\mathcal{H}$, as $F$ intersects $P$ in sheets transverse to the $I$-bundle structure, each component of $P\backslash\backslash F$ inherits an $I$-bundle structure as either $\Sigma\pmttimes I$ or possibly, if $\Sigma$ is not orientable, as $\tilde{\Sigma}\times I$, where $\tilde{\Sigma}$ is the double cover of $\Sigma$.
Thus we can view each component of $P\backslash\backslash F$ as a parallelity piece.

If $H\cong D^2\times D^1$ is a 2-handle of $\mathcal{H}$, a component $C$ of $H\backslash\backslash F$ is itself an $I$-bundle over $D^2$, with two horizontal boundary components, each of which arises from intersection with a 3-handle, with $\mathcal{I}$, or with $\del M - \mathcal{I}$.
If both components are in $\mathcal{I}$, set $C$ to be a parallelity piece; otherwise, view $C$ as a 2-handle.

If $H\cong D^1\times D^2$ is a 1-handle of $\mathcal{H}$, a component $C$ of $H\backslash\backslash F$ is a parallelity piece if its boundary consists of the following components: first, two components in the forbidden region $\mathcal{I}$ (whether from the forbidden region in $M$ or arising from intersection with $F$); second, two components arising from intersection with pieces of 0-handles from $M$; and finally, the two remaining components where each is a component of intersection with one of $\del M$, a single 2-handle, or a single parallelity piece.
In this case, $C$ has an $I$-bundle structure by setting $\del_h C$ to be the two components in the forbidden region, and choosing a product structure on the remaining boundary, $\del_v C$, such that each of the four components described above is a union of fibers, and then interpolating.
(We can choose this product structure to be compatible with the product structure on any parallelity pieces defined thus far that $C$ intersects).

Finally, if $H\cong D^0\times D^3$ is a 0-handle and $C$ is a component of it, consider the boundary of $C$.
Set $C$ to be a parallelity piece if its boundary contains two components of intersection with $\mathcal{I}$ and if there is a product structure $D^2\times I$ on $C$ such that its intersection with $\mathcal{I}$ is $D^2\times \del I$, and each component of its intersection with any of the other handles created so far is of the form $\alpha \times I$, where $\alpha$ is an arc or curve in $\del D^2$.
Again, note that we can choose this product structure to be compatible with any parallelity pieces that $C$ intersects.
Otherwise, set $C$ to be a 0-handle.

Now, take the parallelity pieces of $\mathcal{H}$ to be the union of the parallelity pieces described so far, which we equipped with compatible product structures where they intersected.

We write $\mathcal{H}\backslash\backslash F$ for the induced \ths{} on $M\backslash\backslash F$.
\end{defn}

Figures~\ref{fig:splithandleexample} and~\ref{fig:splithandleexamplenottriangulationdual} both give examples of an induced split 0-handle from cutting along an elementary disc.
Some more examples are shown in Figure~\ref{fig:elementaryex}, where one of the three components produced after cutting along the discs in Figure~\ref{fig:elementaryex:3} is a parallelity handle, as it lies between a forbidden region and a parallel elementary disc.

\begin{lemma}
    If $F$ is a normal surface in a \ths{} $\mathcal{H}$, then the induced \ths{} $\mathcal{H}\backslash\backslash F$ is (as the name suggests) a \ths{}.
   \label{lemma:normalsurfacegivesths}
\end{lemma}

\begin{proof}
    If a piece from a $k$-handle in $\mathcal{H}\backslash\backslash F$ becomes a parallelity piece in the induced \ths{}, then any adjacent pieces from handles of higher index will also be parallelity pieces.
    As a result, each $k$-handle $D^k\times D^{3-k}$ will intersect handles of lower index in $\del D^k\times D^{3-k}$ and will be disjoint from the other $k$-handles.
    The boundary graph of each 0-handle will be connected as the boundary graph of each 0-handle arises from taking a connected boundary graph, removing some discs from it, and adding in the boundary of these discs (as we add a new suture for each time the boundary of one of the discs runs through a lake).
    As $F$ does not intersect any 3-handles, so each one becomes a 3-handle in $\mathcal{H}\backslash\backslash F$, each parallelity piece will be disjoint from the 2- and 3-handles.
    As intersections of parallelity pieces with 1-handles must necessarily arise from intersections between a 2-handle and a 1-handle, and the product structure of the 2-handle is compatible with that of the parallelity piece, the intersections between parallelity pieces and 1-handles are of the required form.
    The intersection of $\del \mathcal{I}$ with the boundary has the required form as $F$ is normal.
\end{proof}

\begin{figure}[th]
  \centering
  \begin{subfigure}[b]{0.3\textwidth}
  	\centering
  \resizebox{\textwidth}{!}{
\begin{tikzpicture}
\node[circle, draw=black, line width=0.75pt, minimum size=18pt, fill=white] (2) at (-1,-1) {};
\node[circle, draw=black, line width=0.75pt, minimum size=18pt, fill=white] (3) at (1,-1) {};
\node[circle, draw=black, line width=0.75pt, minimum size=18pt, fill=white] (4) at (0,0.5) {};
\draw[line width=0.75pt, double, double distance=4pt, rounded corners = 4pt] (2) -- (3) -- (4) -- (2);

\coordinate (a) at ($(4)!0.5!(3) + (0.4,0.1)$);
\coordinate (b) at ($(4)!0.5!(3) + (-0.4,-0.1)$);
\draw[line width=2pt, color=teal, rounded corners=2pt] (4.center) -- ($(4.center) + (0.15,-0.05)$) .. controls (a) .. ($(3.center) + (0.15,0.05)$) -- (3.center) -- ($(3.center) + (-0.08,0.05)$).. controls (b) .. ($(4.center) + (-0.08,-0.05)$) -- cycle;

\end{tikzpicture}
  }
  \caption{An elementary disc boundary (in teal) in the boundary graph of a split 0-handle.}
  \label{fig:elementaryex:1}
  \end{subfigure}
  \begin{subfigure}[b]{0.58\textwidth}
  	\centering
   \resizebox{0.55\textwidth}{!}{
    \begin{tikzpicture}
            \coordinate (1) at (0,0.5);
            \coordinate (2) at (1,-1);
            \coordinate (3) at (-1, -1);
            \coordinate (a) at (0.4, -0.5);
            \coordinate (b) at (0.9, 0);
            \coordinate (c) at ($(1.east)+(0.1,0)$);
            \coordinate (f) at ($(1.east)+(0.2,0)$);
            \coordinate (d) at ($(2.east)+(0.1,0)$);
            \coordinate (e) at ($(2.east)+(0.2,0)$);
            \filldraw[color = red, fill opacity=0.4, rounded corners=1pt, draw=red, line width=2pt] (1.center) -- (c) .. controls (a) .. (d) -- (e) .. controls (b) .. (f) -- (c);
            \draw[line width=0.75pt, double, double distance=4pt] (1) -- (3) -- (2);
            \node[circle, draw=black, line width=0.75pt, minimum size=16pt, fill=white] at (1) {};
            \node[circle, draw=black, line width=0.75pt, minimum size=16pt, fill=white] at (2) {};
            \node[circle, draw=black, line width=0.75pt, minimum size=16pt, fill=white] at (3) {};
        \end{tikzpicture}
  }
\resizebox{0.25\textwidth}{!}{
    \begin{tikzpicture}
            \coordinate (1) at (0,0.95);
            \coordinate (2) at (0,-0.95);
            \coordinate (a) at (0.4, 0);
            \coordinate (b) at (0.9, 0);
            \coordinate (c) at ($(1.east)+(0.1,0)$);
            \coordinate (f) at ($(1.east)+(0.2,0)$);
            \coordinate (d) at ($(2.east)+(0.1,0)$);
            \coordinate (e) at ($(2.east)+(0.2,0)$);
            \filldraw[color = red, fill opacity=0.4, rounded corners=1pt, draw=red, line width=2pt] (1.center) -- (c) .. controls (a) .. (d) -- (e) .. controls (b) .. (f) -- (c);
            \draw[line width=0.75pt, double, double distance=4pt] (1) -- (2);
            \node[circle, draw=black, line width=0.75pt, minimum size=16pt, fill=white] at (1) {};
            \node[circle, draw=black, line width=0.75pt, minimum size=16pt, fill=white] at (2) {};
        \end{tikzpicture}
  }
  \caption{The two resulting split 0-handles from cutting along the disc in Figure~\ref{fig:elementaryex:1}. The forbidden region is shaded in red and the sutures are thick red lines.}
  \label{fig:elementaryex:2}
  \end{subfigure}
  \begin{subfigure}[b]{0.3\textwidth}
  	\centering
   \resizebox{\textwidth}{!}{
    \begin{tikzpicture}
            \coordinate (1) at (0,0.5);
            \coordinate (2) at (1,-1);
            \coordinate (3) at (-1, -1);
            \coordinate (a) at (0.4, -0.5);
            \coordinate (b) at (0.9, 0);
            \coordinate (c) at ($(1.east)+(0.1,0)$);
            \coordinate (f) at ($(1.east)+(0.2,0)$);
            \coordinate (d) at ($(2.east)+(0.1,0)$);
            \coordinate (e) at ($(2.east)+(0.2,0)$);
            \filldraw[color = red, fill opacity=0.4, rounded corners=1pt, draw=red, line width=2pt] (1.center) -- (c) .. controls (a) .. (d) -- (e) .. controls (b) .. (f) -- (c);
            \draw[line width=0.75pt, double, double distance=4pt] (1) -- (3) -- (2);
            \node[circle, draw=black, line width=0.75pt, minimum size=16pt, fill=white] at (1) {};
            \node[circle, draw=black, line width=0.75pt, minimum size=16pt, fill=white] at (2) {};
            \node[circle, draw=black, line width=0.75pt, minimum size=16pt, fill=white] at (3) {};

            \coordinate (a) at ($(1)!0.5!(3) + (0.5,-0.1)$);
            \coordinate (b) at ($(1)!0.5!(3) + (-0.5,0.2)$);
            \draw[line width=2pt, color=teal, rounded corners=2pt] ($(1.center) + (-0.05,-0.05)$)-- ($(1.center) + (0,-0.15)$) .. controls (a) .. ($(3.center) + (0.15,0.05)$) -- (3.center) -- ($(3.center) + (-0.08,0.05)$).. controls (b) .. ($(1.center) + (-0.2,-0.05)$) -- cycle;
            \draw[line width=2pt, color=teal, rounded corners=3pt] ($(1.center)+(0.15,0.1)$) .. controls (1.2,0.2) .. ($(2.east)+(0.25,0.3)$) -- ($(2.east)$) .. controls (0.2,-0.5) .. ($(1.center)+(0.08,-0.1)$) -- cycle;
        \end{tikzpicture}
  }
  \caption{Two elementary discs (drawn in teal) in the boundary of the split 0-handle on the left of Figure~\ref{fig:elementaryex:2}.}
  \label{fig:elementaryex:3}
  \end{subfigure}
  \begin{subfigure}[b]{0.58\textwidth}
  	\centering
   \resizebox{0.48\textwidth}{!}{
    \begin{tikzpicture}
            \coordinate (1) at (0,0.5);
            \coordinate (2) at (1,-1);
            \coordinate (3) at (-1, -1);
            \coordinate (a) at (0.4, -0.5);
            \coordinate (y) at (-0.4, -0.5);
            \coordinate (b) at (0.9, 0);
            \coordinate (v) at (-0.9, 0);
            \coordinate (c) at ($(1.east)+(0.1,0)$);
            \coordinate (z) at ($(1.west)+(-0.1,0)$);
            \coordinate (f) at ($(1.east)+(0.2,0)$);
            \coordinate (u) at ($(1.west)+(-0.2,0)$);
            \coordinate (d) at ($(2.east)+(0.1,0)$);
            \coordinate (x) at ($(3.west)+(-0.1,0)$);
            \coordinate (e) at ($(2.east)+(0.2,0)$);
            \coordinate (w) at ($(3.west)+(-0.2,0)$);
            \filldraw[color = red, fill opacity=0.4, rounded corners=1pt, draw=red, line width=2pt] (1.center) -- (c) .. controls (a) .. (d) -- (e) .. controls (b) .. (f) -- (c);
            \filldraw[color = red, fill opacity=0.4, rounded corners=1pt, draw=red, line width=2pt] (1.center) -- (z) .. controls (y) .. (x) -- (w) .. controls (v) .. (u) -- (z);
            \draw[line width=0.75pt, double, double distance=4pt] (3) -- (2);
            \node[circle, draw=black, line width=0.75pt, minimum size=16pt, fill=white] at (1) {};
            \node[circle, draw=black, line width=0.75pt, minimum size=16pt, fill=white] at (2) {};
            \node[circle, draw=black, line width=0.75pt, minimum size=16pt, fill=white] at (3) {};
        \end{tikzpicture}
  }
\resizebox{0.49\textwidth}{!}{
    \begin{tikzpicture}
            \coordinate (1) at (0,0.95);
            \coordinate (2) at (0,-0.95);
            \coordinate (a) at (0.4, 0);
            \coordinate (b) at (0.9, 0);
            \coordinate (c) at ($(1.east)+(0.1,0)$);
            \coordinate (f) at ($(1.east)+(0.2,0)$);
            \coordinate (d) at ($(2.east)+(0.1,0)$);
            \coordinate (e) at ($(2.east)+(0.2,0)$);
            \filldraw[color = red, fill opacity=0.4, rounded corners=1pt, draw=red, line width=2pt] (1.center) -- (c) .. controls (a) .. (d) -- (e) .. controls (b) .. (f) -- (c);
            \draw[line width=0.75pt, double, double distance=4pt] (1) -- (2);
            \node[circle, draw=black, line width=0.75pt, minimum size=16pt, fill=white] at (1) {};
            \node[circle, draw=black, line width=0.75pt, minimum size=16pt, fill=white] at (2) {};
            \node at (2,0) {$D^2\times I$};
        \end{tikzpicture}
  }
  \caption{The three resulting split handles from cutting along the disc in Figure~\ref{fig:elementaryex:3}. The forbidden region is shaded in red and the sutures are thick red lines. The first two handles are split 0-handles. The third (not drawn) is a parallelity handle $D^2\times I$ in the induced split handle structure.}
  \label{fig:elementaryex:4}
  \end{subfigure}
  \caption{The boundary graphs of the induced split handles from cutting along a sequence of elementary discs in a 0-handle.}
  \label{fig:elementaryex}
\end{figure}

We wish to give an abstract characterisation of the sort of split handle structures that arise from taking a triangulation of a 3-manifold, considering its dual handle structure, and cutting this along a normal surface.
At the dual handle structure stage, the combinatorial 0-handles we can obtain are well-understood.
If such a 0-handle is in the interior of the manifold, its boundary graph will be the complete graph on four vertices; if it intersects the boundary, its boundary graph will be a subgraph of this -- it will be subtetetrahedral in the language of Lackenby~\cite[\S6.2]{LackenbyCertificationKnottedness}.
We generalise these ideas to the split handle structure setting.

\begin{defn}
\label{defn:subtetrahedral}
    We call a 0-handle $H$ {semitetrahedral} if it is disjoint from the forbidden region and its boundary graph is a connected subgraph of the complete graph on four vertices.
    We call a \ths{} {semitetrahedral} if all of its 0-handles are semitetrahedral.
    
    A 0-handle $H$ is \emph{subtetrahedral} if there is some 0-handle $H'$ in some semitetrahedral handle structure $\mathcal{H}$ with normal surface $F$ such that $H$ has the same boundary graph as one of the non-parallelity pieces obtained from $H'$ in the induced \ths{} $\mathcal{H}\backslash\backslash F$.
    We say that a \ths{} is \emph{subtetrahedral} if all of its 0-handles are subtetrahedral.
\end{defn}

\begin{restatable}{lemma}{subtetrahedralths}
\label{lemma:subtetrahedralths}
Let $\mathcal{T}$ be a (material) triangulation of a (compact) 3-manifold $M$, such that the intersection of each tetrahedron with $\del M$ is connected and contractible.
The dual handle structure to $\mathcal{T}$ (see Definition~\ref{defn:dual_handle}) is a subtetrahedral \ths{} with empty forbidden region and no parallelity pieces.

If $\mathcal{H}$ is a subtetrahedral \ths{}, and $F$ is a normal surface in $\mathcal{H}$, then the induced \ths{} on $\mathcal{H}\backslash\backslash F$ is also subtetrahedral.
\end{restatable}

\begin{proof}
    The complement of the boundary graph of a 0-handle will deformation retract to the intersection of the corresponding tetrahedron with the boundary, so the handles of this dual handle structure have connected boundary graphs.
    The remainder follows by definition.
\end{proof}

As in the triangulation setting, we can view normal surfaces algebraically.
Let $\mathcal{H}$ be a {subtetrahedral} \ths{}.
Let $d_H$ be the maximal number of types of elementary disc in any subtetrahedral 0-handle, which is finite by Lemma~\ref{lemma:finiteelementarydisctypes}.
We can associate a vector in $\ZZ^{d_H\norm{\mathcal{H}}}$ to each normal surface $F$ in $\mathcal{H}$  by counting the number of elementary discs in $F$ of each type.

\begin{lemma}
    Let $\mathcal{H}$ be a subtetrahedral \ths{}.
    Given a vector $v$ in $\ZZ^{d_H\norm{\mathcal{H}}}$, there is at most one normal surface in $H$ corresponding to this vector.
    \label{lemma:vectorsgiveuniquenormalinths}
\end{lemma}

\begin{proof}
    Let $F$ be some normal surface corresponding to $v$.
    We wish to show that $F$ is uniquely determined (up to normal isotopy).
    As in the theory of normal surfaces in triangulations, since the boundary of each elementary disc in a 0-handle $H$ separates $\del H$, the vector $v$ determines $F$ in the 0-handles.
    We claim that the intersection of $F$ with the 0-handles determines its intersection with the 1- and 2-handles.
    First, the 1-handles: by Definition~\ref{defn:normalsurface}, a normal surface intersects any 1-handle $D^1\times D^2$ in $D^1\times \alpha$ where $\alpha$ is an arc in $D^2$.
    Thus the image of $F$ in this 1-handle is determined by its image in $\del D^1\times D^2$, which is exactly the intersection of the 1-handle with 0-handles.
    The 2-handles follow similarly: each 2-handle $D^2\times D^1$ intersects handles of lower index in $\del D^2\times D^1$.
    As each 1-handle intersects some 0-handle, each 2-handle intersects a 0-handle.
    Now, $F$ must run through each 2-handle in sheets, so to determine the image of $F$ in this 2-handle it suffices to count the number of sheets.
    Thus the number of components of $F$ in the intersection of the 2-handle with the 0-handle is sufficient information to pin $F$ down.
    
    The remaining handles through which $F$ may run are the parallelity pieces; again, it is enough to determine the number of sheets.
    By the same reasoning as in the 2-handle case, if a given parallelity piece intersects any 1-handle, it intersects a 0-handle, so the only pathological possibility is if a parallelity piece $P$ intersects no 0- or 1-handles at all.
    As parallelity pieces are disjoint from the 2- and 3-handles, $P$ is an entire connected component of $\mathcal{H}$.
    But then by condition~\ref{defn:normalsurface:itm:parallelitycpt} of Definition~\ref{defn:normalsurface}, $F$ is disjoint from $P$.
    Thus $F$ is uniquely determined by $v$.
\end{proof}

We can interpret summation in $\ZZ^{d_h\norm{\mathcal{H}}}$ geometrically as follows.

\begin{procedure}[Normal surface sum in \ths{}s]
\label{procedure:normalsurfacesum}
Suppose $G_1$ and $G_2$ are normal surfaces in a subtetrahedral \ths{} $\mathcal{H}$, such that if we view them as vectors in $\ZZ^{d_H\norm{\mathcal{H}}}$, $F = G_1 + G_2$ is also a vector representing a normal surface.
We can construct $F$ as follows.
Realise $G_1$ and $G_2$ so that they are disjoint in the 0-handles and transverse elsewhere.
In the 1-handles, use the product structure to ensure that any two components of $G_1\cap\mathcal{H}^{1}$ and $G_2\cap\mathcal{H}^{1}$ intersect in a single arc that is contained in some $\{*\}\times D^2$ in the 1-handle $D^1\times D^2$.
As we have two embedded surfaces, there are no triple points of intersection.
Now, in the 2-handles and parallelity pieces, we know that $G_1$ and $G_2$ intersect them in sheets transverse to the $I$-bundle which are determined by their boundaries.
Fix a 2-handle or parallelity piece $H\cong \Sigma\pmttimes I$.
So long as at most one of $G_1\cap H$ and $G_2\cap H$ contains a nonorientable piece, we can realise all the components of $G_1\cap H$ and $G_2\cap H$ as follows.
Let $H'$ be the complement in $H$ of a collar of the boundary of $H$.
Take the correct number of sheets of $G_1$ and $G_2$ in $H'$.
Then interpolate in the collar to achieve the required intersection pattern of $G_1$ and $G_2$ in $\del_v H$, such that $G_1\cap G_2\cap H$ contains no closed curves.
If both of them contain a nonorientable sheet in $H$, then we can realise the two nonorientable sheets to intersect in $H'$ in a single arc, and then again interpolate in the collar to achieve the required intersection pattern.
Now at each arc of intersection of $G_1$ and $G_2$ in $H$, if we cut along the arc then we have two choices of how to reglue to resolve the intersection: a choice of \emph{switch}.
This is determined by the picture at one of the arc's endpoints $p$ in $\del_v H$.
The \emph{regular switch}, which is the choice of gluing that produces sheets transverse to the product structure, produces two sheets of the same types as those we started with (unless both the sheets were nonorientable, in which case it produces one sheet that is their orientable double).
The \emph{irregular switch} is the other choice; the two are shown (at $p$) in Figure~\ref{fig:irregularregularswitch}.

\begin{figure}[th]
    \centering
    \begin{subfigure}{0.43\textwidth}
        \centering
        \resizebox{\textwidth}{!}{

\tikzset{every picture/.style={line width=0.75pt}} 

\begin{tikzpicture}[x=0.75pt,y=0.75pt,yscale=-1,xscale=1]

\draw  [draw opacity=0] (346.44,195.9) .. controls (333.16,207.78) and (296.12,216.44) .. (252.46,216.62) .. controls (206.86,216.81) and (168.39,207.69) .. (156.67,195.09) -- (252.33,186.62) -- cycle ; \draw   (346.44,195.9) .. controls (333.16,207.78) and (296.12,216.44) .. (252.46,216.62) .. controls (206.86,216.81) and (168.39,207.69) .. (156.67,195.09) ;  
\draw  [draw opacity=0] (343.06,142.97) .. controls (334.18,156.67) and (295.67,167.11) .. (249.45,167.31) .. controls (205.22,167.49) and (167.95,158.23) .. (156.96,145.48) -- (249.33,137.31) -- cycle ; \draw   (343.06,142.97) .. controls (334.18,156.67) and (295.67,167.11) .. (249.45,167.31) .. controls (205.22,167.49) and (167.95,158.23) .. (156.96,145.48) ;  
\draw[color=teal]    (159,162.5) .. controls (178.5,180.75) and (293,204.75) .. (342.5,190.25) ;
\draw[color=teal]    (161,186.75) .. controls (173.91,194.4) and (207.99,194.28) .. (243.97,189.55) .. controls (283.21,184.39) and (324.72,173.75) .. (343.5,161.75) ;
 \node[fill,circle,label=below:$p$, scale=0.7] at (244,190) {};

\end{tikzpicture}
        }
        \caption{A point $p$ of intersection of $G_1$ and $G_2$ in the boundary of a 2-handle or parallelity piece.}
    \end{subfigure}
    \begin{subfigure}{0.43\textwidth}
        \centering
        \resizebox{\textwidth}{!}{
\tikzset{every picture/.style={line width=0.75pt}} 

\begin{tikzpicture}[x=0.75pt,y=0.75pt,yscale=-1,xscale=1]

\draw  [draw opacity=0] (346.44,195.9) .. controls (333.16,207.78) and (296.12,216.44) .. (252.46,216.62) .. controls (206.86,216.81) and (168.39,207.69) .. (156.67,195.09) -- (252.33,186.62) -- cycle ; \draw   (346.44,195.9) .. controls (333.16,207.78) and (296.12,216.44) .. (252.46,216.62) .. controls (206.86,216.81) and (168.39,207.69) .. (156.67,195.09) ;  
\draw  [draw opacity=0] (343.06,142.97) .. controls (334.18,156.67) and (295.67,167.11) .. (249.45,167.31) .. controls (205.22,167.49) and (167.95,158.23) .. (156.96,145.48) -- (249.33,137.31) -- cycle ; \draw   (343.06,142.97) .. controls (334.18,156.67) and (295.67,167.11) .. (249.45,167.31) .. controls (205.22,167.49) and (167.95,158.23) .. (156.96,145.48) ;  
\draw[color=teal]    (159,193.25) .. controls (165.16,194.24) and (205.81,186.16) .. (247.5,187.25) .. controls (289.19,188.34) and (330,194.75) .. (343,193.75) ;
\draw[color=teal]    (159,162.5) .. controls (174.5,171.75) and (218.46,181) .. (245.5,180.75) .. controls (272.54,180.5) and (324.32,173.05) .. (344.38,161.82) ;
\end{tikzpicture}
        }
        \caption{The regular switch resolution of the intersection.}
    \end{subfigure}
    \begin{subfigure}{0.43\textwidth}
        \centering
        \resizebox{\textwidth}{!}{
        \tikzset{every picture/.style={line width=0.75pt}} 

\begin{tikzpicture}[x=0.75pt,y=0.75pt,yscale=-1,xscale=1]

\draw  [draw opacity=0] (346.44,195.9) .. controls (333.16,207.78) and (296.12,216.44) .. (252.46,216.62) .. controls (206.86,216.81) and (168.39,207.69) .. (156.67,195.09) -- (252.33,186.62) -- cycle ; \draw   (346.44,195.9) .. controls (333.16,207.78) and (296.12,216.44) .. (252.46,216.62) .. controls (206.86,216.81) and (168.39,207.69) .. (156.67,195.09) ;  
\draw  [draw opacity=0] (343.06,142.97) .. controls (334.18,156.67) and (295.67,167.11) .. (249.45,167.31) .. controls (205.22,167.49) and (167.95,158.23) .. (156.96,145.48) -- (249.33,137.31) -- cycle ; \draw   (343.06,142.97) .. controls (334.18,156.67) and (295.67,167.11) .. (249.45,167.31) .. controls (205.22,167.49) and (167.95,158.23) .. (156.96,145.48) ;  
\draw[color=teal]   (159,162.5) .. controls (173,173.75) and (241,173.25) .. (241,188.25) .. controls (241,203.25) and (184,197.75) .. (161.88,186.82) ;
\draw[color=teal]    (343,188.25) .. controls (325,201.75) and (255,202.75) .. (255,188.25) .. controls (255,173.75) and (327,176.75) .. (344.38,161.82) ;
\end{tikzpicture}
        }
        \caption{The irregular switch resolution of the intersection.}
    \end{subfigure}
    \caption{The regular and irregular switches at an intersection curve of normal surfaces $G_1$ and $G_2$.}
    \label{fig:irregularregularswitch}
\end{figure}

We wish to show that choosing all regular switches is compatible.
As the product structures of the 1-handles and the 2-handles or parallelity pieces are transverse, if we consider the arc of intersection in the 1-handle containing $p$, the choice of regular switch at $p$ induces the regular switch on the other end of the arc.
Similarly, as the regular switch is the choice giving sheets transverse to the product structure on the 2-handle or parallelity piece, if we consider the arc in $H$ containing $p$, the choice of regular switch at $p$ induces the regular switch along the whole arc.
Thus there is a global \emph{regular switch} for $G_1\cup G_2$, which produces $F$.
\end{procedure}

We will need to measure the complexity (the \emph{weight}) of a normal surface in a \ths{}.
We will often argue by minimising weight to show that a small normal surface with desired properties exists.
Thus far, when working in the triangulation setting, we have been using edge weight, which is the number of points of intersection of a surface $F$ with the edges of the triangulation.
In the dual handle structure, the corresponding number is the \emph{plate degree}, $p(F)$.
We will need a lexicographically-ordered notion of weight for some of our inductive arguments, which will also incorporate the \emph{beam degree}, $b(F)$.
These notions of complexity are used in normal surface theory in standard handle structures; for example, see the proof of Theorem 3.4.7 in~\cite{Matveev}.

\begin{defn}
    \label{defn:weight}
    The weight of a normal surface $F$ is $(p(F), b(F), |F\cap \del M|)$, where the \emph{plate degree}, $p(F)$, is $|F\cap (\del\mathcal{H}^{2}\cup \del\mathcal{H^P})|$ and the \emph{beam degree}, $b(F)$, is $|F\cap \mathcal{H}^{1}|$.
\end{defn}

\begin{lemma}
    Let $G_1$ and $G_2$ be normal surfaces such that the vector $F = G_1 + G_2$ also corresponds to a normal surface.
    Let $F'$ be an incompressible and $\del$-incompressible surface constructed by resolving the curves of intersection of $G_1\cup G_2$ that includes at least one irregular switch.
    Then $F'$ is isotopic to a normal surface of lower weight than $F$.
\label{lemma:irregularswitchreducestrace}
\end{lemma}

\begin{proof}
Take $G_1$ and $G_2$ to minimise $|G_1\cap G_2|$ in their normal isotopy class.
The plate degree $p(F)$ is $p(G_1) + p(G_2)$ as the intersection of $F$ with the boundaries of the 0-handles is the union of the intersections of $G_1$ and $G_2$ with them.
Suppose we take the irregular switch at some point $p$ in the boundary of a 2-handle or parallelity piece $H$.
Write $F'$ for the result of this irregular switch at $p$.
Let $c$ be the number of essential curves of $G_1$ and $G_2$ in the component of $\del_v H$ containing $p$.

If we take the irregular switch at $p$, then the annular component of $\del_v H$ containing $p$ has a spanning arc that intersects the surface $F'$ at most $c-2$ times (see Figure~\ref{fig:irregularswitchreducesweight}).
Thus $F'$ intersects $\del_v H$ in at most $c-2$ essential curves.
If there are further curves of intersection, they are trivial in $\del_v H$, so when we normalise $F'$ using Procedure~\ref{procedure:normalisation}, in Move 2 we will compress $F'$ and thus remove these trivial curves of intersection altogether.
As $F'$ is incompressible and $\del$-incompressible, following normalisation there is a component of the resulting surface that is isotopic to $F'$.
The plate degree of this representative of $F'$ is lower than that of $F$, so we have reduced the weight.
\end{proof}

\begin{figure}[th]
    \centering
    \begin{subfigure}{0.43\textwidth}
        \centering
        \resizebox{\textwidth}{!}{

\tikzset{every picture/.style={line width=0.75pt}} 

\begin{tikzpicture}[x=0.75pt,y=0.75pt,yscale=-1,xscale=1]

\draw  [draw opacity=0] (346.44,195.9) .. controls (333.16,207.78) and (296.12,216.44) .. (252.46,216.62) .. controls (206.86,216.81) and (168.39,207.69) .. (156.67,195.09) -- (252.33,186.62) -- cycle ; \draw   (346.44,195.9) .. controls (333.16,207.78) and (296.12,216.44) .. (252.46,216.62) .. controls (206.86,216.81) and (168.39,207.69) .. (156.67,195.09) ;  
\draw  [draw opacity=0] (343.06,142.97) .. controls (334.18,156.67) and (295.67,167.11) .. (249.45,167.31) .. controls (205.22,167.49) and (167.95,158.23) .. (156.96,145.48) -- (249.33,137.31) -- cycle ; \draw   (343.06,142.97) .. controls (334.18,156.67) and (295.67,167.11) .. (249.45,167.31) .. controls (205.22,167.49) and (167.95,158.23) .. (156.96,145.48) ;  
\draw[color=teal]    (159,162.5) .. controls (178.5,180.75) and (293,204.75) .. (342.5,190.25) ;
\draw[color=teal]    (161,186.75) .. controls (173.91,194.4) and (207.99,194.28) .. (243.97,189.55) .. controls (283.21,184.39) and (324.72,173.75) .. (343.5,161.75) ;
 \node[fill,circle,label=below:$p$, scale=0.7] at (244,190) {};
\draw[dashed, line width=1.5pt]    (205.5,163.75) -- (206,213.75) ;

\end{tikzpicture}
        }
        \caption{The dashed spanning arc $\alpha$ minimally intersects $G_1\cup G_2$ $c$ times (away from double points), where $c=2$ in this case.}
    \end{subfigure}
    \begin{subfigure}{0.43\textwidth}
        \centering
        \resizebox{\textwidth}{!}{
        \tikzset{every picture/.style={line width=0.75pt}} 

\begin{tikzpicture}[x=0.75pt,y=0.75pt,yscale=-1,xscale=1]

\draw  [draw opacity=0] (346.44,195.9) .. controls (333.16,207.78) and (296.12,216.44) .. (252.46,216.62) .. controls (206.86,216.81) and (168.39,207.69) .. (156.67,195.09) -- (252.33,186.62) -- cycle ; \draw   (346.44,195.9) .. controls (333.16,207.78) and (296.12,216.44) .. (252.46,216.62) .. controls (206.86,216.81) and (168.39,207.69) .. (156.67,195.09) ;  
\draw  [draw opacity=0] (343.06,142.97) .. controls (334.18,156.67) and (295.67,167.11) .. (249.45,167.31) .. controls (205.22,167.49) and (167.95,158.23) .. (156.96,145.48) -- (249.33,137.31) -- cycle ; \draw   (343.06,142.97) .. controls (334.18,156.67) and (295.67,167.11) .. (249.45,167.31) .. controls (205.22,167.49) and (167.95,158.23) .. (156.96,145.48) ;  
\draw[color=teal]   (159,162.5) .. controls (173,173.75) and (241,173.25) .. (241,188.25) .. controls (241,203.25) and (184,197.75) .. (161.88,186.82) ;
\draw[color=teal]    (343,188.25) .. controls (325,201.75) and (255,202.75) .. (255,188.25) .. controls (255,173.75) and (327,176.75) .. (344.38,161.82) ;
\draw[dashed, line width=1.5pt] (247.47,166.8) -- (247.97,216.8) ;
\end{tikzpicture}
        }
        \caption{After resolving via an irregular switch, the minimum intersection number of $\alpha$ and the new surface is $c-2$ (which is 0 in this case).}
    \end{subfigure}
    \caption{Performing an irregular switch on $G_1\cup G_2$ at $p$ produces a surface $F'$ with $p(F') \leq p(G_1) + p(G_2) -2$.}
    \label{fig:irregularswitchreducesweight}
\end{figure}

\begin{lemma}
   A normal surface $F$ in a subtetrahedral \ths{} has non-zero beam degree $b(F) = |F\cap \mathcal{H}^{1}|$.
   \label{lemma:nonzerobeamdegree}
\end{lemma}

\begin{proof}
As no component of a normal surface is contained in a parallelity piece, every normal surface contains at least one elementary disc.
Since the boundary graph of each 0-handle is connected and contains at least one island, every curve contained in a lake is trivial and each bridge starts and ends on an island.
The boundary of any elementary disc runs through a bridge or lake and thus through some island.
\end{proof}

\subsection{Fundamental normal surfaces}
\label{appendix:fundamentalsurfaces}

Fix a subtetrahedral \ths{} $\mathcal{H}$ of a manifold $M$ and recall that we can associate a vector in $\ZZ^{d_H\norm{\mathcal{H}}}$ to each normal surface in $\mathcal{H}$, where by Remark~\ref{remark:boundondH} $d_H$ is
at most $13\cdot 13!$.
A vector $v$ in $\ZZ^{d_H\norm{\mathcal{H}}}$ corresponds to a normal surface $F$ if and only if it satisfies the following conditions:
\begin{enumerate}
    \item each coordinate of $v$ is nonnegative;
    \item fixing a parallelity piece $P$ and a component of intersection $C_1$ of $P$ with a 0-handle, for any other component $C_2$ of intersection of $P$ with a 0-handle, the number of elementary discs that intersect $C_1$ is equal to the number that intersect $C_2$;
    \item for each 1-handle $D^1\times D^2$, which has two components of intersection with 0-handles, for each arc type in $\{0\}\times D^2$ up to normal isotopy, the number of elementary discs in each 0-handle of the types that intersect the 1-handle in that arc type are the same; and
    \item for each pair of elementary disc types in a subtetrahedral 0-handle, if they have an essential intersection then $v$ contains only one of them.
\end{enumerate}

The first three conditions give us a cone $C$ in $\ZZ^{d_H\norm{\mathcal{H}}}$, as they are linear.
Note that if $v\in C$ satisfies the fourth condition and $v = t + u$ for $t, u\in C$, then $t$ and $u$ also satisfy the fourth condition.

As in the triangulation setting (recall Definition~\ref{defn:fundamental}), a normal surface $F$ is \emph{fundamental} if whenever $F = G_1 + G_2$ as a sum of normal surfaces, one of $G_1$ and $G_2$ is empty.

\begin{defn}
    The \emph{size} of a normal surface $F$, $s(F)$, is the total number of elementary discs in $F$.
    \label{defn:size}
\end{defn}

A fundamental normal surface must be in any (minimal) set that spans $C$ with $\ZZ^+$ coefficients; that is to say, in any \emph{minimal Hilbert basis}.
Linear integer programming techniques, as used by Haken~\cite{HakenNormalSurfaces} and Hass-Lagarias-Pippenger~\cite[\S 6]{HassLagariasPippenger} (see Lemma~\ref{lemma:triangulationfundamentalsurfacebound}), tell us that there is a universal constant $c$ such that any fundamental normal surface with respect to a triangulation with $t$ tetrahedra has size at most $c^t$.
Their work was in the setting of normal surfaces in triangulations but the general linear programming approach applies generally, as summarised by Lackenby:

\begin{prop}[Theorem 8.1~\cite{LackenbyCertificationKnottedness}]
Suppose that $A$ is a $m\times n$ matrix.
Consider the cone of solutions to $Ax = 0$, subject to the constraint that all entries of $x$ are nonnegative.
Suppose that each row of $A$ has $\ell^2$ norm at most $k$.
If $x$ is a fundamental solution (i.e. is in the integral Hilbert basis) to this system, then each coordinate of $x$ is bounded by $n^{3/2}k^{n-1}$.
\label{prop:generalfundamentalsolutionbounds}
\end{prop}

We can apply these bounds in our setting as follows.

\begin{restatable}{lemma}{fundamentalsurfacebound}
    \label{lemma:fundamentalsurfacebound}
    There exists a constant $c_F$, which we can take to be $2^{74+74\cdot 13\cdot 13!}$, such that if $G$ is a fundamental normal surface in a subtetrahedral \ths{} $\mathcal{H}$ then the size of $G$ is bounded by $c_F^{\norm{\mathcal{H}}}$.
\end{restatable}

To prove Lemma~\ref{lemma:fundamentalsurfacebound}, we will use the following result from the appendix on the combinatorics of subtetrahedral split handle structures:
    
\begin{restatable*}{lemma}{combinatorialconstraintssubtetrahedral}
    Let $H$ be a subtetrahedral 0-handle.
    Let $G$ be its boundary graph in $\del H \cong S^2$.
    Then $G$ contains between one and four islands; each island has at most three components of intersection with bridges; if $b$ is the number of bridges of $G$, then the number of sutures of $G$ is at most $12-2b$; and an island intersecting $v$ bridges has at most $6-2v$ intersections with sutures.
    \label{lemma:combinatorialconstraintssubtetrahedral}
\end{restatable*}
    
\begin{proof}[Proof of Lemma~\ref{lemma:fundamentalsurfacebound}]
    The cone $C$ of in $\ZZ^{d_H\norm{\mathcal{H}}}$ is defined by $m$ linear equations in $d_H\norm{\mathcal{H}}$ variables.
    We will bound the sum of the squares of the coefficients of each equation; that is, the $\ell^2$ norm of the corresponding vector.
    The constraints from the parallelity piece equations give at most $6\norm{\mathcal{H}}$ equations (one for each bridge), where the form of each equation is to set the sum of two lists of elementary discs to be equal, where each list is contained in one 0-handle, so each coefficient is between $-1$ and $1$.
    The $\ell^2$ norm of the coefficients of each equations is thus at most $d_H$.

   As each 0-handle has at most four islands, there are at most $\frac{4}{2}\norm{\mathcal{H}}$ 1-handles.
   Fix one of these islands, $I$.
   By Lemma~\ref{lemma:combinatorialconstraintssubtetrahedral} it intersects at most three bridges -- let $v$ be this number -- and at most $6-2v$ sutures.
   Only considering the bridges gives $v$ components of intersection of $I$ with lakes.
   Adding in the sutures contributes at most $3-v$ more, as a part of the forbidden region only increases the number of  components of intersection with lakes if it is bordered by sutures on both sides (see Figure~\ref{fig:suturesaddinpairs} for an illustration).
   We thus have divided the boundary of $I$ into at most 6 possible segments for elementary discs to enter and leave by (corresponding to intersections with bridges and lakes, but not the forbidden region), where an arc of intersection of an elementary disc with $I$ is determined by its entry and exit segments.
   There are thus at most $6\cdot 5 = 30$ arc types in $I$ that may be in the boundary of an elementary disc.
   Thus each 1-handle gives at most 30 equations, of the same form as in the parallelity piece case: setting some sums of elementary disc types to be equal, which will be an equation with $\ell^2$ norm at most $d_H$.

    Apply Proposition~\ref{prop:generalfundamentalsolutionbounds} to the matrix of these equations, which is a $m\times d_H\norm{\mathcal{H}}$ matrix where the $\ell^2$ norm of each row is at most $d_H$.
    The number of elementary discs of any given type of $G$ is thus bounded by $(d_H\norm{\mathcal{H}})^{3/2}d_H^{d_H\norm{\mathcal{H}}-1}$.
    We can then find a $c_F$ as follows:
    \begin{align*}
        |G| &\leq \norm{\mathcal{H}}(d_H\norm{\mathcal{H}})^{3/2}d_H^{d_H\norm{\mathcal{H}}}\\
            &= 2^{3/2\log(d_H) + 5/2\log(\norm{\mathcal{H}})+d_H\log(d_H)\norm{\mathcal{H}}}\\
            &\leq 2^{2(d_H+1)\log(d_H)\norm{\mathcal{H}}}
    \end{align*}
    so we can take $c_F$ to be $2^{2(d_H+1)\log(d_H)}$ which, as by Remark~\ref{remark:boundondH} $d_H$ is at most $13\cdot 13!< 2^{37}$, is bounded above by $2^{74+74\cdot 13\cdot 13!}$.
\end{proof}

\begin{figure}[th]
\centering
\resizebox{0.4\textwidth}{!}{
    \begin{tikzpicture}
            \coordinate (1) at (-1,0);
            \coordinate (2) at (0.5,0);
            \coordinate (3) at (0.5,0.7);
            \coordinate (4) at (-0.5, 0.5);
            \coordinate (5) at ($(1.center)+(-1,-0.2)$);
            \coordinate (6) at ($(1.center)+(-1,-0.8)$);
            \coordinate (7) at ($(1.center)+(-0.5,0)$);
            \coordinate (8) at ($(1.center)+(-0.3,-0.4)$);
            \draw[line width=1pt, red, fill=red, fill opacity=0.4] (2) -- (1) .. controls (4) .. (3);
            \draw[line width=1pt, red, fill=red, fill opacity=0.4] (5) .. controls (7) .. (1) .. controls (8) .. (6);

            \draw[line width=0.75pt, double, double distance=5pt] (1) -- (2);
            \draw[white,fill=white,line width=0.75pt, double, double distance=5pt] (2) -- (0.55,0);
            \node[circle, draw=black, line width=0.75pt, minimum size=16pt, fill=white] at (1) {};
        \end{tikzpicture}
    }
    \caption{A neighbourhood of an island in the boundary graph of a split 0-handle. The sutures are in red, and the forbidden region is shaded in red. Counting the number of components of intersection of the island with bridges and lakes, the pair of sutures that border a forbidden region increase the count by one, while the suture that borders a forbidden region with a bridge does not affect it.}
    \label{fig:suturesaddinpairs}
\end{figure}


Next, we will bound the size of a maximal collection of non-duplicate normal surfaces.

\begin{lemma}
    There is a constant $c \leq 4d_h \leq 26\cdot 13!$ such that if $M$ is a manifold with a triangulation $\mathcal{T}$, then any collection of disjoint normal surfaces in $M$, no two of which are normally isotopic, contains at most $c\norm{\mathcal{T}}$ surfaces.
    \label{lemma:kneserhaken}
\end{lemma}

\begin{proof}
    The proof is a standard Kneser-Haken finiteness argument.
    Let $n$ be the number of surfaces.
    We can assume that each surface is connected.
    Write $\mathcal{H}$ for the subtetrahedral split handle structure that is dual to $\mathcal{T}$.
    For each pair of surfaces where one is non-orientable and the other is its orientable double, discard the double, which reduces our count by at most a factor of two.
    Take the surfaces in sequence as $\Sigma_1,\ldots,\Sigma_n$ where we order the surfaces by putting some ordering on the elementary disc types, and for each elementary disc type in turn, adding to the list (in some order) any normal surfaces that contain the elementary discs of that type that do not border parallelity handles in $\mathcal{H}\backslash\backslash\{\Sigma_1\cup\cdots\cup\Sigma_{n}$.
    As we have removed orientable doubles of nonorientable cores, no normal surface consists solely of elementary discs that only border parallelity handles, as then it would be normally isotopic to the surface(s) neighbouring it.
    Thus when we enumerate the surfaces, we list at most two for each elementary disc type, so $n$ is at most $2d_H\norm{\mathcal{H}}$.
    By Remark~\ref{remark:boundondH}, $d_H$ is at most $13\cdot 13!$, so we can take $c = 26\cdot 13!$.
\end{proof}

\begin{restatable}{lemma}{stackedfundamentalsurfaces}
\label{lemma:stackedfundamentalsurfaces}
There exists a constant $c_B$, which we can take to be $2^{117\cdot 13!}$, such that the following holds.
Let $M$ be a manifold with a subtetrahedral \ths{} $\mathcal{H}$.
Let $f: \ZZ\to\ZZ$ be an increasing function with $f(1) \geq 1$.
Let $\{\Sigma_i\}$ be a collection of $n$ disjoint surfaces in $M$ such that, if we set $\mathcal{H}_0 := \mathcal{H}$ and $\mathcal{H}_i := \mathcal{H}_{i-1}\backslash\backslash \Sigma_i$, then $\Sigma_{i+1}$ is a non-\pns{} normal surface of size at most $2^{f(\norm{\mathcal{H}_i})}$ in the induced \ths{} on $\mathcal{H}_i$.
Then there is a normal surface representative of the collection in $\mathcal{H}$ whose size is at most $c_B^{\norm{\mathcal{H}}f(2^{13}\norm{\mathcal{H}})}$.
\end{restatable}

\begin{proof}
By Lemma~\ref{lemma:normalsurfacegivesths}, $\mathcal{H}_i$ is a \ths{} for each $i$.
Let $s(i)$ be the size of surfaces $\Sigma_{i},\ldots,\Sigma_n$ included into $\mathcal{H}_{i-1}$, where they are disjoint normal surfaces, and let $s(\Sigma_i)$ be the size of $\Sigma_i$ in $\mathcal{H}_{i-1}$.
When we include the surfaces $\Sigma_{i+1},\ldots,\Sigma_n$ from $\mathcal{H}_i$ into $\mathcal{H}_{i-1}$ and consider how their size changes, there are two ways in which the count of elementary discs may change: some of the parallelity pieces of $\mathcal{H}_i$ may be broken up into many 0-handle pieces and some of the elementary disc types of $\mathcal{H}_{i}$ may be identified.
This second type of modification does not change the total count of elementary discs in the surfaces and hence does not change the total size, so we can ignore it.

As the parallelity pieces had to have been created from cutting along $\Sigma_i$, each component of the intersection of the surfaces with a parallelity piece (each \emph{sheet}) in $\mathcal{H}_{i}$ is decomposed into at most $s(\Sigma_{i})$ elementary discs in 0-handles of $\mathcal{H}_{i-1}$.
As no component of a normal surface is contained in a parallelity piece, each sheet is adjacent to some elementary disc in a 0-handle of $\mathcal{H}_i$.
As there are at most six bridges in each of these 0-handles, each elementary disc runs through at most six bridges, so the total number of components of intersection of $\Sigma_{i+1}\cup\dots\cup\Sigma_n$ with parallelity pieces in $\mathcal{H}_i$ is at most $6s(i+1)$.
Thus when we include $\Sigma_{i+1}\cup\dots\cup\Sigma_n$ into $\mathcal{H}_{i-1}$, we increase the number of elementary discs by at most $6s(i+1)s(\Sigma_i)$.
Thus $s(i)$ is at most $s(\Sigma_i)+s(i+1)+6s(i+1)s(\Sigma_i)\leq 8s(i+1)s(\Sigma_i)$.

By Lemma~\ref{lemma:boundoncutopensize}, $\norm{\mathcal{H}_{i}}$ is at most $2^{13}\norm{\mathcal{H}}$ for all $i$.
Thus by the assumptions in the statement of this lemma, $s(\Sigma_i) \leq 2^{f(2^{13}\norm{\mathcal{H}})}$ for all $i$, so by induction, $s(i) \leq s(n)(8\cdot2^{f(2^{13}\norm{\mathcal{H}})})^{n-i}$.
We can now see that $s(0)$ is at most $8^n2^{(n+1)f(2^{13}\norm{\mathcal{H}})}$.
By Lemma~\ref{lemma:kneserhaken}, $n$ is at most $2d_H\norm{\mathcal{H}}$.
We can now bound $c_B$ as follows:
\begin{align*}
    s(0)	&\leq 8^n2^{(n+1)f(2^{13}\norm{\mathcal{H}})}\\
    		&\leq 2^{6d_H\norm{\mathcal{H}}}2^{(2d_H\norm{\mathcal{H}}+1)f(2^{13}\norm{\mathcal{H}})}\\
            &\leq 2^{9d_H\norm{\mathcal{H}}f(2^{13}\norm{\mathcal{H}})}
\end{align*}
which, as by Remark~\ref{remark:boundondH} $d_H$ is at most $13\cdot 13!$, gives us the bound.
\end{proof}

\begin{restatable}{corr}{stackedfundbound}
    \label{corr:stackedfundbound}
    There exists a constant $c_{S}$, which we can take to be $10^{10^{30}}$, such that the following holds.
    Let $M$ be a manifold with a subtetrahedral split handle structure $\mathcal{H}$.
    Let $\{\Sigma_i\}$ be a collection of $n$ disjoint surfaces in $M$ such that, if we set $\mathcal{H}_0 := \mathcal{H}$ and $\mathcal{H}_i := \mathcal{H}_{i-1}\backslash\backslash\Sigma_i$, then $\Sigma_{i+1}$ is a non-duplicate normal surface in $\mathcal{H}_i$ that is either fundamental or the double of a fundamental surface.
    There is a normal surface representative of the collection in $\mathcal{H}$ whose size is at most $c_{S}^{\norm{\mathcal{H}}^2}$.
\end{restatable}

\begin{proof}
    This immediately follows from Lemmas~\ref{lemma:fundamentalsurfacebound} and~\ref{lemma:stackedfundamentalsurfaces}, as well as a calculation that $c_B^{\norm{H}f(2^{13}\norm{H})} \leq 2^{221\cdot 13!\cdot 2^{13}(75+74\cdot 13\cdot 13!)\norm{H}^2}< \left(10^{10^{30}}\right)^{\norm{H}^2}$.
\end{proof}

\section{A complete bounded collection of normal vertical annuli}
\label{section:fundamentalannulicollection}

We will use normal surfaces in \ths{}s, as developed in Section~\ref{section:ths} and Appendix~\ref{appendix:normalsurfacesinths}, to show that there is a maximal collection of normal vertical annuli of bounded size.
The main result of Appendix~\ref{appendix:normalsurfacesinths} is the following, which is a generalisation of Theorem~\ref{thm:matveev}.

\begin{restatable*}{prop}{incompressiblenormalsurface}
\label{prop:incompressiblenormalsurface}
    Let $M$ be an irreducible $\del$-irreducible 3-manifold and let $\mathcal{H}$ be a subtetrahedral \ths{} for $M$.
    Let $F$ be a normal surface of minimal weight in its (admissible) isotopy class which is incompressible and $\del$-incompressible, with $F = G_1 + G_2$.
    Then $G_1$ and $G_2$ are also incompressible and $\del$-incompressible, and neither $G_1$ nor $G_2$ is a copy of $S^2$, $\RR P^2$ or a disc.
\end{restatable*}

\subsection{Fundamental M\"obius bands}

\begin{lemma}
\label{lemma:injmobiusincompdouble}
If $S$ is a M\"obius band in an orientable irreducible 3-manifold $M$ where $\del S$ is non-trivial in $\pi_1(M)$, then the double of $S$ is essential unless $M$ is $S^1\times D^2$.
\end{lemma}

\begin{proof}
As the double of $S$, $\tilde{S}$, is an annulus, there is only one loop (up to isotopy) in $\tilde{S}$ that does not bound a disc in $\tilde{S}$: one of its boundary components.
As this is isotopic to $\del S$, it is non-trivial in $\pi_1(M)$ so cannot bound a disc in $M$ either, so $\tilde{S}$ is incompressible.
As $\tilde{S}$ is the boundary of a neighbourhood of $S$, it is separating.
Since $M$ is irreducible, an annulus is $\del$-parallel if and only if it is $\del$-compressible.
In this case, we can note double of a M\"obius band is not, in fact, boundary-parallel in the neighbourhood of the M\"obius band.
So if $\tilde {S}$ is not essential, it cuts $M$ into $S^1\times I\times I$ and $S\ttimes I$, so $M$ is $S\ttimes I\cong S^1\times D^2$.
\end{proof}

\begin{lemma}[Lemma 2.4~\cite{MijatovicTriangulationsSFS}]
\label{lemma:essentialannuliisotopicboundary}
Suppose that $M$ is irreducible, $\del$-irreducible, and is not $T^2\times I$ or $K\ttimes I$.
Let $S$ be a toroidal boundary component of $M$.
If $A$ and $B$ are properly-embedded incompressible and $\del$-incompressible annuli in $M$ (that are not necessarily disjoint), such that at least one boundary component of each is on $S$, then these boundary components are isotopic.
\end{lemma}

\begin{lemma}
\label{lemma:fundamentalmobiusband}
    Suppose that $M$ is a Seifert fibered space with non-empty boundary, not homeomorphic to a solid torus or $K\ttimes I$, and has a subtetrahedral \ths{} $\mathcal{H}$, where the forbidden region (if non-empty) is a collection of vertical annuli.
    If $M$ has any multiplicity two singular fibers, then for each boundary component $T$ of $M$, there is a fundamental normal pseudo-vertical M\"obius band containing one of these fibers whose boundary is on $T$.
\end{lemma}

\begin{proof}
   Note that no Seifert fibration of $T^2\times I$ has multiplicity two singular fibers.
   Let $S$ be a pseudo-vertical M\"obius band that is incompressible, $\del$-incompressible, and disjoint from $\mathcal{I}$.
   Up to isotopy, we can consider it to be a union of regular fibers and a multiplicity two singular fiber.
   By Lemma~\ref{lemma:injmobiusincompdouble} the double of $S$, $\tilde{S}$, is essential in $M$ with empty forbidden region, and so is essential in $\mathcal{H}$.

   Let $F$ be a minimal normal surface representative of $S$.
   Suppose that $F = G_1 + G_2$.
   By Proposition~\ref{prop:incompressiblenormalsurface}, $G_1$ and $G_2$ are incompressible and $\del$-incompressible, and each has Euler characteristic at most zero, and so (as $\chi(F) = 0$) equal to zero.
   By Lemma~\ref{lemma:nonorientablesurfacesinSFS} they are horizontal or (pseudo-)vertical.
   If one is a M\"obius band, say $G_1$, taking the double and observing by the same reasoning as above that it is essential, we can see by Lemma~\ref{lemma:essentialannuliisotopicboundary} that $\del G_1$ is isotopic to $\del F$.
   It is thus a pseudo-vertical M\"obius band, so is (up to isotopy) a union of regular fibers and a single multiplicity two singular fiber.
   As the boundary of $G_1$ is a summand of that of $F$, $\del G_1$ is contained in $T$.
   By Lemma~\ref{lemma:nonzerobeamdegree} the weight of $G_2$ is non-zero so the weight of $G_1$ is strictly less than that of $F$.
   
   Otherwise, as at least one of $G_1$ and $G_2$ has boundary, either both are annuli or one is an annulus and the other is closed.
   In either of these cases, we can apply Lemma~\ref{lemma:essentialannuliisotopicboundary} to the boundary components of the essential annuli to see that their boundary curves are isotopic to $\del F$.
   There are an even number of boundary curves among $G_1$ and $G_2$, so $\del G_1 + \del G_2$ is trivial in $H_1(\del M; \ZZ_2)$.
   But $\del F$ is not as it is a single curve and normal sum is additive on $\ZZ_2$-homology, so we have a contradiction. 
   Thus there is a fundamental pseudo-vertical M\"obius band.
\end{proof}

\subsection{Fundamental vertical annuli}

\begin{lemma}
    \label{lemma:noduplicatesummands}
    Let $F$ be a connected non-\pns{} normal surface in a subtetrahedral \ths{} $\mathcal{H}$.
    If $F = G_1 + G_2$ then neither $G_1$ nor $G_2$ is both connected and \pns{}.
\end{lemma}

\begin{proof}
If $G_1$ were connected and \pns{}, which is to say that there was a normal isotopy taking it to the boundary of a collar of a forbidden region component, then this isotopy would also make it disjoint from $G_2$, so then $F = G_1 + G_2$ would have two components.
\end{proof}

\begin{lemma}
\label{lemma:oneessentialannulus}
    Let $M$ be a Seifert fibered space with non-empty boundary other than the solid torus.
    Let $\mathcal{H}$ be a subtetrahedral \ths{} for $M$ where the forbidden region $\mathcal{I}$ (if it is non-empty) is a collection of vertical annuli.
    Then $\mathcal{H}$ contains a fundamental normal surface, disjoint from the forbidden region, that is either an essential M\"obius band or an essential non-\pns{} annulus.
    Furthermore, this surface must be vertical unless $\mathcal{I}$ is empty and $M$ is homeomorphic to $T^2\times I$ or $K\ttimes I$.
\end{lemma}

\begin{proof}
We have already proved this in Lemma~\ref{lemma:fundamentalmobiusband} if $M$ has a multiplicity two singular fiber and is not $K\ttimes I$, by doubling the M\"obius band given in the lemma.
If $M$ does have a multiplicity two singular fiber, then, we can assume that it is $K\ttimes I$ with the $[D^2, 1/2, 1/2]$ Seifert fibration.

Let $A$ be an incompressible $\del$-incompressible annulus or M\"obius band in $M$ that is disjoint from the forbidden region and is not isotopic into it.
This exists as $M$ is not the solid torus.
By Proposition~\ref{prop:cannormaliseincompressiblesurfaces} $A$ has a normal representative $F$.
As $A$ is incompressible, it cannot be isotoped to be inside a single handle, so we can let $F$ be of minimal weight in the admissible isotopy class of $A$ (that is, up to isotopy fixing $\mathcal{I}$).
Suppose that $F = G_1 + G_2$ as a non-trivial sum of normal surfaces, where we choose $G_1$ and $G_2$ to minimise $|G_1\cap G_2|$.
By Proposition~\ref{prop:incompressiblenormalsurface}, $G_1$ and $G_2$ are also incompressible and $\del$-incompressible, and have Euler characteristic at most 0.
If (say) $G_1$ has multiple connected components, so $F = G_1' + G_1'' + G_2$, then by resolving the sum $G_1'' + G_2$ (which we note does not add any curves of intersection with $G_1'$) we get $F = G_1' + (G_1'' + G_2)$, where $|G_1\cap G_2| > |G_1' \cap (G_1'' + G_2)|$, so both $G_1$ and $G_2$ must be connected.
As $\chi(G_1) + \chi(G_2) = \chi(A) = 0$, both have Euler characteristic exactly zero -- that is, each component is an annulus, M\"obius band, Klein bottle or torus.
By Lemma~\ref{lemma:noduplicatesummands} neither is \pns{}.

If one is closed, the other, say $G_1$, has the same boundary as $A$, so is also an essential non-\pns{} annulus or M\"obius band.
If both are essential annuli, then we have two essential non-\pns{} annuli of smaller weight than $A$.

If $G_1$ is a M\"obius band then by Lemma~\ref{lemma:nonorientablesurfacesinSFS} it is either horizontal, in which case $\mathcal{I}$ must be empty and $M$ must be $K\ttimes I$, or pseudo-vertical so (as $M$ is not $S^1\times D^2$) its double is an essential vertical annulus.

In each case, we have produced an incompressible and $\del$-incompressible non-\pns{} annulus or M\"obius band that is of smaller weight than $F$, so there is a fundamental such surface $G$.
If $\mathcal{I}$ is not empty then $G$ cannot be horizontal, so must be (pseudo-)vertical and, if a M\"obius band, have an essential vertical double (as $M\not\cong S^1\times D^2$).
Otherwise, if $G$ is horizontal, then by Lemma~\ref{lemma:horizontalmobiusclassification} $M$ is homeomorphic to $K\ttimes I$ or $T^2\times I$.
\end{proof}

We will use Lemmas~\ref{lemma:fundamentalsurfacebound} and~\ref{lemma:stackedfundamentalsurfaces} to prove that we can obtain a complete collection of these annuli of bounded size.
Recall that by Lemma~\ref{lemma:subtetrahedralths}, cutting along a normal surface in a subtetrahedral \ths{} produces another subtetrahedral \ths{}.

\begin{prop}
\label{prop:verticalannulicollectionhandlestructure}
There exists a constant $c_A$, which we can take to be $10^{10^{30}}$, such that the following holds.
Let $M$ be a Seifert fibered space with non-empty boundary.
Let $\mathcal{H}$ be a subtetrahedral handle structure for $M$ that is dual to a triangulation $\mathcal{T}$ (in the sense of Definition~\ref{defn:dual_handle}) such that the intersection of each tetrahedron of $\mathcal{T}$ with $\del M$ is connected and contractible.
There is a maximal collection of disjoint normal essential annuli in $\mathcal{H}$, so that no two are isotopic, of total size at most $c_A^{\norm{\mathcal{H}}^2}$.
If $M$ is not $T^2\times I$ or $K\ttimes I$, then these annuli are vertical.
\end{prop}

\begin{proof}
If $M$ is $S^1\times D^2$, there are no essential annuli so the result is trivial.
Otherwise, by Lemma~\ref{lemma:subtetrahedralths}, $\mathcal{H}$ can be viewed as a subtetrahedral \ths{} with empty forbidden region.
Then by Lemma~\ref{lemma:oneessentialannulus} there is a fundamental normal surface $F$ in $\mathcal{H}$ that is either an essential M\"obius band or an essential non-duplicate annulus.
Note that by Lemma~\ref{lemma:subtetrahedralths} $\mathcal{H}\backslash\backslash F$ is also a subtetrahedral \ths{}.
If $M$ is not $T^2\times I$ or $K\ttimes I$, we can take the first annulus to be vertical, so the forbidden region of $\mathcal{H}\backslash\backslash F$ will be a non-empty collection of vertical annuli.
We can then continue to apply Lemma~\ref{lemma:oneessentialannulus} to generate a collection of disjoint normal incompressible and $\del$-incompressible annuli $A_i$ and handle structures $\mathcal{H}_{i+1} = \mathcal{H}_i\backslash\backslash A_i$, where the annulus $A_i$ is fundamental in $\mathcal{H}_i$ or is the double of a fundamental surface.
This is the situation of Corollary~\ref{corr:stackedfundbound}, so there is a normal surface representative of the whole collection of size at most $\left(10^{10^{30}}\right)^{|\mathcal{H}|^2}$.

If $M$ is $T^2\times I$, then by Proposition~\ref{prop:incompressibleinthickenedtorus} any essential annulus in $T^2\times I$ is $\gamma\times I$ for some essential curve $\gamma$ in $T^2$, and then the complement of this annulus is a solid torus with two parallel annuli as its forbidden region, so there are no further isotopy classes of disjoint essential annuli.
If $M$ is $K\ttimes I$, it has a Seifert fibered structure as $S\ttimes S^1$, where $S$ is the M\"obius band.
An essential annulus is either horizontal or vertical.
If horizontal, it separates $M$ into two copies of $S\ttimes I$, which are each homeomorphic to a solid torus with a single annulus as their forbidden region so contain no essential annuli that are not isotopic to the one we already cut along.
If vertical, it cuts $M$ into a single solid torus, with two parallel annuli as its forbidden region, so again there are no further annuli.
In these cases we took a single fundamental annulus, or double of a fundamental M\"obius band, which is of size at most $2c_F^{\norm{\mathcal{H}}}$ for $c_F = 2^{74+74\cdot 13\cdot 13!}$ by Lemma~\ref{lemma:fundamentalsurfacebound}, which is certainly smaller than bound on $c_A$ we have given.
\end{proof}

\begin{lemma}
\label{lemma:surfaceinhandlestructuretodual}
Let $\mathcal{T}$ be a triangulation of an irreducible $\del$-irreducible 3-manifold $M$ such that the intersection of each tetrahedron of $\mathcal{T}$ with $\del M$ is connected and contractible.
Let $\mathcal{H}$ be the handle structure dual to $\mathcal{T}$ as defined in Definition~\ref{defn:dual_handle}.
Let $F$ be a non-\pns{} incompressible $\del$-incompressible normal surface in $\mathcal{H}$, of size $s(F)$.
There is a normal surface $F'$ in $\mathcal{T}$ that is isotopic to $F$ such that the edge weight of $F'$ (that is, its number of intersections with the 1-skeleton) is at most $16\norm{\mathcal{T}}s(F)$.
\end{lemma}

\begin{proof}
By Lemma~\ref{lemma:subtetrahedralths} $\mathcal{H}$ is a subtetrahedral split handle structure.
Each elementary disc of $F$ runs over at most four bridges, so its number of intersections with the 1-skeleton of $\mathcal{T}$ will be at most four.
The barrier to directly including $F$ into $\mathcal{T}$ is that an elementary disc in a subtetrahedral handle does not necessarily correspond to one in the dual tetrahedron: the ones that do not are the ones that run over vertices in the boundary of $\mathcal{T}$.
It suffices then to perturb $F$ enough that it is transverse to $\mathcal{T}$, without increasing its edge weight too much, as then, as $F$ is incompressible and $\del$-incompressible, we can normalise it which does not increase its edge weight.

Map $F$ into $\mathcal{T}$ by sending each elementary disc of $F$ to a disc in a tetrahedron of $\mathcal{T}$ that is transverse to the triangulation except that it may run over vertices in $\del \mathcal{T}$.
Each elementary disc of $F$ ran over at most four bridges, so this surface intersects the 1-skeleton of $\mathcal{T}$ in at most $4s(F)$ points.
At each vertex in $\del \mathcal{T}$, consider its link $L$, which is a disc as $M$ is a manifold.
The elementary discs of $F$ intersect $L$ in a set of disjoint arcs.
Pick a coherent choice of \emph{direction} transverse to each arc, so that no two arcs point towards each other.
At each arc, replace the portion of $F$ that runs through the vertex with the subset of $L$ in the chosen direction.
We thus obtain a surface $F'$ isotopic to $F$ and transverse to the triangulation.
As $L$ intersects each edge of $\mathcal{T}$ at most twice, this operation adds at most two points of intersection of each elementary disc with each edge of $\mathcal{T}$.
There are at most $6\norm{\mathcal{T}}$ edges so the edge weight of $F'$ is at most $(4+12\norm{\mathcal{T}})s(F)$.
Normalising $F'$ produces a surface isotopic to $F$ and of no greater edge weight than $F'$.
\end{proof}

\begin{corr}
\label{corr:verticalannulicollectiontriangulation}
There exists a constant $c_T$, which we can take to be $10^{10^{36}}$, such that the following holds.
Let $\mathcal{T}$ be a triangulation of a Seifert fibered space $M$ with non-empty boundary, other than the solid torus.
There is a collection of disjoint normal essential annuli in $\mathcal{T}$, such that their complement is a collection of solid tori, of total edge weight at most $c_T^{\norm{\mathcal{T}}^2}$.
We can take these annuli to be vertical so long as $M$ is not $T^2\times I$ or $K\ttimes I$.
\end{corr}

\begin{proof}
Consider the handle structure $\mathcal{H}$ that is dual to the $1^{st}$ barycentric subdivision of $\mathcal{T}$, $\mathcal{T}^{(1)}$, which contains $24\norm{\mathcal{T}}$ tetrahedra.
In $\mathcal{T}^{(1)}$ the intersection of each tetrahedron with $\del M$ is connected and contractible.
Thus $\mathcal{H}$ satisfies the hypotheses of Proposition~\ref{prop:verticalannulicollectionhandlestructure} so there is a maximal collection $C$ of disjoint normal essential annuli in $\mathcal{H}$, such that no two are isotopic, of total size at most $c_A^{\norm{\mathcal{H}}^2}$, and so that if $M$ is not $T^2\times I$ or $K\ttimes I$ then these annuli are vertical.
We can now apply Lemma~\ref{lemma:surfaceinhandlestructuretodual} to this collection $C$ to obtain an isotopic set of normal surfaces in $\mathcal{T}^{(1)}$ of edge weight at most $16\cdot 24\norm{\mathcal{T}}c_A^{24^2\norm{\mathcal{T}}^2} \leq c_A^{2\cdot 24^2\norm{\mathcal{T}}^2}$ where $c_A < 10^{10^{30}}$.
We will set $c_T$ to be $c_A^{2\cdot 24^2}$, which gives the desired bound.

Now, if $M$ is $T^2\times I$ or $K\ttimes I$, we took a single annulus that cut $M$ into one or two solid tori, so let this be the minimal collection.
Otherwise, all the annuli are vertical.
They lie over essential arcs in the base orbifold of $M$.
If there are $n$ singular fibers, there must be $n$ annuli separating neighbourhoods of singulars fiber from the remainder of $M$.
Consider the dual graph to the rest of the collection: it has a vertex for each region of $M\backslash\backslash C$, and an edge for each annulus, connecting the vertices of the regions it bounds.
Take the complement of a spanning tree of this dual graph.
This collection of annuli will cut the remainder of the orbifold into a single disc, and so will cut the remainder of $M$ into a single solid torus.
For the minimal collection, take these annuli as well as the ones that cut off singular fiber neighbourhoods.
\end{proof}

\section{Recognising circle bundles over surfaces with boundary is in NP}
\label{section:surfacebundlerecognition}

\begin{lemma}
\label{lemma:solidtorusdata}
    Let $\mathcal{T}$ be a triangulation of the solid torus.
    There is a fundamental normal meridian disc in $\mathcal{T}$ and a normal curve in $\del \mathcal{T}$ that intersects the disc once and whose number of intersections with the edges of $\del\mathcal{T}$ is bounded by an exponential in $\norm{\mathcal{T}}$.
\end{lemma}

\begin{proof}
Corollary 6.4 of~\cite{JacoTollefson} tells us that if $M$ is $\del$-irreducible then there is a fundamental normal essential disc, and the only essential disc in $S^1\times D^2$ is the meridian disc.
Take a normal curve $\gamma$ in $\del \mathcal{T}$ that intersects this meridian disc once.
If $\gamma = \alpha_1 + \alpha_2$ as a normal sum, then one of the $\alpha_i$ must also intersect the disc once, so there is a fundamental normal curve in the triangulation on the boundary that satisfies this property, which (as with fundamental normal surfaces) contains at most an exponential number of arcs in the size of the triangulation of $\del\mathcal{T}$ and hence in $\norm{\mathcal{T}}$.
\end{proof}

To determine the slope of curves in a torus we will need to compute their algebraic intersections quickly.
A normal curve is not equipped with an orientation, so the algebraic intersection does not come with a sign; however, if we have more than two (pairwise-intersecting) curves, then picking a sign for $i(\alpha, \beta)$ and $i(\beta, \gamma)$ determines that of $i(\alpha, \gamma)$.
As in~\cite[\S6]{EricksonNayyeri}, we can represent an oriented normal curve by giving its algebraic intersection number with each (oriented) edge of the triangulation.
We will use the following results to assign orientations to normal curves and then compute their algebraic intersections.

\begin{prop}[Corollary 6.11~\cite{EricksonNayyeri}]
\label{prop:ericksonnayyeri}
Let $M$ be a compact 2-manifold with triangulation $T$, containing $\norm{T}$ triangles.
Let $\gamma$ be a connected normal curve in $T$, represented by its unsigned normal coordinates.
There is an algorithm to compute the signed normal coordinates of some orientation of $\gamma$ in time polynomial in $\norm{T}$ and $\log(s(\gamma))$, where $s(\gamma)$ is the number of normal arcs in $\gamma$.
\end{prop}

The following lemma is a small modification of the approaches of Erickson-Nayyeri \cite[Corr.\ 6.12]{EricksonNayyeri} and Schaefer-Sedgwick-\v{S}tefankovi\v{c} \cite[\S5.6]{SedgwickNormalCurveAlgos}.

\begin{lemma}
\label{lemma:intersectionnumberpolynomialtime}
    Let $\gamma_1,\ldots,\gamma_n$ be a collection of connected normal curves in a triangulation $T$ of a compact 2-manifold (containing $\norm{T}$ triangles), with an orientation of the surface that is given by an orientation of each triangle.
    Let $s(\gamma_i)$ be the number of arcs in $\gamma_i$.
    We can compute the algebraic intersections of each pair of curves for some (fixed) choice of curve orientations in polynomial time in $\norm{T}$ and $\sum_{i=1}^n \log(s(\gamma_i))$.
\end{lemma}

\begin{proof}
    Compute the signed normal coordinates of each curve using Proposition~\ref{prop:ericksonnayyeri} and thus fix orientations on them.
    Pick a pair of curves $\alpha$ and $\beta$, and for each edge of $T$, arbitrarily assign one end of the edge to each curve.
    If we isotope each curve so that it intersects each edge only in that half, and draw each elementary segment as a straight line arc within each triangle, then this assignment forces each pair of elementary arc types of $\alpha$ and $\beta$ to intersect either once or not at all.
    Fix a triangle $t$.
    We wish to compute the total (algebraic) number of (oriented) intersections of $\alpha$ and $\beta$ in $t$.
    Modify the orientation of the edges of $t$ so that they are induced by some orientation of $t$, changing the sign of the signed normal coordinates of $\alpha$ and $\beta$ as necessary.
    See Figure~\ref{fig:algintersectioncalc} for an example.
    
    Any set of elementary arcs for $\alpha$ and $\beta$ in $t$ that intersect the edges of $t$ at the chosen ends and induce the same signed normal coordinates contributes the same total algebraic intersection in $t$, so we can pick a convenient representation.
    The total of the signed normal coordinates of each curve with respect to the edges of $t$ is zero, so if (say) the signed normal coordinates of $\alpha$ with respect to the edges $A$, $B$ and $C$ is $(a, b, -c)$, where $a$, $b$ and $c$ are non-negative, then we can realise $\alpha$ using $c$ (oriented) elementary arcs: $a$ go from $A$ to $C$, and $b$ from $B$ to $C$.
    Repeating this with $\beta$, we see that we can also realise $\beta$ with two oriented elementary arc types, so  (as described in~\cite[Corr.\ 6.12]{EricksonNayyeri} and~\cite{SedgwickNormalCurveAlgos}), since each pair of elementary arcs from $\alpha$ and $\beta$ intersects zero or one times, we can compute the number of positive and negative crossings of the pair within each triangle by multiplying at most six pairs of integers, which were entries in the signed normal coordinates.
    This takes time that is polynomial in the length of the bitstring representation of the normal coordinates, so repeating this for each triangle and summing the results, we overall have an algorithm whose running time is polynomial in $\norm{T}$ and $\sum_{i=1}^n \log(s(\gamma_i))$.
\end{proof}
    
\begin{example}
In Figure~\ref{fig:algintersectioncalc:a}, the signed normal coordinates of $\alpha$ and $\beta$ in $t$ with the indicated edge orientations are $(2, -4, 2)$ and $(1, -3, 2)$ respectively.
We can realise these coordinates (as shown in Figure~\ref{fig:algintersectioncalc:b}) if we represent $\alpha$ with two elementary arcs running from $A$ to $B$ and two from $C$ to $B$ -- that is, $(\alpha(AB), \alpha(AC), \alpha(BC)) = (2, 0, -2)$; we represent $\beta$ with $(1, 0, -2)$ in the same fashion.
With the chosen assignments of each end of each edge of $t$ to $\alpha$ and $\beta$, we find that $i(\alpha, \beta)$ gets a contribution from $t$ of
$$\alpha(AB)\beta(AC) + \alpha(AB)\beta(BC) + \alpha(A C)\beta(A C) + \alpha(A C)\beta(B C)$$
which in this case is $2\cdot 0 + 2\cdot -2 + 0\cdot 0 + 0\cdot -2=-4$.
      
 \begin{figure}[th]
      \centering
      \begin{subfigure}{0.43\textwidth}
      \centering
      \resizebox{\textwidth}{!}{
      \begin{tikzpicture}
                  \coordinate (1) at (0, 0);
                  \coordinate (2) at (-1.4, -2);
                  \coordinate (3) at (1.4, -2);
                  \begin{scope}[decoration = {markings, mark=at position 0.5 with {\arrow{<}}}]
                  \draw[postaction={decorate}] (1) -- (2);
                  \draw[postaction={decorate}] (2) -- (3);
                  \draw[postaction={decorate}] (3) -- (1);
                  \end{scope}
                  \begin{scope}[decoration = {markings, mark=at position 0.5 with {\arrow{Latex[width=0pt 6, length=4pt]}}}]
                                    \draw[postaction = {decorate}, color=orange] ($(1)!0.2!(2)$) -- ($(1)!0.7!(3)$);
                                    \draw[postaction = {decorate}, color=orange] ($(1)!0.3!(2)$) -- ($(2)!0.7!(3)$);
                                    \draw[postaction = {decorate}, color=orange] ($(2)!0.65!(3)$) -- ($(1)!0.35!(2)$);
                                    \draw[postaction = {decorate}, color=orange] ($(1)!0.4!(2)$) -- ($(2)!0.6!(3)$);
                                    \draw[postaction = {decorate}, color=orange] ($(1)!0.75!(3)$) -- ($(2)!0.75!(3)$);
                                    \draw[postaction = {decorate}, color=orange] ($(1)!0.8!(3)$) -- ($(2)!0.8!(3)$);
                                    \draw[postaction = {decorate}, color=orange] ($(1)!0.85!(3)$) -- ($(2)!0.85!(3)$);
		 \end{scope}
		 \begin{scope}[decoration = {markings, mark=at position 0.5 with {\arrow{Latex[width=0pt 4, length=4pt]}}}]
                                    \draw[postaction = {decorate}, thick, color=violet, densely dotted] ($(1)!0.75!(2)$) -- ($(2)!0.3!(3)$);
                                    \draw[postaction = {decorate}, thick, color=violet, densely dotted] ($(1)!0.8!(2)$) -- ($(2)!0.25!(3)$);
                                    \draw[postaction = {decorate}, thick, color=violet, densely dotted] ($(2)!0.2!(3)$) -- ($(1)!0.85!(2)$);
                                    \draw[postaction = {decorate}, thick, color=violet, densely dotted] ($(1)!0.6!(2)$) -- ($(1)!0.25!(3)$);
                                    \draw[postaction = {decorate}, thick, color=violet, densely dotted] ($(1)!0.3!(3)$) -- ($(1)!0.7!(2)$);
                                    \draw[postaction = {decorate}, thick, color=violet, densely dotted] ($(1)!0.35!(3)$) -- ($(2)!0.35!(3)$);
                                    \draw[postaction = {decorate}, thick, color=violet, densely dotted] ($(1)!0.4!(3)$) -- ($(2)!0.4!(3)$);
\end{scope}
                  \node () at ($(1)!0.5!(2)+(-0.2,0.2)$) {\tiny{$A$}};
                  \node () at ($(1)!0.5!(3)+(0.2,0.2)$) {\tiny{$C$}};
                  \node () at ($(2)!0.5!(3)+(0,-0.2)$) {\tiny{$B$}};

      \end{tikzpicture}}
      \caption{With the indicated edge orientations, the signed normal coordinates of $\alpha$ and $\beta$ are $(2, -4, 2)$ and $(1, -3, 2)$.}
      \label{fig:algintersectioncalc:a}
      \end{subfigure}
       \begin{subfigure}{0.43\textwidth}
      \centering
      \resizebox{\textwidth}{!}{
      \begin{tikzpicture}
                  \coordinate (1) at (0, 0);
                  \coordinate (2) at (-1.4, -2);
                  \coordinate (3) at (1.4, -2);
                  \begin{scope}[decoration = {markings, mark=at position 0.5 with {\arrow{<}}}]
                  \draw[postaction={decorate}] (1) -- (2);
                  \draw[postaction={decorate}] (2) -- (3);
                  \draw[postaction={decorate}] (3) -- (1);
                  \end{scope}
                  \begin{scope}[decoration = {markings, mark=at position 0.5 with {\arrow{Latex[width=0pt 6, length=4pt]}}}]
                                    \draw[postaction = {decorate}, color=orange] ($(1)!0.35!(2)$) -- ($(2)!0.65!(3)$);
                                    \draw[postaction = {decorate}, color=orange] ($(1)!0.4!(2)$) -- ($(2)!0.6!(3)$);
                                    \draw[postaction = {decorate}, color=orange] ($(1)!0.75!(3)$) -- ($(2)!0.75!(3)$);
                                    \draw[postaction = {decorate}, color=orange] ($(1)!0.8!(3)$) -- ($(2)!0.8!(3)$);
		 \end{scope}
		 \begin{scope}[decoration = {markings, mark=at position 0.5 with {\arrow{Latex[width=0pt 4, length=4pt]}}}]
                                    \draw[postaction = {decorate}, thick, color=violet, densely dotted] ($(1)!0.75!(2)$) -- ($(2)!0.3!(3)$);
                                    \draw[postaction = {decorate}, thick, color=violet, densely dotted] ($(1)!0.35!(3)$) -- ($(2)!0.35!(3)$);
                                    \draw[postaction = {decorate}, thick, color=violet, densely dotted] ($(1)!0.4!(3)$) -- ($(2)!0.4!(3)$);
\end{scope}
                  \node () at ($(1)!0.5!(2)+(-0.2,0.2)$) {\tiny{$A$}};
                  \node () at ($(1)!0.5!(3)+(0.2,0.2)$) {\tiny{$C$}};
                  \node () at ($(2)!0.5!(3)+(0,-0.2)$) {\tiny{$B$}};

      \end{tikzpicture}}
      \caption{A reduced form of the elementary arcs that preserves the algebraic intersection count of $\alpha$ and $\beta$ in $t$.}
      \label{fig:algintersectioncalc:b}
      \end{subfigure}
      \caption{Computing the algebraic intersection number of two normal curves in a triangle $t$. The solid orange elementary arcs are from $\alpha$, and the dotted violet arcs are from $\beta$.}
      \label{fig:algintersectioncalc}
      \end{figure}
\end{example}

Agol, Hass and Thurston~\cite{AgolHassThurston} created a polynomial time algorithm that, given a normal surface in a triangulation as a vector, computes its topology. 
Lackenby, Haraway and Hoffman have used it to quickly cut triangulations along normal surfaces.

\begin{prop}[Corollary 17~\cite{AgolHassThurston}]
\label{prop:agolhassthurstonnormalsurface}
Let $M$ be a 3-manifold with a triangulation $\mathcal{T}$ and let $F$ be a normal surface in $M$.
There is a procedure for counting the number of components of $F$ and determining the topology of each component that runs in time polynomial in $\norm{\mathcal{T}}\log w(F)$, where $w(F)$ is the edge weight of $F$.
\end{prop} 

The edge weight $w(F)$ for a normal surface in a triangulation is at most $4s(F)$ -- the bound is obtained for a collection of quadrilaterals in a single tetrahedron. 

\begin{lemma}
\label{lemma:cutalongsurfacestriangulation}
There is an algorithm that takes as its input both a triangulation $\mathcal{T}$ of a compact orientable 3-manifold $M$ and a (possibly disconnected) orientable normal surface $F$ in $\mathcal{T}$ given as a vector, such that no two components of $F$ are normally isotopic, and provides as its output a triangulation of $M\backslash\backslash F$ whose size is bounded by a polynomial in $\norm{\mathcal{T}}$, $\log w(F)$, and the minimal Euler characteristic of the components of $F$, and runs in time polynomial in those same three parameters.
It also produces normal surface vectors for each component of the boundary of the copy of the double of $F$ in $\del (M\backslash\backslash F)$. 
\end{lemma}

As the proof is a slight modification of the methods in the proof of Proposition 13 of~\cite{HarawayHoffman} and Theorem 9.2 of~\cite{LackenbyCertificationKnottedness}, we defer it to Appendix~\ref{appendix:cutalongsurfacetriangulation}.

\begin{prop}
\label{prop:circlebundlerecognitionNP}
Deciding whether a 3-manifold $M$ is an orientable circle bundle over a surface with non-empty boundary is in \NP.
Deciding whether such a 3-manifold is a circle bundle over a certain surface (given as genus, number of boundary components, and orientability) is in \NPcoNP{}.
Furthermore, unless $M \cong K\ttimes I$ or $T^2\times I$, there exists a normal section and one normal fiber on each boundary component, such that these properties can be certified in time polynomial in $\norm{\mathcal{T}}$.
\end{prop}

Haraway and Hoffman have previously shown that certifying $K\ttimes I$ and $T^2\times I$ is in NP~\cite{HarawayHoffman}; these are the orientable circle bundles over the M\"obius band and the annulus respectively.

\begin{proof}[Proof of Proposition~\ref{prop:circlebundlerecognitionNP}]
Let $\mathcal{T}$ be a triangulation of $M$.
The data given in the certificate is the following:
\begin{enumerate}
    \item A compatible orientation of each tetrahedron;
    \item If $M$ is the solid torus, a fundamental normal disc and a curve in $\del M$ intersecting it once;
    \item If $M$ is $T^2\times I$, a normal annulus $F$ of edge weight at most exponential in a polynomial of $\norm{\mathcal{T}}$, a triangulation of $M\backslash\backslash F$ constructed using the algorithm in Lemma~\ref{lemma:cutalongsurfacestriangulation}, and a fundamental normal essential disc in this triangulation;
    \item If $M$ is $K\ttimes I$ (which we consider to have Seifert structure $S\ttimes S^1$, where $S$ is the M\"obius band), a fundamental annulus $F$, and:
        \begin{enumerate}
            \item If the fundamental annulus is horizontal, the triangulation of $M\backslash\backslash F\cong S^1\times D^2\sqcup S^1\times D^2$ from Lemma~\ref{lemma:cutalongsurfacestriangulation} and a fundamental normal essential disc in each component of the result;
            \item If the fundamental annulus is vertical, the triangulation of $M\backslash\backslash F\cong S^1\times D^2$ from Lemma~\ref{lemma:cutalongsurfacestriangulation} and a fundamental normal essential disc in it;
        \end{enumerate}
    \item Otherwise:
        \begin{enumerate}
            \item  a collection of non-isotopic normal vertical annuli $F$, of total edge weight at most exponential in $\norm{\mathcal{T}}^2$, whose complement is a solid torus;
            \item the triangulation of $M\backslash\backslash F$ from Lemma~\ref{lemma:cutalongsurfacestriangulation};
            \item a fundamental normal essential disc in this triangulation of $M\backslash\backslash F$;
            \item a fundamental normal section (that is, as $M\cong \Sigma\pmttimes S^1$, a normal surface in $M$ isotopic to $\Sigma\times\{*\}$).
        \end{enumerate}
\end{enumerate}

\setcounter{claimcounter}{0}
\begin{claim}
    The data of this certificate exists and has size bounded by a polynomial in $\norm{\mathcal{T}}$.
\end{claim}

\begin{claimproof}
    When $M$ is a solid torus, the data exists by Lemma~\ref{lemma:solidtorusdata}.
    
    When $M$ is $T^2\times I$, by Lemma~\ref{lemma:oneessentialannulus} there is a normal annulus $F$ of the required size.
    As the edge weight of $F$ is at most exponential in $\norm{\mathcal{T}}$, Lemma~\ref{lemma:cutalongsurfacestriangulation} allows us to construct a triangulation of $M\backslash\backslash F$ of polynomial size in $\norm{\mathcal{T}}$.
    By Lemma~\ref{lemma:solidtorusdata} then there is a fundamental normal essential disc in this new triangulation.

    When $M$ is $K\ttimes I$, by Lemma~\ref{lemma:oneessentialannulus} there is a fundamental normal essential annulus, which (as discussed in the proof of Proposition~\ref{prop:verticalannulicollectionhandlestructure}) is either horizontal and cuts $M$ into two solid tori, or is vertical and cuts it into one solid torus.
    Either way we can use Lemma~\ref{lemma:cutalongsurfacestriangulation} to construct a triangulation of its complement that is of polynomial size in $\norm{\mathcal{T}}$, and by Lemma~\ref{lemma:solidtorusdata} there is a fundamental normal essential disc in each component of this new triangulation.

    In the general case, by Corollary~\ref{corr:verticalannulicollectiontriangulation} there is a collection of vertical essential annuli $C$ in $\mathcal{T}$, of total edge weight at most $c_T^{\norm{\mathcal{T}}^2}$, whose complement is a single solid torus.
    By Lemma~\ref{lemma:cutalongsurfacestriangulation} we can construct a triangulation of their complement, recording the boundary of $C$, that is of size at most polynomial in $\norm{\mathcal{T}}$.
    By Lemma~\ref{lemma:solidtorusdata} there is a fundamental normal essential disc in this new triangulation.
    By Proposition~\ref{prop:fundamentalhorizontalsurface} there is a fundamental section in $M$ whose normal vector is a bitstring whose size is bounded by a polynomial in $\norm{\mathcal{T}}$ (by Lemma~\ref{lemma:triangulationfundamentalsurfacebound}).
\end{claimproof}

\begin{claim}
    The homeomorphism type of $M$ can be verified in polynomial time.
\end{claim}

\begin{claimproof}
Check that $M$ has boundary; that is, that some face of the triangulation is not identified with any other.
Check that the given orientations of the tetrahedra are compatible, and thus certify that $M$ is orientable.

If $M$ is a solid torus, it is known by work of Ivanov that recognising it is in $\NP$~\cite{Ivanov}.
To certify the data from Lemma~\ref{lemma:solidtorusdata}, check that the given surface is in fact a disc using Proposition~\ref{prop:agolhassthurstonnormalsurface} then check (using Lemma~\ref{lemma:intersectionnumberpolynomialtime}) that the algebraic intersection number of the curve and the boundary of the disc, given as normal curves, is $\pm 1$.
We thus know that the curve and disc are essential and hence must be a fiber and meridian disc.

If $M$ is $T^2\times I$, certify that $F$ is an annulus using Proposition~\ref{prop:agolhassthurstonnormalsurface}.
Apply Lemma~\ref{lemma:cutalongsurfacestriangulation} to produce a triangulation of $M\backslash\backslash F$ and a record of the normal curves in the triangulation from $\del F$.
We have already seen that we can certify that $M\backslash\backslash F$ is a solid torus.
Compute the algebraic intersection numbers of the boundary curves from $F$ with the boundary of the meridian disc, which (as the meridian curve is from a fundamental normal surface so contains at most an exponential number of arcs in $\norm{\mathcal{T}}^2$) we can do in polynomial time by Lemma~\ref{lemma:intersectionnumberpolynomialtime}.
Verify that they are each $\pm 1$, so the core curve of $F$ in $M\backslash\backslash F$ is a longitude.
The mapping class group of the annulus up to non-$\del$-preserving isotopy has two elements: the class of the identity, and the class that exchanges the two boundary components, so $M$ is either (in the former case) $T^2\times I$ or (in the latter) $K\ttimes I$.
The first homology of $T^2\times I$ is $\ZZ^2$ and that of $K\ttimes I$ is $\ZZ\times\ZZ_2$, so it suffices to compute the homology of $M$, which, as the dimension of $M$ is fixed, can be done in polynomial time by work of Iliopoulos~\cite{ComputingSmithNormalForm} as explained in~\cite[Problem 33]{KaibelPfetsch}.

Suppose $M$ is $K\ttimes I \cong S\ttimes S^1$ where $S$ is the M\"obius band.
If the essential annulus $F$ is horizontal then it covers the M\"obius band, so as it has two boundary components it is a degree two cover and separates $M$ into two copies of $S\ttimes I$ (where it is the horizontal boundary of this bundle), which is the solid torus.
As in the $T^2 \times I$ case, we can check (using the certificate) that this fundamental surface separates $M$ into two solid tori.
We can use the normal curve vectors for the boundary of $F$ and the boundaries of the meridian discs to verify with Lemma~\ref{lemma:intersectionnumberpolynomialtime} that these curves have algebraic intersection number $\pm 2$.
Up to choice of coordinates there is only one curve on the boundary of the solid torus that intersects a meridian disc twice, so this is enough to certify that $\mathcal{T}$ is a triangulation of $K\ttimes I$.

If $F$ is vertical, it sits over a spanning arc in the M\"obius band and cuts $M$ into a single solid torus.
As in the $T^2\times I$ case, certify that $M\backslash\backslash F$ is a solid torus, and produce a normal vector for a meridian curve.
Compute the algebraic intersection numbers of the boundary of the two copies of $F$ with this meridian curve.
We find that they are $\pm 1$ and by the same homology computation as in the $T^2\times I$ case we can certify that $M$ is $K\ttimes I$.

We are left with the general case.
By Proposition~\ref{prop:agolhassthurstonnormalsurface} we can quickly check that each surface in $F$ is an annulus and compute the number of components $n$ in $F$.
As we have already described, we can certify that $M\backslash\backslash F$ is a solid torus and that the given surface is indeed a meridian disc, and record the corresponding meridian curve.
Compute the intersection number of each annulus boundary and this meridian curve (using Lemma~\ref{lemma:intersectionnumberpolynomialtime}) to certify that the annulus core curves are longitudes in the solid torus.
Now, when we glue up the annuli we will get a circle bundle over a surface $\Phi$.
This surface has Euler characteristic $1-n$.
Note that as $M$ is not a circle bundle over a disc, annulus or M\"obius band, $1-n$ is negative.

Let $\Sigma$ be the normal surface that we claim is a normal section.
Verify (using Proposition~\ref{prop:agolhassthurstonnormalsurface}) that $\chi(\Sigma) = 1-n$.
Compute algebraic intersection numbers (using Lemma~\ref{lemma:intersectionnumberpolynomialtime}) to certify that the boundary curve of $\Sigma$ intersects each boundary curve of each annulus in $A$ exactly once, which implies that $\Sigma$ intersects each annulus in one spanning arc and possibly some trivial curves.
This shows that, after compressions and $\del$-compressions (which increase Euler characteristic), $\Sigma$ is horizontal; as $\chi(\Phi) < 0$ and $\chi(\Sigma) = \chi(\Phi)$, $\Sigma$ must be a degree one horizontal surface: that is, a section.
\end{claimproof}

For the collection of fibers, on each boundary component of $M$, take one boundary component of one of the essential vertical annuli.
\end{proof}

\section{Recognising Seifert fibered spaces with boundary is in NP}
\label{section:SFSrecognition}

\begin{prop}
Deciding if a 3-manifold $M$ with triangulation $\mathcal{T}$ is a Seifert fibered space with non-empty boundary and (a non-zero number of) singular fibers of only multiplicity two, other than $S^1\times D^2$, $T^2\times I$ and $K\ttimes I$, is in \NP.
When $M$ is such a Seifert fibered space, deciding if it admits a certain set of Seifert data is in \NPcoNP{}, and there is a degree two normal horizontal surface and one normal fiber on each boundary component such that these properties can be certified in polynomial time.
\label{prop:recognisemultiplicitytwo}
\end{prop}

\begin{proof}
The data given in the certificate is the following:
\begin{enumerate}
    \item the Seifert data of $M$;
    \item a collection of disjoint normal essential vertical annuli $C$ in $\mathcal{T}$, of total edge weight at most $c^{\norm{\mathcal{T}}^2}$, for a fixed constant $c$, where the complement of $C$ is the union of a solid torus neighbourhood of each singular fiber and a circle bundle over a surface;
    \item \label{enum:mult21} a triangulation of $\mathcal{T}\backslash\backslash C$ from Lemma~\ref{lemma:cutalongsurfacestriangulation}, with a record of the normal curves from the boundaries of the components of $C$;
    \item the data from the certificate in Proposition~\ref{prop:circlebundlerecognitionNP} for $\mathcal{T}\backslash\backslash C$;
    \item \label{enum:mult22} a fundamental normal essential disc in each solid torus component of $\mathcal{T}\backslash\backslash C$;
    \item a degree two horizontal fundamental normal surface $F$ in $\mathcal{T}$;
\end{enumerate}

\setcounter{claimcounter}{0}
\begin{claim}
    The data of this certificate exists and is of size at most polynomial in $\norm{\mathcal{T}}$.
\end{claim}

\begin{claimproof}
Corollary~\ref{corr:verticalannulicollectiontriangulation} gives us the collection of normal vertical annuli.
Use Lemma~\ref{lemma:cutalongsurfacestriangulation} to build the required triangulation, and then by Lemma~\ref{lemma:solidtorusdata} we can find a fundamental normal essential disc in each solid torus component of this triangulation.
Proposition~\ref{prop:circlebundlerecognitionNP} gives us the certificate for $\mathcal{T}\backslash\backslash C$.
By Proposition~\ref{prop:fundamentalhorizontalsurface}, the desired degree two horizontal fundamental normal surface exists, and by Lemma~\ref{lemma:triangulationfundamentalsurfacebound} its normal vector is a bitstring of length bounded by a polynomial in $\norm{\mathcal{T}}$.
\end{claimproof}

\begin{claim}
    The homeomorphism type of $M$ can be verified from the certificate in polynomial time.
\end{claim}

\begin{claimproof}
As in Proposition~\ref{prop:circlebundlerecognitionNP}, certify that $M$ has boundary and is orientable.

Build the triangulation of $\mathcal{T}\backslash\backslash C$ using Lemma~\ref{lemma:cutalongsurfacestriangulation}, verifying that it agrees with the one in the certificate.
We can then certify that all but one of the resulting pieces are solid tori by work of Ivanov~\cite{Ivanov}, and certify the homeomorphism type of the remaining piece $M'$ using Proposition~\ref{prop:circlebundlerecognitionNP}.
By computing algebraic intersection numbers using Lemma~\ref{lemma:intersectionnumberpolynomialtime}, we can check that for each of the solid torus meridian discs, each annulus (algebraically) intersects it twice or not at all. 
We thus know that $M$ is a copy of $M'$ with solid tori glued on by gluing a $(1, 2)$ slope curve in the boundary of the solid torus to a fiber of $M'$; this is enough to certify that these solid tori are neighbourhoods of singular fibers of multiplicity two.
\end{claimproof}

We know that the boundary curves of the annuli are fibers by construction, so take one on each boundary component.
It remains to certify that $F$ is a degree two horizontal surface.
Check that the algebraic intersection number of $\del F$ with each annulus is $\pm 2$ using Lemma~\ref{lemma:intersectionnumberpolynomialtime}, and deduce that $F$ compresses and $\del$-compresses to a degree two horizontal surface by Lemma~\ref{lemma:nonorientablesurfacesinSFS}, so $\chi(F) \leq 2\chi(\Sigma)-n$.
Now check that $\chi(F) = 2\chi(\Sigma)-n$, so we could not have done any compressions or boundary compressions since they increase Euler characteristic.
\end{proof}

We are now almost ready to prove these recognition results for general Seifert fibered spaces with boundary: that is, when we have singular fibers of multiplicities other than two.
We need two more results: a theorem from previous work of the author and a result about computing singular fiber data in Seifert fibered spaces.

\begin{theorem}[Theorem 1.2~\cite{JacksonTriangulationComplexity}]
    \label{thm:simplicialfibers}
    Let $M$ be a Seifert fibered space with non-empty boundary and let $\mathcal{T}$ be a triangulation of $M$.
    The collection of singular fibers of $M$ that are not of multiplicity two have disjoint simplicial representatives in $\mathcal{T}^{(79)}$, the 79\textsuperscript{th} barycentric subdivision of $\mathcal{T}$.
    In $\mathcal{T}^{(82)}$, these simplicial singular fibers have disjoint simplicial solid torus neighbourhoods such that there is a simplicial meridian curve of length at most 48 for each such neighbourhood.
\end{theorem}

\begin{lemma}
\label{lemma:singularfiberdata}
    Let $M$ be an oriented Seifert fibered space with $n$ of its singular fibers drilled out.
    Let $F$ be a degree $k$ horizontal surface in $M$, and let $\eta_i$, $1\leq i \leq n$ be the collection of its curves of intersection with the boundary component arising from drilling out the $i$\textsuperscript{th} singular fiber.
    Let $\gamma_i$, $1\leq i \leq n$, be a regular fiber on each of these boundary components.
    Pick orientations of $\eta_i$ and $\gamma_i$ such that the algebraic intersection number $i(\eta_i, \gamma_i)$ is positive with respect to the orientation of $\del M$ induced by the orientation of $M$.
    Letting $\mu_i$ be a meridian of the drilled out solid torus neighbourhood of the $i$\textsuperscript{th} singular fibre, the Seifert data of this singular fiber $q/p$ (with respect to this orientation of $M$ and the basis of $H_1(\del M, \ZZ)$ induced by $[\frac{1}{k}\eta_i]$ and $\gamma_i$ for each $i$) is $(i(\eta_i, \mu_i)/k)/i(\gamma_i,\mu_i)$.
\end{lemma}

\begin{proof}
    Work in the boundary component from drilling out the $i$\textsuperscript{th} singular fibre.
    If we flip the orientations of $\eta_i$ and $\gamma_i$ simultaneously, this ratio of intersection numbers does not change sign, so choosing $i(\eta_i, \gamma_i)$ to be positive is sufficient to determine it.
    Note that $\eta_i$ is $k$ copies of a curve, and $(\frac{1}{k}\eta_i$, $\gamma_i)$ is a positive basis for the homology of this torus with its induced orientation.
    By the definition of the construction of Seifert fibered spaces (for example, see~\cite[Defn.\ 10.3.1]{Martelli}), as $\gamma_i$ is a regular fiber and $\frac{1}{k}\eta_i$ intersects it once, a $(p_i, q_i)$ singular fiber is when the meridian curve is $\frac{p_i}{k}\eta_i + q_i\gamma_i$.
\end{proof}

Recall the results we wish to prove.

\SFSrecognitionNP*

\SFSnamingNP*

\begin{proof}[Proof of Theorems~\ref{thm:SFSrecognitionNP} and~\ref{thm:SFSnamingNP}]
Let $M$ be a Seifert fibered space with non-empty boundary.
Our certificate will be of the following form.
First, it will be the Seifert data of $M$.
If $M$ is a circle bundle over a surface or a Seifert fibered space with only multiplicity two singular fibers, it will be the certificate from Proposition~\ref{prop:circlebundlerecognitionNP} or Proposition~\ref{prop:recognisemultiplicitytwo}, respectively.
Otherwise, it will be:
\begin{enumerate}
    \item a compatible orientation of each tetrahedron of $\mathcal{T}$;
    \item the triangulation $\mathcal{T}^{(82)}$, constructed by subdividing each tetrahedron of $\mathcal{T}$ in the order given;
    \item \label{enum:removesingularfibers} the non-multiplicity-two singular fibers in $\mathcal{T}^{(82)}$, a solid torus neighbourhood of each one, a meridian disc for it with boundary of length at most 48, a longitude curve in its boundary that intersects the meridian once, a triangulation $\mathcal{T}'$ of $\mathcal{T}^{(82)}$ with these neighbourhoods of singular fibers removed, meridian curves of length at most 48 marked, and a compatible choice of orientations of the tetrahedra of $\mathcal{T}'$;
    \item if there are singular fibers of multiplicity two, the certificate from  Proposition~\ref{prop:recognisemultiplicitytwo} for $\mathcal{T}'$, or otherwise, the certificate from Proposition~\ref{prop:circlebundlerecognitionNP} for $\mathcal{T}'$.
\end{enumerate}

It is straightforward to see that the certificate exists, as giving a triangulation of $\mathcal{T}^{(82)}$ is constructive, and Theorem~\ref{thm:simplicialfibers} gives us the required singular fiber neighbourhoods.

To verify the certificate, first, as discussed in the proof of Proposition~\ref{prop:circlebundlerecognitionNP}, check that $M$ is orientable and has boundary.
One definition of Seifert fibered spaces is that they are the orientable manifolds built as follows.
Take a circle bundle over a surface with boundary, which always has toroidal boundary, and glue solid tori into some of the boundary components so that the meridian of each solid torus is glued to some essential curve in the boundary that is \emph{not} a fibre of the circle bundle; that is, to some $q/p$ slope with $p\not= 0$.
To show that $M$ is a Seifert fibered space, we will show that it can be constructed in this manner. 

The first step is thus to check that, in the given Seifert data, for each singular fibre fraction $q/p$, that $p$ is not zero and hence that this is a valid set of data for a Seifert fibered space.

Construct $\mathcal{T}^{(82)}$ by barycentrically subdividing the tetrahedra in order and verify that it agrees with the given triangulation.
Certify that the removed regions are solid tori using Proposition~\ref{prop:circlebundlerecognitionNP} and use Proposition~\ref{prop:agolhassthurstonnormalsurface} to verify that the given meridian discs are indeed discs.
Check that the longitude intersects the meridian curve once for each singular fiber neighbourhood using Lemma~\ref{lemma:intersectionnumberpolynomialtime}, thus certifying that the given meridian discs are essential.

Note that $\mathcal{T}'\not\cong K\ttimes I$ as $M'$ has more than one boundary component.
Also, $\mathcal{T}'\not\cong T^2\times I$ as the only Seifert fibered structure for $T^2\times I$ is as $A\times S^1$, where $A$ is the annulus, but then we can only have removed one singular fiber, so $M$ was a solid torus to begin with.
As we have the data of Proposition~\ref{prop:circlebundlerecognitionNP} or Proposition~\ref{prop:recognisemultiplicitytwo} to certify the Seifert data of $\mathcal{T}'$, since $\mathcal{T}' \not\cong K\ttimes I$ or $T^2\times I$, take the normal horizontal surface and complete collection of fibers in the boundary contained in this certificate.
Note that the edge weight of this surface and annulus is at most $c_T^{\norm{\mathcal{T'}}^2}$, where $c_T$ is at most $10^{10^{36}}$.
In each boundary component of $M'$ that bounds a singular fiber, consider the boundary of the horizontal surface, $\eta$, and the boundary of a normal vertical annulus  fiber, $\gamma$, each of whose length is at most $c_T^{\norm{\mathcal{T}}^2}$.
Push the simplicial meridian $\nu$ off the 1-skeleton to get a normal meridian $\mu$.
The length of $\mu$ is at most the length of $\nu$ plus the total valence of the vertices of this boundary torus: that is, at most three times the number of edges in this boundary torus, which is at most $3\cdot 6\cdot 24^{82} \norm{\mathcal{T}}$.

Check that the given orientation of the tetrahedra of $\mathcal{T}'$ is in fact consistent and compute the orientation of each boundary triangle induced by it.
By Lemma~\ref{lemma:intersectionnumberpolynomialtime}, we can arbitrarily orient each of these three curves and then compute $i(\eta, \gamma)$ in polynomial time in the original input, with respect to the induced orientation of the boundary triangles.
Set the orientation of $\eta$ such that $i(\eta, \gamma)$ is positive and then compute $q' = i(\mu, \eta)$ and $p = i(\mu, \gamma)$.
If $\mathcal{T}'$ had multiplicity two fibers, $\eta$ is two copies of a horizontal curve and intersects $\gamma$ twice, so in any case, set $q = \frac{q'}{i(\eta, \gamma)}$.
By Lemma~\ref{lemma:singularfiberdata} this singular fiber has Seifert data $q/p$.
With the certificate for $M'$, this certifies the homeomorphism type of $M$, and we can check that this matches the given Seifert data.
Finally, we can verify if $M$ is homeomorphic to some other set of Seifert data in polynomial time using Lemma~\ref{lemma:SFShomeopolytime}.
\end{proof}

\appendix

\section{Normal surfaces in \ths{}s}
\label{appendix:normalsurfacesinths}

\subsection{Normalisation}
\begin{lemma}
    \label{lemma:incompressibleinsurfacetimesI}
    Let $S$ be an incompressible (connected) surface in the orientable $I$-bundle $\Sigma\pmttimes I$, other than a trivial disc or sphere, that is disjoint from the horizontal boundary and does not admit any $\del$-compression discs with respect to the vertical boundary of $\Sigma\pmttimes I$.
    Then $S$ is isotopic to $\Sigma\times \{*\}$ or its double cover.
\end{lemma}

\begin{proof}
    First, suppose that $\Sigma$ is a disc.
    Then as $\Sigma\times I$ is a ball (so $S$ is two-sided) and $S$ is incompressible and hence $\pi_1$-injective, $S$ is a sphere or a disc.
    If it is a sphere, it is a trivial one.
    If $\del S$ contains a trivial curve in the vertical boundary, then $S$ must be a boundary-parallel disc.
    Thus $S$ is a disc and its boundary is essential in the vertical boundary, so it is isotopic to $\Sigma\times\{*\}$.

    For the general case, suppose that $\Sigma$ has boundary.
    Decompose $\Sigma$ into one 0-handle, which we view as a polygon, and some number of 1-handles, which we view as rectangles.
    Above each edge of the 0-handle and 1-handles in $\Sigma\pmttimes I$ is a quadrilateral with two boundary edges in the vertical boundary of $\Sigma\pmttimes I$ and two in the horizontal boundary.
    Isotope $S$ so that it is transverse to these quadrilaterals and thus intersects each quadrilateral in a collection of arcs and closed curves.
    As $S$ is incompressible, each of these closed curves bounds a disc in $S$, so by minimising the number of components of intersection between $S$ and the quadrilaterals we can assume that the intersection is only of arcs.
    Note that $S$ is disjoint from the horizontal boundary, so each of these arcs starts and ends on either the same or different vertical boundary arcs of a quadrilateral.
    If one starts and ends on the same vertical boundary arc, by taking an outermost such arc we can obtain a $\del$-compression disc for $S$ with respect to the vertical boundary, so by an isotopy we can remove this arc.
    Thus up to isotopy $S$ intersects each quadrilateral in arcs that are transverse to the product structure.
    
    Consider one of the handles of $\Sigma$, $H$, which is a disc.
    Consider the intersection of $S$ with $H\pmttimes I \cong D^2\times I$ in $\Sigma\pmttimes I$.
    We can assume (by minimising the number of components of intersection between $S$ and the quadrilaterals) that $S$ is incompressible in this copy of $D^2\times I$, so as we have already discussed, each component of it is a (trivial) sphere (which we have ruled out) or disc.
    The boundary of $H\pmttimes I$ has three parts: its horizontal boundary, from which $S$ is disjoint, quadrilaterals, and pieces of $\del_v(\Sigma\pmttimes I)$.
    For similar reasoning as with the quadrilaterals, up to isotopy $S$ also intersects the vertical boundary pieces in arcs that are transverse to the product structure.
    Thus the boundary of each disc of $S\cap H\pmttimes I$ is an essential curve in $\del_v (H\pmttimes I)$.
    Thus $S$ intersects $H\pmttimes I$ in a collection of horizontal discs for each $H$, and so intersects all of $\Sigma\pmttimes I$ in either $\Sigma\times\{*\}$ or (if $\Sigma$ is nonorientable) possibly $\Sigma\ttimes S^0$.

    If $\Sigma$ is closed, let $D$ be a disc of $\Sigma$.
    Isotope $S$ to minimise the number of components of $S\cap \del (D\times I)$.
    The intersection of $S$ with $D\times I$ is incompressible since otherwise, as $S$ is incompressible, it would not be minimal.
    For the same reason $S\cap (D\times I)$ contains no spheres or trivial discs.
    By the first part of this proof, $S\cap (D\times I)$ is a collection of discs of the form $D\times \{*\}$.
    If $S\cap (\Sigma - D)\times I$ admits a $\del$-compression disc with respect to the vertical boundary, we can use it to isotope $S$ to reduce $|S\cap \del (D\times I)|$.
    Thus by the previous part of the proof, $S\cap (\Sigma-D)\times I$ is isotopic to $(\Sigma-D)\times \{*\}$ or its double cover, which gives us the result.
\end{proof}

We wish to modify an arbitrary surface $F$ in $M$, whose boundary is disjoint from the forbidden region, by a series of admissible isotopies and normalisation moves such that the result is normal (but may not be isotopic to $F$).
Compare the following procedure to the proof of Theorem 3.4.7 in~\cite{Matveev}.
The only substantial difference is in Move 2, as we need to consider parallelity pieces that are more complicated than 2-handles.

\begin{procedure}[Normalisation]
\label{procedure:normalisation}
Let $F$ be a properly-embedded surface in $M$ that is disjoint from the forbidden region.
The \emph{weight} of $F$ is $w(F) = (p(F), b(F), |F\cap \del M|)$, which we will sort lexicographically, where $p(F)$, the \emph{plate degree}, is $|F\cap (\del\mathcal{H}^{2} \cup \del\mathcal{H^P})|$ and $b(F)$, the \emph{beam degree}, is $|F\cap \mathcal{H}^{1}|$.
We will see in Proposition~\ref{prop:normalisationterminates} that almost all the \emph{normalisation moves} which follow will reduce $w(F)$.
All isotopies in these moves are required to be admissible.

The \emph{normalisation procedure} is to perform Move 1 once, and then repeat Moves 2-7 in sequence as long as possible.

\textsc{Move 1.} Note that the boundary of the forbidden region $\mathcal{I}$ is disjoint from the boundaries of the 2-handles.
(Admissibly) isotope $F$ so that $F$ is transverse to the handle structure and $\del F$ is disjoint from the horizontal boundaries of the 2-handles.
Each 3-handle contains an open ball that is disjoint from $F$, so by expanding these balls, we can isotope $F$ to be disjoint from the 3-handles.
Discard any components of $F$ that are entirely contained in a parallelity piece.

\textsc{Move 2.} Consider each component $H\cong \Sigma\pmttimes I$ of $\mathcal{H}^{2}\cup \mathcal{H^P}$.
If any of the components of $F\cap \del H$ are trivial curves in the vertical boundary, compress $F$ along the discs in $\del H$ that these curves bound and isotope this part of $F$ off $\del H$.
Similarly, if $F\cap H$ admits a compression disc, compress $F$ along it.
If $\del_v H$ intersects $\del M$ (in which case $H$ is a parallelity piece), and $F\cap H$ admits a $\del$-compression disc $D$ that is disjoint from the horizontal boundary -- that is, where the arc of $\del D$ in $\del H$, $\alpha$, is in $\del_v H$ -- then compress along it.
Now, as each component of $F\cap \del H$ is an essential curve in $\del_v H$, up to an isotopy supported in a collar of $\del_v H$ within $H$, we can arrange that each of these curves is transverse to the induced $I$-bundle structure on $\del_v H$.

If we performed a $\del$-compression, repeat the first two steps of this move.
Finally, discard any components of $F\cap (\mathcal{H}^{2}\cup \mathcal{H^P})$ that are spheres or $\del_vH$-parallel discs.

\textsc{Move 3.} For each 1-handle $D^1\times D^2$, take a disc $D = \{*\}\times D^2$, transverse to $F$, that minimises the number of components in $D\cap F$.
We can blow a regular neighbourhood of $D$ out to be the whole 1-handle.

\textsc{Move 4.}
If any component of intersection of $F$ with a 1-handle is a tube $S^1\times I$, compress it and isotope the two resulting discs out of the 1-handle.
Similarly, if any component of intersection with a 1-handle $D^1\times D^2$ is a disc whose intersection with $D^1\times\del D^2$ is contained in a single region of $(\del M-\mathcal{I})\cap (D^1\times D^2)$, we can $\del$-compress this tunnel and isotope the pieces out of the 1-handle.

\textsc{Move 5.}
Compress any compressible pieces of $F\cap \mathcal{H}^{0}$ and discard any trivial spheres in the 0-handles.

\textsc{Move 6.}
If $F$ intersects a lake in a trivial curve, compress it and throw away the resulting $\del$-parallel disc.
If $F$ intersects a lake in an arc that starts and ends on the same component of the intersection of the lake with an island, by an isotopy (again as $\del\mathcal{I}$ intersects the cell structure on the boundary in a normal curve) push this piece $F$ off the lake and through the 1-handle.

\textsc{Move 7.} If a disc of $F\cap \mathcal{H}^{0}$ crosses a bridge twice or a bridge and an adjacent lake, isotope the portion of disk between them into the bridge.
If it crosses a lake twice, $\del$-compress the resulting tunnel along both of its intersections with the boundary of the lake. 
\end{procedure}

\begin{prop}
\label{prop:normalisationterminates}
    Let $F$ be a properly-embedded surface in a \ths{} $\mathcal{H}$.
    Applying Procedure~\ref{procedure:normalisation} terminates and the result (if non-empty) is a normal surface.
\end{prop}

\begin{proof}

After Move 1 $F$ satisfies condition~\ref{defn:normalsurface:itm:3handles} of Definition~\ref{defn:normalsurface}.
If Move 2 has no effect, $F\cap H$ is incompressible and $\del$-incompressible with respect to the vertical boundary, so by Lemma~\ref{lemma:incompressibleinsurfacetimesI} $F$ satisfies condition~\ref{defn:normalsurface:itm:2handles} of Definition~\ref{defn:normalsurface}.
If Moves 3 and 4 have no effect then it satisfies condition~\ref{defn:normalsurface:itm:1handles} of Definition~\ref{defn:normalsurface}.
After Move 5 it intersects each 0-handle in discs, which is the first part of condition~\ref{defn:normalsurface:itm:0handles} of Definition~\ref{defn:normalsurface}.
Then Move 6 ensures condition~\ref{defn:normalsurface:itm:lakes} of Definition~\ref{defn:normalsurface} and Move 7 ensures condition~\ref{defn:normalsurface:itm:0handles}.
As a result, if no more moves can be performed, then $F$ is normal.

Let the normalisation complexity of $F$ be $$(p(F), b(F), |F\cap\, \del M|, \gamma(F), \eta(F), n(F)),$$ ordered lexicographically, where the new terms are $\gamma(F) = \sum_{i=1}^m(1-\chi(F_i))$ where $\{F_i\}$ is the collection of connected components of $F\cap\mathcal{H}^{0}$ that are not spheres, $\eta(F) = \sum_{i=1}^m(1-\chi(F_j))$ where $\{F_j\}$ is the collection of connected components of $F\cap(\mathcal{H}^{2}\cup \mathcal{H}^{\mathcal{P}})$ that are not spheres, and $n(F)$ is the number of connected components of $F$.
We will show that each of Moves 2-7 reduces the normalisation complexity of $F$.

For Move 2, suppose we perform a $\del$-compression on $F\cap H$.
As $\mathcal{H}$ is a \ths{}, if $\del \mathcal{I}$ intersects a component of $\del_v H\cap \del M$ then it does so in two curves or arcs.
Thus this component is either contained in $\mathcal{I}$, or (as $\del_h H$ is contained in $\mathcal{I}$) it intersects $\mathcal{I}$ in a neighbourhood of its boundary with the horizontal boundary of $H$.
If $F\cap H$ admits a $\del$-compression disc $D$ that is disjoint from the horizontal boundary, then let $\bar{\alpha}$ be a properly-embedded arc in $\del_v H$, such that $\alpha$ is a subarc of it, which is isotopic in $\del_v H$ to one of the $I$-fibers of $\del_v H$ and has minimal intersection with each of the components of $F\cap \del_v H$ -- that is, it intersects each of the essential curves in this collection once.
This is possible since $\alpha$ runs between two different components of $F\cap\del_v H$; otherwise we would be able to upgrade our $\del$-compression disc to a compression disc.
Note that $|\bar{\alpha}\cap F|$ is the degree of the projection map from $F\cap H$ to $\Sigma$.
(When $\Sigma$ is orientable, this is the number of sheets of $F$ in $H$.)

The effect of the $\del$-compression on $F\cap \del H$ is to remove from it $\del \alpha \times I$ and add to it $\alpha \times \del I$, for some small thickening $\alpha\times I$ of $\alpha$.
This reduces $|\bar{\alpha}\cap F|$ by two, and thus we see that the number of essential curves in $F\cap \del_v H$ has reduced by two, so either $p(F)$ has reduced or we reduce it in the next step of Move 2.
After this move, $F\cap H$ does not admit any $\del$-compression discs with respect to the vertical boundary of $H$.

Compressing $F$ along trivial curves in the vertical boundary of $H$ reduces $p(F)$.
Compressing $F\cap H$ reduces $\eta(F)$ and fixes the earlier terms in the normalisation complexity.
We can thus see that we either reduce $p(F)$ or reduce $\eta(F)$ and fix the earlier terms in the normalisation complexity.

If Move 3 is non-trivial then it does not increase $p(F)$ and reduces $b(F)$, and the same applies for Move 4.
Move 5 does not change $p(F)$ or $b(F)$. If there are compressible pieces then it reduces $\gamma(F)$, and otherwise if all the pieces are 2-spheres or discs then it reduces $n(F)$ and fixes everything else.
Move 6 either reduces $|F\cap \del M|$ and does not change $p(F)$ and $b(F)$, or fixes $p(F)$ and reduces $b(F)$.
Move 7 either decreases $p(F)$ or fixes it and reduces $b(F)$.
As a result, the procedure terminates with a normal surface.
\end{proof}

\begin{defn}
    A properly-embedded disc in a 3-manifold $M$ is \emph{essential} if its boundary does not bound a disc in $\del M$.
\end{defn}

Note that this definition does not change if there is forbidden region.

\begin{lemma}
Suppose that $M$ is irreducible with a \ths{} $\mathcal{H}$ such that there is some essential disc in $M$ that avoids the forbidden region.
Then there is a normal essential disc, disjoint from the forbidden region.
\end{lemma}

\begin{proof}
Apply the normalisation procedure to this essential disc $D$.
As $M$ is irreducible and $D$ is (trivially) incompressible, whenever we compress $D$ in the normalisation procedure we will produce a surface admissibly isotopic to $D$ and a trivial sphere, which we can discard.
Whenever we $\del$-compress $D$ we will produce two discs, both of which are disjoint from the forbidden region.
At least one of them must be essential as $\del D$ does not bound a disc in $\del M$.
Discard the other.
\end{proof}

\cannormaliseincompressiblesurfaces

\begin{proof}
Apply the normalisation procedure to $F$.
 As $F$ is incompressible and $\del$-compressible and $M$ is irreducible and $\del$-irreducible, each time we ($\del$-)compress $F$ in the normalisation procedure we will produce two components: a surface that is admissibly isotopic to $F$ and a trivial sphere or disc.
Thus we produce a normal surface admissibly isotopic to $F$, as well as some collection of trivial spheres and discs, which we can discard. 
\end{proof}

\subsection{Subtetrahedral \ths{} combinatorics}

We give bounds on the number of elementary discs types.

\combinatorialconstraintssubtetrahedral

\begin{proof}
    Consider a semitetrahedral handle structure $\mathcal{H}$ and normal surface $F$ such that $H$ is homeomorphic to one of the pieces of $\mathcal{H}\backslash\backslash F$ from some 0-handle $H'$ in $\mathcal{H}$.
    We can think of forming the pieces of $H'$ in the induced \ths{} as occurring in two steps.
    First, we cut $H'$ along a collection of elementary discs.
    This has the effect on the boundary of $H'$, which we think of as a graph embedded in $S^2$, of cutting it along a collection of (separating) curves and filling these holes in with the forbidden region.
    Second, we possibly replace some of the 1-handles with parallelity pieces, in which case (as, if a 1-handle can be given a parallelity structure, so can any 2-handles it borders) the effect on the boundary graph is to merge a valence one or two island with the bridge(s) adjacent to it.

    Consider cutting along one elementary disc $D$ of $F$ in $H'$ at a time.
    As $\mathcal{H}$ is semitetrahedral, the islands of $H'$ are at most trivalent, so $\del D$ runs through each island at most once.
    As $\del D$ runs through each island or bridge at most once and is separating, the subgraph of islands and bridges after cutting along it is a subgraph of the complete graph on four vertices, so after cutting into pieces, there are between zero and four islands, each of which has valence at most three.
    We can rule out the case when the boundary graph is empty as, since the boundary graph of $H'$ is connected, the boundary of any elementary disc runs through at least one island, so there will be an island in each of the pieces of $H'$ that it separates.
    Turning 1-handles into parallelity handles does not increase the number of islands or their valence.
    If all of the 1-handles become parallelity pieces (so there were zero islands) then in fact the entire 0-handle will become a parallelity piece, so that is impossible.

    Let $G$ be the boundary graph of $H'$ (containing $b$ bridges) and $\mathcal{I}$ its forbidden region, where we assume that $G$ is connected and contains at most $12-2b$ sutures.
    The elementary disc boundary $\del D$ separates $\del H'$ into two components.
    Pick one of them, $C$, to consider.
    When we cut along $D$, the other component from $\del H'$ has one less bridge than $G$ for every bridge fully contained in $C$, one less suture for each suture in $C$, and one more suture for each arc of $D$ that runs through a lake.
    Thus for the count of sutures it suffices to prove that twice the number of bridges fully contained in $C$ plus the number of sutures of $H'$ in $C$ is at least the number of arcs in $\del D$ that run through a lake.
    Consider one of these arcs $\alpha$.
    It separates the lake into two components, one of which is in $C$, and each of which is bordered by at least one bridge or suture.
    If the component in $C$ is bordered by a bridge, assign the arc $\alpha$ to that bridge.
    Note that as the bridge is adjacent to at most two lakes, and $\del D$ crosses each lake only once, by doing this we will associate at most two arcs to the bridge.
    Also note that as the elementary disc does not run through an adjacent bridge and lake, this bridge is contained in $C$.
    Otherwise, this component in $C$ is bordered by at least one suture, as $\alpha$ runs between two different components of intersection of an island and a lake, so assign $\alpha$ to this suture, which again we see is contained in $C$.
    The forbidden region is on the other side of this suture so we will not assign any other arcs to it.
    We can thus see that the count is as we claimed.
    Turning 1-handles into parallelity regions only reduces the number of bridges further.

    For the final statement, fix an island in $\mathcal{H}$ of valence $v$ that intersects at most $6-2v$ sutures.
    Consider some elementary disc boundary running through this island.
    By the same procedure as for the total suture count, assign any new sutures that are created to a bridge or a suture that intersects the island to get the result: we start with three adjacent bridges, and add at most two sutures each time we remove one.
\end{proof}

Recall that $\norm{\mathcal{H}}$ is the number of 0-handles in $\mathcal{H}$.

\begin{lemma}
\label{lemma:bridgeslakesbound}
    The total number of bridges and lakes in a subtetrahedral trace 0-handle $H$ is at most 13.
\end{lemma}

\begin{proof}
    From Lemma~\ref{lemma:combinatorialconstraintssubtetrahedral}, each 0-handle contains at most 6 bridges and $12-2b$ sutures, where $b$ is the number of bridges.
    The number of lakes from the subgraph on islands and bridges alone is at most $b+1$ by an Euler characteristic argument.
    Each pair of sutures can increase the number of lakes by one, if they bound a forbidden region that borders a lake on both sides.
    (If a suture bounds a forbidden region that borders a bridge on one side then removing this region does not change the number of lakes.)
    Thus having sutures does not in fact increase the count, so that maximum count is $6 + 6 + 1 = 13$.
\end{proof}

\begin{lemma}
    \label{lemma:boundoncutopensize}
    If $F$ is a (possibly disconnected) non-\pns{} normal surface in a subtetrahedral \ths{} $\mathcal{H}$ then $\norm{\mathcal{H}\backslash\backslash F}$ is at most $2^{13}\norm{\mathcal{H}}$.
\end{lemma}

This is a substantial overestimate, but will do for our purposes.

\begin{proof}
    We know that $\mathcal{H}\backslash\backslash F$ is a \ths{} from Lemma~\ref{lemma:normalsurfacegivesths}.
    Define the complexity of a subtetrahedral 0-handle $H$, $c(H)$, to be the sum of the number of lakes of $H$ and the number of bridges.
    By Lemma~\ref{lemma:bridgeslakesbound} $c(H)$ is at most 13.
    Consider an elementary disc $D$ that separates $H$ into two pieces.
    After we cut along $D$ and take the induced handle structure, we obtain at most two 0-handles, $H_1$ and $H_2$.
    We show that $c(H_i)$ is at most $c(H)$ for each $i$, with equality only if the other piece is a parallelity piece.
    
    As $\del D$ does not run through a lake or bridge twice, cutting along an edge of $D$ cannot increase either of the terms in the complexity.
    The boundary of $D$ separates the boundary of $H$ and is not a trivial curve, so as it does not run through a bridge or lake twice or through a bridge and an adjacent lake, it must separate off at least one forbidden region, bridge, or lake on both sides.
    If it separates off exactly one forbidden region and no bridges or lakes on one side, then it is parallel to the boundary of the forbidden region so one of the $H_i$ is a parallelity handle.
    If it separates off more than one forbidden region then two of those forbidden regions had a lake or bridge between them, so it has reduced the number of lakes.
    If it separates off a bridge or lake on one side, then cutting along it reduces the number of bridges or lakes.

   Thus each time we cut a 0-handle $H$ into pieces, if we produce two pieces then they both have lower complexity than $H$, so we may obtain at most $2^{13}$ 0-handles in $\mathcal{H}\backslash\backslash F$.
\end{proof}

\begin{lemma}
\label{lemma:finiteelementarydisctypes}
    There are a finite number of subtetrahedral 0-handles up to homeomorphism, and each of these can contain a finite number of elementary disc types up to normal isotopy.
\end{lemma}

\begin{proof}
    That there are a finite number of types of subtetrahedral 0-handles follows from the combinatorial constraints of Lemma~\ref{lemma:combinatorialconstraintssubtetrahedral}.
    An elementary disc $D$ in such a 0-handle $H$ is determined by its boundary $\del D$, which is itself completely described (as the complement of the boundary graph is a collection of discs) by the ordered list of islands it runs through and, for each, which intersection with a bridge or lake it enters and leaves by.
    As it may run through each island, bridge or lake at most once, there are a finite number of possibilities.
\end{proof}

\begin{notation}
Write $h_H$ for the number of possible subtetrahedral 0-handle types, and write $d_H$ for the maximum possible number of elementary disc types up to normal isotopy in a subtetrahedral 0-handle.
\end{notation}

\begin{remark}
\label{remark:boundondH}
By Lemma~\ref{lemma:bridgeslakesbound}, each 0-handle has a total of at most 13 lakes and bridges, so, by labelling each disc type by the number of edges in it and then the order of them, $d_H$ is at most $13\cdot 13!$.
This bound is far from sharp; enumerating by hand shows that the true number for a tetrahedral 0-handle (whose boundary graph is the complete graph on four vertices) is 59 elementary disc types and that this is maximal (see Appendix A of the author's thesis~\cite{JacksonThesis}).
\end{remark}

Note that there is a natural inclusion map from surfaces in $M\backslash\backslash F$, disjoint from the forbidden region, to surfaces in $M$, which takes normal surfaces to normal surfaces.

\subsection{Summands of incompressible normal surfaces}
In this section all \ths{}s are subtetrahedral.
In Matveev's excellent book~\cite{Matveev}, he carefully lays out the foundations of normal surface theory in the triangulation setting.
Although we are in a more general setting, it is similar enough for the ideas in the proofs he gives to apply.
See the proof of Theorem 4.1.36 in~\cite{Matveev} for his treatment of these results in the triangulation setting, drawing on work of Haken, Jaco, Oertel and Tollefson, among others.
The analogues of the structure he relies on are Lemmas~\ref{lemma:irregularswitchreducestrace} and~\ref{lemma:nonzerobeamdegree}, which show respectively that performing an irregular switch reduces the weight of a normal surface sum, and that no normal surface has zero weight.
Recall the definition of the sum of normal surfaces in \ths{}s in Procedure~\ref{procedure:normalsurfacesum} and the definition of the weight of a normal surface in Definition~\ref{defn:weight}.

\begin{defn}
    Suppose that $F = G_1 + G_2$.
    Let $\gamma$ be the collection of double curves of $G_1\cup G_2$.
    A \emph{patch} of $G_1\cup G_2$ is a component of $G_1\cup G_2 - \gamma$.
\end{defn}

\begin{defn}
    A surface $F$ is \emph{minimal} if it is of minimal weight in its admissible isotopy class.
\end{defn}

\incompressiblenormalsurface

To set up, isotope $G_1$ and $G_2$ to minimise the number of components in $G_1\cap G_2$ subject to the requirement that $F = G_1 + G_2$, so that $G_1\cup G_2$ is in \emph{reduced form}.
Note that as $M$ is irreducible and $\del$-irreducible, no component of $F$ is a disc or sphere, as then we could isotope it to the interior of a 0-handle and so, in fact, there was no minimal normal surface representative of this surface.
A \emph{disc patch} is a patch, homeomorphic to a disc, whose boundary is either a closed curve in the interior of $M$, or an arc in the interior and an arc on $\del M$ (which we note must be disjoint from the forbidden region).

To prove Proposition~\ref{prop:incompressiblenormalsurface}, we first give a lemma that is analogous to Lemma 4.1.8 of~\cite{Matveev}.

\begin{lemma}
The patches of $G_1\cup G_2$ are incompressible and $\del$-incompressible, and none are disc patches.
\end{lemma}

\begin{proof}
First, we show there are no disc patches.
Suppose that a disc patch $E$ did exist, say in $G_1$.
Write $s$ for its double curve in $G_1\cup G_2$, and $s_1$ and $s_2$ for the corresponding trace curves in $F$, where $s_1$ bounds $E$ and $s_2$ is the other curve.
If the boundary of this patch is one-sided in $G_1$, then the connected component of $G_1$ containing $E$ is a normal projective plane $P$.
As $M$ is irreducible, it is $\RR P^3$ and $P$ is incompressible.
By Lemma~\ref{lemma:nonzerobeamdegree} all normal surfaces in $M$ have non-zero beam degree.
In particular, neither $b(G_1)$ nor $b(G_2)$ is zero, so $b(P)$ is strictly less than $b(F)$ and $p(P)$ is at most $p(F)$.
Now, $\RR P^3$ contains (up to isotopy) only one closed incompressible surface that does not contain a sphere: this projective plane.
Thus $F$ was not minimal.

Suppose that $s_2$ does \emph{not} bound a disc in $F$.
But then taking a parallel copy of $E$, bounded by $s_2$, we get a compression or $\del$-compression disc for $F$.

Suppose that $s_2$ bounds a disc $E'$ in $F$ that is opposite to $E$ at the intersection curve.
If $E'$ does not contain $E$, then $E\cup E'$ is a sphere that (as $M$ is irreducible) bounds a 3-ball, so $E$ and $E'$ are isotopic.
We can take the irregular switch at $s$ instead to obtain an isotopic surface to $F$ which by Lemma~\ref{lemma:irregularswitchreducestrace} is of lower weight than $F$, so $F$ was not minimal.

If $E'$ does contain $E$, again taking the irregular switch along $s$, we get, if $E$ was an interior disc, two surfaces $F_1\cup T$.
The surface $F_1$ is homeomorphic as a surface to $F$ and $\del F_1 = \del F$, while $T$ is formed by taking $E'\backslash E$ and identifying the boundary circles, and so is a torus or Klein bottle.
As the switch was irregular, at least one of $F_1$ and $T$ contains a return and hence after normalising has its weight decreased.
Let $F'$ be the result of normalising $F_1$.
If the return is in $F_1$ then the weight of $F'$ is less than that of $F$; if the return is in $T$ then the weight of $T$ is non-zero and so we can draw the same conclusion.
As, since $M$ is irreducible, $E'$ is isotopic to a parallel copy of $E$, $F$ and $F'$ are isotopic and so $F$ was not minimal.

If $E$ intersected the boundary, we can follow a similar line of reasoning. Now after an irregular switch we get $F_1\cup A$ where $A$ is an annulus or M\"obius band.
The same reasoning however still holds.

The remaining case in which there are disc patches is when every disc patch has an adjacent companion disc.
This is the case considered in Lemma 4.1.4 of~\cite{Matveev}, where Matveev shows the following result in the triangulation setting:

    Suppose that every disc patch of $G_1\cup G_2$ has an adjacent companion disc $E'$.
    Then either $F$ can be presented as $F = F_1 + T$ where $T$ has Euler characteristic zero and $F_1$ is admissibly isotopic to $F$ and of lower weight, or there is a disc patch $E$ whose adjacent companion disc is also a patch of $F$.

We give the idea of Matveev's proof, which also goes through in our setting.
It is to associate to $G_1\cup G_2$ an oriented graph, whose vertices are disc patches and there is an edge from $E_1$ to $E_2$ if $E_2$ is contained in the adjacent companion disc $E_1'$ of $E_1$.
As every vertex has at least one outgoing edge, this graph contains a cycle.
Resolve $G_1\cup G_2$ with the irregular switch at the edges of the patches in this cycle, and the regular switch otherwise.
We get a decomposition $F = F_1 + T$ where $T$ is the union of the pieces $E_i'-E_{i+1}$.
As by Lemma~\ref{lemma:nonzerobeamdegree} $T$ has non-zero weight, $F_1$ is of lower weight than $F$.
Note that $F$ is homeomorphic to $F_1$ as if we perform regular switches everywhere, then at $\del E_i$ we replace $E_i$ with $A_i\cup E_{i+1}$.
Now, $A_i\cup E_{i+1}\cup E_i$ is a disc or a sphere, so as $M$ is irreducible and $\del$-irreducible they separate off balls.
We can use this to construct the desired isotopy.

Now, as $F$ is minimal, there must be some disc patch $E$ whose adjacent companion disc $E'$ is also a patch.
Let $G_1'$ and $G_2'$ be the surfaces resulting from swapping $E$ and $E'$ within $G_1$ and $G_2$.
Now, $E\cup E'$ is either a sphere or a properly-embedded disc, so as $M$ is irreducible and $\del$-irreducible, we can isotope $G_1$ to $G_1'$ and $G_2$ to $G_2'$, reducing the number of intersections in $G_1\cup G_2$.
Thus $G_1 + G_2$ was not in reduced form, as we assumed.

We can therefore conclude that there are no disc patches. Suppose $D$ is a compression disc for some patch $P$, which we perturb to be transverse to $F$.
Let $\gamma$ be an innermost curve of $F\cap D$, which, as $F$ is incompressible, bounds a disc $D'$ in $F$.
If $P\sub D'$ then $P$ is planar and all but one of its boundary components bound discs in $F$.
But then either $P$ is a disc patch or the disc bounded by one of these boundary components contains a disc patch.
Thus $D'$ is disjoint from $P$ so we can therefore isotope $D$ to remove all intersections with $F$ other than its boundary.
Now, $\del D$ bounds a disc in $F$.
This disc must be contained in $P$ as otherwise some component of $\del P$ bounds a disc, and so there is a disc patch.
Thus $P$ is incompressible.

Suppose that $D$ is a $\del$-compression disc for $P$.
By the same reasoning, we can isotope $D$ to remove any curves of intersection with $F$.
If $\alpha$ is an outermost arc of $F\cap D$ other than the segment of $\del D$ in $P$, as $F$ is $\del$-incompressible, it bounds a disc in $F$ together with an arc in the boundary.
By the same reasoning, this disc is disjoint from $P$ so we can isotope $D$ to further remove these arcs.
Thus $D\cap P$ bounds a disc in $F$, which, as there are no disc patches, is contained in $P$.
\end{proof}

\begin{proof}[Proof of Proposition~\ref{prop:incompressiblenormalsurface}]
As any decomposition of a disc, sphere or projective plane has a disc patch, each component of $G_1$ and $G_2$ has nonpositive Euler characteristic.
If $G_1$ is compressible or $\del$-compressible, we will construct a nontrivial compression or $\del$-compression disc $D$ for $G_1$ that is disjoint from $G_2$.
Let the \emph{weight} of a compressing disc $D$ for $G_1$ be $w(D) = |D\cap G_2| + |\del D \cap (G_2\backslash\backslash \del M|$.
The idea is to give moves to reduce this weight.

\setcounter{claimcounter}{0}
\begin{claim}
Suppose that $G_1$ has a nontrivial compression or $\del$-compression disc $D$, which we can assume is transverse to $G_2$.
If $D\cap G_2$ contains a closed curve or an arc with both endpoints in $\del M$, then there is a compression or $\del$-compression disc $D'$ for $G_1$, of lower weight, that does not contain such intersections with $G_2$.
\end{claim}

The triangulation setting analogue of this claim is shown in the first part of Lemma 4.1.35 of~\cite{Matveev}.
We sketch the idea of the proof, which does not in fact depend on its setting.
Take an innermost such curve or outermost such arc, bounding a disc $\Delta$ in $D$ such that $\Delta \cap G_2 = \del \Delta$ or $\Delta\cap G_2$ is a single arc.
As $\del \Delta\cap G_2$ is connected, it is contained in a single patch of $G_2$.
Now, as the patches of $G_1\cup G_2$ are incompressible and $\del$-incompressible and $M$ is irreducible and $\del$-irreducible, there is a disc in this patch of $G_2$ isotopic to $\Delta$.
We can replace $\Delta$ in $D$ with this disc to remove this intersection with $G_2$ and reduce the weight of $D$.

Let $D$ be a nontrivial compression or $\del$-compression disc for $G_1$ that intersects $G_2$ minimally.
By the claim, $D$ intersects $G_2$ only in arcs that have at least one endpoint away from $\del M$.
Let $\Delta$ be a region of $D$ cut out by $D\cap G_2$.
Then $\Delta$ is a ($\del$)-compressing disc for the polyhedron $G_1\cup G_2$, as it is transverse to the singular subcomplex of $G_1\cup G_2$ and $D\cap \del M$ is connected.
Following~\cite[\S4.1.6]{Matveev}, label the vertices of $\Delta$ as ``good'' or ``bad'' angles according to whether, in $F$, the patches forming the two edges are pasted together or not respectively.
Note that if all angles of $\Delta$ are good, then $\Delta$ is a compression disc for $F$.

The following claim is Lemma 4.1.33 of~\cite{Matveev}, and his proof goes through in our setting.
\begin{claim}
Suppose that $\Delta$ has exactly one bad angle.
Then $F$ is not minimal.
\end{claim}

The idea of the proof is to consider the component $l$ of $G_1\cap G_2$ that passes through this bad angle.
Let $A$ be the annulus or quadrilateral joining the two copies of $l$ in $F$.
One can build a ($\del$-)compressing disc for $F$ from, if $l$ is a closed curve, two copies of $\Delta$ and one of the components of $A-\Delta$, or if $l$ is an arc, $\Delta$ and a component of $A-\Delta$.
Now as $F$ is incompressible and $\del$-incompressible and $M$ is irreducible and $\del$-irreducible, this disc bounds a ball with a disc of $F$ and possibly a disc in $\del M$.
We can use this to isotope $F$ through a region it bounds with $A$ to a surface of lower weight.

By this claim, each region has either no bad angles or at least two bad angles.
Now, each arc in $D\cap G_2$ contributes one bad angle for each endpoint it has that is not on $\del M$.
Let $m$ be the number of arcs that have no endpoints on $\del M$, and $n$ be the number with one endpoint on $\del M$.
There are $m+n+1$ regions in $D$, and $2m+n$ bad angles.
Thus, by the pigeonhole principle, there is a region $\Delta$ with at most one bad angle, which hence has no bad angles.

This region is a compression or $\del$-compression disc for $F$ itself, which is incompressible and $\del$-incompressible, so $\Delta\cap F$ bounds a disc $\Delta'$ in $F$.
We sketch the ideas of Lemma 4.1.34 of~\cite{Matveev}, which deals with this situation.
First, $\Delta'$ is decomposed by the trace curves into smaller regions.
As there are no disc patches, none of these trace curves are closed curves or are arcs with both endpoints on $\del M$.
Thus, taking an outermost trace curve, there is a region $\Delta'_0$ of $\Delta'$ whose boundary is one trace curve segment, one segment of $\Delta\cap F$, and possibly one segment from $\del M$.

Now, as shown in Lemma 4.1.35 of~\cite{Matveev}, if this region is in $G_1$, so $\del \Delta'_0 = \Delta'_0 \cap (G_2\cup \Delta \cup \del M)$, then as $M$ is irreducible we can isotope $D$ through $\Delta'_0$ to remove an arc of intersection with $G_2$.
If this region is in $G_2$, then compress $\Delta$ along $\Delta'_0$, removing an intersection. At least one of the resulting discs is essential and gives a non-trivial compression or $\del$-compression disc.
Either of these reduces the weight of $\Delta$.
Continuing this procedure, we eventually get a non-trivial compression or $\del$-compression disc for $G_1$ that is disjoint from $G_2$.
But then its boundary is contained in a patch of $G_1$, which is incompressible and $\del$-incompressible, so this is a contradiction.
\end{proof}

\section{Proof of Lemma~\ref{lemma:cutalongsurfacestriangulation}}
\label{appendix:cutalongsurfacetriangulation}
The proof of this result follows arguments of Lackenby and Haraway-Hoffman.
We will principally use the following result of Lackenby.

\begin{theorem}[\cite{LackenbyCertificationKnottedness}]
\label{thm:lackenbycutsurface}
   There is an algorithm that takes, as its input,
   \begin{enumerate}
       \item a triangulation $\mathcal{T}$ for a compact orientable manifold $M$, and
       \item a vector $(F)$ for an orientable normal surface $F$ with no two components normally isotopic
   \end{enumerate}
   and outputs, if $\mathcal{P}$ is the union of parallelity handles in $\mathcal{T}\backslash\backslash F$, 
   \begin{enumerate}
       \item a handle structure for $(\mathcal{T}\backslash\backslash F)-\mathcal{P}$;
       \item for each component of $\mathcal{P}$, the genus and number of boundary components of its base surface, whether the component is a product or a twisted product, and the annulus in $(\mathcal{T}\backslash\backslash F)-\mathcal{P}$ that is its boundary.
   \end{enumerate}
   Furthermore, this algorithm runs in time polynomial in $\norm{\mathcal{T}}$ and the weight of $F$.
\end{theorem}

The idea of the proof of Theorem~\ref{thm:lackenbycutsurface} is as follows.
Let $\mathcal{H}$ be the dual handle structure to $\mathcal{T}.$
By examining which elementary disc types are present in $F$, we can enumerate the handles of $(\mathcal{H}\backslash\backslash F)-\mathcal{P}$.
There are at most $O(\norm{\mathcal{T}})$ of them: six from each tetrahedron.
By considering the entries of the vector for $F$, we can identify when these handles are glued together.
For $\mathcal{P}$, Lackenby uses Proposition~\ref{prop:agolhassthurstonnormalsurface} to determine the topological type of the parallelity regions.

We now turn to work of Haraway and Hoffman.
As the handles are subtetrahedral split handles, the 0-handles have a uniformly bounded number of combinatorial types, so we can triangulate each of them so that their boundary graph is simplicial and so that each component of $\mathcal{P}\cap (\mathcal{H}\backslash\backslash F-\mathcal{P})$, which is naturally a square, is triangulated using two triangles.
In total this will use a linear number of tetrahedra in $\norm{\mathcal{T}}$.
For $\mathcal{P}$, by Theorem~\ref{thm:lackenbycutsurface} we know the topology of each component and the annuli in $(\mathcal{H}\backslash\backslash F)-\mathcal{P}$ to which it is glued.
It suffices to show that we can triangulate $\mathcal{P}$ in polynomial time.
This is the matter that Haraway and Hoffman concern themselves with in the proof of Proposition~\ref{prop:cutalongsurfacetriangulation}.

\begin{prop}[Proposition 13~\cite{HarawayHoffman}]
\label{prop:cutalongsurfacetriangulation}
There is an algorithm that takes as its input a triangulation $\mathcal{T}$ of a compact orientable 3-manifold and a connected normal surface $F$ in $\mathcal{T}$ (given as a vector in $\ZZ^{7\norm{\mathcal{T}}}$), and outputs a triangulation of an exterior of $F$ whose size is bounded by a polynomial in $\norm{\mathcal{T}}$, $\log w(F)$, and the Euler characteristic of $F$, $\chi(F)$.
This algorithm runs in polynomial time in those same three parameters.
\end{prop}

In the proof (which is itself confined to an appendix), they show that the Euler characteristic of each parallelity region is bounded by a linear function of $\norm{\mathcal{T}}$ and $|\chi_-(F)|$, where $\chi_-(F)$ is the sum of the Euler characteristics of the components of $F$ whose Euler characteristic is negative.
They do this by, for each parallelity region, observing that its Euler characteristic comes from three parts: genus (which is bounded by that of $F$ and hence by $|\chi_-(F)|$), boundary components of $F$ (also controlled by $|\chi_-(F)|$), and boundary components in the interior of $F$.
The number of boundary components in the interior of $F$ is bounded by the number of boundary components of $(\mathcal{H}\backslash\backslash F)-\mathcal{P}$, which is at most a linear function of $\norm{\mathcal{T}}$.
Let $\chi$ be the least Euler characteristic of the components of $F$; as by Lemma~\ref{lemma:kneserhaken}, there are at most $26\cdot 13!\norm{\mathcal{T}}$ components of $F$, $|\chi_-(F)|$ is at most $26\cdot 13!|\chi|\norm{\mathcal{T}}$.
Thus the Euler characteristic of the base surface of each parallelity region is bounded by a polynomial function in $|\chi|$ and $\norm{\mathcal{T}}$.
We can triangulate each base surface with a number of triangles proportional to its Euler characteristic, then barycentrically subdivide the product of this triangulation with an interval to obtain a triangulation of the parallelity region in question.
As the boundary squares of each parallelity region are thus endowed with a standard triangulation, which is the first barycentric subdivision of a square, and the triangulation graph of a polygon is connected, we can use a uniform number of additional tetrahedra -- corresponding to a uniform number of Pachner moves -- to walk from this triangulation to either of the two triangulations of the square using two triangles.
Some choice of those is compatible with the previously-chosen triangulation of the attaching annuli of the parallelity region.

\printbibliography

\end{document}